\documentclass[10pt]{amsart}

\usepackage[dvips]{color}
\usepackage{color}
\usepackage{amsmath}
\usepackage{amsthm,amssymb,amsfonts}
\usepackage{epsfig}

\theoremstyle{plain}
\newtheorem{thm}{Theorem}[section]
\newtheorem{lem}[thm]{Lemma}
\newtheorem{cor}[thm]{Corollary}
\newtheorem{prop}[thm]{Proposition}
\newtheorem{fact}[thm]{Fact}
\newtheorem*{fact*}{Fact}

\newtheorem*{thm*}{Theorem}
\newtheorem*{cor*}{Corollary}

\theoremstyle{remark}
\newtheorem{rmk}[thm]{Remark}
\newtheorem{example}[thm]{Example}
\newtheorem*{case}{Case}

\theoremstyle{definition}
\newtheorem{defn}[thm]{Definition}
\newtheorem{ques}[thm]{Question}

\newtheorem*{ques*}{Question}

\numberwithin{equation}{section}

\DeclareMathOperator{\ex}{ex}
\DeclareMathOperator{\bary}{bar}

\DeclareMathOperator{\ind}{ind}
\DeclareMathOperator{\supp}{supp}
\DeclareMathOperator{\diam}{diam}
\DeclareMathOperator{\dist}{dist}

\def\Z{\mathbb Z}

\def\R{\mathbb R}

\def\N{\mathbb N}

\def\al{\alpha}
\def\o{\omega}
\def\M{\mathcal M}

\def\cA{\mathcal A}

\def\P{\mathcal P}

\def\H{\mathcal H}

\def\0{\mathbf{0}}
\def\F{\mathcal F}
\def\G{\mathcal G}

\def\SM{\mathcal{SM}}

\def\erg{\text{erg}}
\def\hext{h_{\text{ext}}}
\def\hres{h_{\text{res}}}
\def\hsex{h_{\text{sex}}}
\def\charfun{\mathbf{1}}

\title{Orders of Accumulation of Entropy}

\author{David Burguet and Kevin McGoff}
\begin{document}


\maketitle

\begin{abstract}
For a continuous map $T$ of a compact metrizable space $X$ with finite topological entropy, the order of accumulation of entropy of $T$ is a countable ordinal that arises in the context of entropy structure and symbolic extensions. We show that every countable ordinal is realized as the order of accumulation of some dynamical system. Our proof relies on functional analysis of metrizable Choquet simplices and a realization theorem of Downarowicz and Serafin. Further, if $M$ is a metrizable Choquet simplex, we bound the ordinals that appear as the order of accumulation of entropy of a dynamical system whose simplex of invariant measures is affinely homeomorphic to $M$. These bounds are given in terms of the Cantor-Bendixson rank of $\overline{\ex(M)}$, the closure of the extreme points of $M$, and the relative Cantor-Bendixson rank of $\overline{\ex(M)}$ with respect to $\ex(M)$. We also address the optimality of these bounds.
\end{abstract}

\tableofcontents

\section{Introduction}

In this paper, a topological dynamical system is a pair $(X,T)$, where $X$ is a compact metrizable space and $T$ is a continuous mapping of $X$ to itself. For such a system $(X,T)$, the topological entropy $\mathbf{h}_{top}(T)$ provides a well-studied measure of the topological dynamical complexity of the system. We only consider systems with $\mathbf{h}_{top}(T) < \infty$. Let $M(X,T)$ be the space of Borel probability measures on $X$ which are invariant under $T$. The entropy function $h:M(X,T) \rightarrow [0,\infty)$, where $h(\mu)$ is the metric entropy of the measure $\mu$, quantifies the amount of complexity in the system that lies on generic points for $\mu$. In this sense, the entropy function $h$ describes both \textit{where} and \textit{how much} complexity lies in the system. The theory of entropy structures developed by Downarowicz \cite{D} produces a master entropy invariant in the form of a distinguished class of sequences of functions on $M(X,T)$ whose limit is $h$. The entropy structure of a dynamical system completely determines almost all previously known entropy invariants such as the topological entropy, the entropy function on invariant measures, the tail entropy (or topological conditional entropy \cite{M}), the symbolic extension entropy, and the symbolic extension entropy function. Entropy structure also produces new entropy invariants, such as the order of accumulation of entropy. Furthermore, the theory of entropy structures and symbolic extensions provides a rigorous description of \textit{how entropy emerges on refining scales}. Entropy structures and the closely related theory of symbolic extensions \cite{BD} have attracted interest in the dynamical systems literature \cite{A, Bur2, BurNew, DF, D, DM, DN}, especially with the intention of using entropy structure to obtain information about symbolic extensions for various classes of smooth systems. The purpose of the current work is to investigate a new entropy invariant arising from the theory of entropy structures: the order of accumulation of entropy, which is denoted $\al_0(X,T)$.

Given a dynamical system $(X,T)$, one may associate a particular sequence $\H(T)=(h_k)$ to $(X,T)$ with the following properties \cite{D}:
\begin{enumerate}
 \item $(h_k)$ is a non-decreasing sequence of harmonic, upper semi-continuous functions from $M(X,T)$ to $[0,\infty)$;
 \item $\lim_k h_k = h$;
 \item $h_{k+1}-h_k$ is upper semi-continuous for every $k$.
\end{enumerate}
This sequence, or any sequence uniformly equivalent to it (Definition \ref{uniformEquivDef}), is called an entropy structure for the system $(X,T)$ \cite{D}. This distinguished uniform equivalence class of sequences is an invariant of topological conjugacy of the system \cite{D}. Consequently, we sometimes refer to the entire uniform equivalence class of $\H$ as \textit{the} entropy structure of the system $(X,T)$.

Associated to a non-decreasing sequence $\H = (h_k)$ of functions $h_k : M \to [0,\infty]$, where $M$ is a compact metrizable space, there is a transfinite sequence of functions $u_{\al} : M \rightarrow [0,\infty]$, indexed by the ordinals and defined by transfinite induction as follows. Let $\widetilde{f}$ denote the upper semi-continuous envelope of the function $f$ (Definition \ref{USCdef}; by convention $\widetilde{f} \equiv \infty$ if $f$ is unbounded). Let $\tau_k = h - h_k$. Then
\begin{itemize}
 \item let $u_0 \equiv 0$;
 \item if $u_{\al}$ has been defined, let $ u_{\al+1} = \lim_{k} \widetilde{u_{\al} + \tau_k}$;
 \item if $u_{\beta}$ has been defined for all $\beta < \al$ for a limit ordinal  $\al$, let $u_{\al} = \widetilde{\sup_{\beta < \al} u_{\beta}}$.
\end{itemize}
The sequence $(u_{\al})$ is non-decreasing in $\al$ and does not depend on the particular representative of the uniform equivalence class of $\H$. Since $M$ is compact and metrizable, an easy argument (given in \cite{BD}) implies that there exists a countable ordinal $\al$ such that $u_{\beta} \equiv u_{\al}$ for all $\beta \geq \al$. The least ordinal $\al$ with this property is denoted $\al_0(\H)$ and is called the order of accumulation of $\H$. In the case when $M = M(X,T)$ and $\H$ is an entropy structure for $(X,T)$, the order of accumulation of entropy of $(X,T)$ is defined as $\al_0(\H)$. Because the entropy structure of $(X,T)$ is invariant under topological conjugacy, the sequence $(u_{\al})$ associated to $(X,T)$ and the order of accumulation $\al_0(X,T)$ are invariants of topological conjugacy.

To explain the meaning of $\al_0(X,T)$ and $u_{\al_0(X,T)}$, we discuss symbolic extensions and their relationship to entropy structures. A symbolic extension of $(X,T)$ is a (two-sided) subshift $(Y,S)$ on a finite number of symbols, along with a continuous surjection $\pi : Y \to X$ (the \textit{factor map} of the extension) such that $\pi \circ S = T \circ \pi$. Symbolic extensions have been important tools in the study of some dynamical systems, in particular uniformly hyperbolic systems. A symbolic extension serves as a ``lossless finite encoding'' of the system $(X,T)$ \cite{D}. If $\pi$ is the factor map of a symbolic extension $(Y,S)$, we define the extension entropy function $\hext^{\pi} : M(X,T) \rightarrow [0,\infty)$ for $\mu$ in $M(X,T)$ by
\begin{equation*}
\hext^{\pi}(\mu) = \max \{ h(\nu) : \pi^* \mu = \nu \}.
\end{equation*}
The number $\hext^{\pi}(\mu)$ represents the amount of complexity above the measure $\mu$ in the symbolic extension. The symbolic extension entropy function of a dynamical system $(X,T)$, $\hsex : M(X,T) \rightarrow [0, \medspace \infty]$, is defined for $\mu$ in $M(X,T)$ as
\begin{equation*}
\hsex(\mu) = \inf \{ \hext^{\pi}(\mu) : \pi \text{ is the factor map of a symbolic extension of } (X,T) \},
\end{equation*}
where the infimum is understood to be $\infty$ if $(X,T)$ admits no symbolic extensions. The symbolic extension entropy function measures the amount of entropy that must be present above each measure in any symbolic extension of the system. Finally, we define the residual entropy function $\hres : M(X,T) \to [0,\infty]$ as $\hres = \hsex - h$. The residual entropy function then measures the amount of entropy that must be added above each measure in any symbolic extension of the system. The functions $\hres$ and $\hsex$ give much finer information about the complexity of the system than the entropy function $h$.  These quantities are related to the entropy structure of the system by the following remarkable result of Boyle and Downarowicz.
\begin{thm}[\cite{BD}]
Let $X$ be a compact metrizable space and $T: X \to X$ a continuous map. Let $\H$ be an entropy structure for $(X,T)$. Then
\begin{equation*}
\hsex = h + u_{\al_0(X,T)}^{\H}.
\end{equation*}
\end{thm}
The conclusion of the theorem may also be stated as $u_{\al_0(X,T)} = \hres$. In this sense, the order of accumulation $\al_0(X,T)$ and the function $u_{\al_0(X,T)}$ each measures a residual complexity in the system that is not detected by the entropy function $h$.  The order of accumulation of entropy measures, roughly speaking, \textit{over how many distinct layers residual entropy emerges} in the system \cite{BD}. It is then natural to ask the following question.
\begin{ques} \label{realizationQuestion}
Which countable ordinals can be realized as the order of accumulation of entropy of a dynamical system?
\end{ques}
It is shown in \cite{BD} that all finite ordinals can be realized as the order of accumulation of dynamical system. There are constructions in \cite{Bur2, DN} (built for other purposes) that show that some infinite ordinals are realized in this way, but these constructions do not allow one to determine exactly which ordinals appear. Moreover, it is stated without proof in \cite{D} that all countable ordinals are realized.

We prove that all countable ordinals can be realized as the order of accumulation of entropy for a dynamical system (Corollary \ref{DynRealizationCor}), answering Question \ref{realizationQuestion}. On account of the realization theorem of Downarowicz and Serafin (restated as Theorem \ref{realization} in this work), this result reduces to establishing the following result, which is purely functional analytic.
\begin{thm}
For every countable ordinal $\al$, there exists a metrizable Choquet simplex $M$ and a sequence of functions $\H = (h_k)$ on $M$ such that
\begin{itemize}
 \item $(h_k)$ is a non-decreasing sequence of harmonic, upper semi-continuous functions from $M$ to $[0,\infty)$;
 \item $\lim_k h_k$ exists and is bounded;
 \item $h_{k+1}-h_k$ is upper semi-continuous for every $k$;
 \item $\al_0(\H) = \al$.
\end{itemize}
\end{thm}

Building on the approach of Downarowicz and Serafin to reduce questions in the theory of entropy structure to the study of functional analysis, we also consider what constraints, if any, the simplex of invariant measures may place on orders of accumulation of entropy.
\begin{ques} \label{dynQuestion}
Given a metrizable Choquet simplex $M$, which ordinals can be realized as the order of accumulation of a dynamical system $(X,T)$ such that $M(X,T)$ is affinely homeomorphic to $M$?
\end{ques}
For a metrizable Choquet simplex $M$, we let $S(M)$ denote the set of all ordinals that can be realized as the order of accumulation of a sequence $\H$ on $M$ satisfying properties (1)-(3). The realization theorem of Downarowicz and Serafin (Theorem \ref{realization}) reduces Question \ref{dynQuestion} to the following question in functional analysis.
\begin{ques} \label{introQues}
Given a metrizable Choquet simplex $M$, which ordinals are in $S(M)$?
\end{ques}
Theorem \ref{bauerThm} answers Question \ref{introQues} (and therefore Question \ref{dynQuestion}) completely in the event that $M$ is a Bauer simplex by giving a precise description of $S(M)$ in terms of the Cantor-Bendixson rank of the extreme points of $M$. Theorem \ref{topBounds} addresses the general case, giving constraints on $S(M)$ in terms of Cantor-Bendixson rank of the closure $\overline{E}$ of the space $E = \ex(M)$ of extreme points of $M$ and the relative Cantor-Bendixson rank of $\overline{E}$ with respect to $E$. Theorems \ref{threeAlphas} and \ref{optimal} address the optimality of these constraints, and Section \ref{OpenQuestions} summarizes our progress on this question and poses some remaining questions.

In the language of dynamical systems, if $M$ is a metrizable Choquet simplex, we have found constraints on the orders of accumulation of entropy that appear within the class of all dynamical systems $(X,T)$ such that $M(X,T)$ is affinely homeomorphic to $M$. These constraints are in terms of the Cantor-Bendixson ranks of the closure $\overline{E}$ of the space $E$ of ergodic measures and the relative Cantor-Bendixson rank of $\overline{E}$ with respect to $E$.

\section{Preliminaries} \label{preliminaries}

\subsection{Ordinals}
We assume a basic familiarity with the ordinal numbers, ordinal arithmetic, and transfinite induction. The relevant sections in \cite{Pin} provide a good introduction. Here we briefly recall some notions that are used in this work.

We view the ordinal $\al$ as the set $\{ \beta : \beta < \al \}$. The symbols $\o$ and $\o_1$ will always be used to denote the first infinite ordinal and the first uncountable ordinal, respectively.

\begin{defn}
An ordinal $\al$ is \textbf{irreducible} if whenever $\al = \al_1 +\al_2$ with $\al_1 \geq \al_2$, it follows that $\al_2=0$.
\end{defn}

Recall the well-known Cantor Normal Form of an ordinal.
\begin{thm} \label{CNF}
For every ordinal $\al > 0$, there exists natural numbers $n_1,\dots, n_k$ and ordinals $\beta_1 > \dots > \beta_k$ such that $\al = \o^{\beta_1}n_1 + \dots + \o^{\beta_k}n_k$. Furthermore, the numbers $n_1, \dots, n_k$ and the ordinals $\beta_1,\dots,\beta_k$ are unique.
\end{thm}

The following corollary is an easy consequence of the Cantor Normal Form.
\begin{cor} \label{irreducible}
An ordinal $\al > 0$ is irreducible if and only if there exists an ordinal $\beta$ such that $\al = \o^{\beta}$.
\end{cor}

In light of this corollary, one can view the Cantor Normal Form of $\al$ as a decomposition of $\al$ into a finite sum of irreducible ordinals.

The following corollary is then a simple consequence of Corollary \ref{irreducible} and the fact that any non-zero ordinal $\beta$ is either a successor ordinal or a limit ordinal.
\begin{cor}\label{irrCor}
If $\al > 0$ is countable and irreducible, then either (i) there exists an irreducible ordinal $\widetilde{\al} < \al$ such that $\sup_{n\in \N} \widetilde{\al} n = \al$, or (ii) there exists a strictly increasing sequence of irreducible ordinals $(\al_k)_{k \in \N}$ such that $\sup_{k \in \N} \al_k = \al$.
\end{cor}

Any ordinal $\al$ can be viewed as a topological space with the order topology (sets of the form $\{ \gamma \in \al : \gamma < \beta \}$ or $\{ \gamma \in \al : \beta < \gamma \}$ form a subbase for the topology). With this topology, $\al$ is a completely normal, Hausdorff space, and if $\al$ is countable, then it is a Polish space (see below for definition). The space $\al$ is compact if and only if $\al$ is a successor ordinal. The accumulation points in $\al$ are exactly the limit ordinals in $\al$.

For ease of notation, if $\al$ is a successor ordinal, let $\al-1$ denote the unique ordinal $\beta$ such that $\al = \beta+1$. Also, for countable ordinals $\al \leq \beta$, we will write $[\al, \beta]$ to denote the ordinal interval $\{ \gamma : \al \leq \gamma \leq \beta \}$. If $\beta = \o_1$, we make the convention that $[\al, \beta] = \{ \gamma : \al \leq \gamma < \beta \}$. We also make use of the notation $]\al, \beta[ = \{ \gamma : \al < \gamma < \beta \}$, as well as the other possible ``half-open'' and ``half-closed'' notations.


\subsection{Polish Spaces}

A general reference that covers Polish spaces is \cite{S}. We recall that a topological space $E$ is a Polish space if it is separable and completely metrizable. In particular, any compact metrizable space is Polish. Moreover, any closed subset of a Polish space is itself a Polish space. Some of the definitions and statements below hold for more general topological spaces, but we require them only in the case of Polish spaces.

For any Polish space $E$, let $E'$ denote the set of accumulation points of $E$,
\begin{equation*}
E' = \{ x \in E : \exists (x_n) \subset E \setminus \{ x \}, \medspace x_n \rightarrow x \}.
\end{equation*}
Note that $E'$ is closed in $E$.

A subset $A$ of a Polish space $E$ is a perfect set if $A$ is a compact subset of $E$ and $A$ contains no isolated points (in the subspace topology). The following result is a special case of the Cantor-Bendixson Theorem.
\begin{thm}\label{CBThm}
Let $E$ be a Polish space. Then $E = C \cup A$, where $C$ is countable, $A$ is closed and has no isolated points, and $C \cap A = \emptyset$.
\end{thm}

We will also use the following fact (see \cite{S}). Let $\mathcal{C}$ denote the Cantor space.
\begin{thm} \label{embedCperfect}
Let $A$ be a non-empty Polish space with no isolated points. Then there is an embedding of $\mathcal{C}$ into $A$.
\end{thm}

The following statement is an immediate corollary of the previous two theorems.
\begin{cor} \label{embedCunctble}
Let $E$ be any uncountable Polish space. Then there is an embedding of $\mathcal{C}$ into $E$.
\end{cor}

The following corollary is an easy consequence of Corollary \ref{embedCunctble}.
\begin{cor}\label{uncountPolish}
Let $E$ be an uncountable Polish space. Then for every countable ordinal $\al$ and every natural number $n$, there exists an embedding $g: \o^{\al}n+1 \rightarrow E$.
\end{cor}

\subsection{Cantor-Bendixson Rank}

Given a Polish space $E$, we now use transfinite induction to define a transfinite sequence of topological spaces, $\{ \Gamma^{\al}(E) \}$. Let $\Gamma^0(E) = E$. If $\Gamma^{\al}(E)$ has been defined, then let $\Gamma^{\al+1}(E) = (\Gamma^{\al}(E)) ' \subset \Gamma^{\al}(E)$. If $\al$ is a limit ordinal and $\Gamma^{\beta}(E)$ is defined for all $\beta < \al$, then let $\Gamma^{\al}(E) = \cap_{\beta < \al} \Gamma^{\beta}(E)$. Each set $\Gamma^{\al}(E)$ is closed in $E$ and therefore Polish.

Note that $\Gamma^{\al}(E) = \Gamma^{\al+1}(E)$ implies that $\Gamma^{\al}(E)$ has no isolated points (in the subspace topology) and then that $\Gamma^{\beta}(E) = \Gamma^{\al}(E)$ for all $\beta > \al$. For any Polish space $E$, Theorem \ref{CBThm} implies that there exists a countable ordinal $\al$ such that $\Gamma^{\al}(E) = \Gamma^{\al+1}(E)$.

\begin{defn}
With the notation above, the \textbf{Cantor-Bendixson rank} of the space $E$, denoted $|E|_{CB}$, is defined to be the least ordinal $\al$ such that $\Gamma^{\al}(E) = \Gamma^{\al+1}(E)$.
\end{defn}
When $E$ is compact, $\Gamma^{|E|_{CB}}(E)$ is a perfect set (which may be the empty set). Now we mention a pointwise version of Cantor-Bendixson rank.
\begin{defn} \label{ptwiseCBDef}
Let $E$ be a Polish space, and let $x$ be in $E$. We define the \textbf{topological rank} of $x$, $r(x)$, to be
\begin{equation*}
r(x) = \left\{
\begin{array}{ll}
\sup \{ \al : x \in \Gamma^{\al}(E) \} & \text{if } x \notin \Gamma^{|E|_{CB}}(E) \\
\o_1 & \text{if } x \in \Gamma^{|E|_{CB}}(E).
\end{array} \right.
\end{equation*}
\end{defn}

The following proposition follows directly from the definitions and compactness.
\begin{prop}\label{accumFacts} Let $E$ be a countable, compact Polish space. Then
\begin{enumerate}
 \item $|E|_{CB}$ is a successor ordinal.
 \item If $|E|_{CB} = \al+1$, then $\Gamma^{\al}(E)$ is a non-empty, finite set, and $\Gamma^{\al+1}(E) = \emptyset$.
 \item $|E|_{CB} = \big( \sup_{x \in E} r(x) \big)+1 = \big(\max_{x \in E} r(x)\big) + 1$.
\end{enumerate}
\end{prop}

Now we state a well-known classification of countable, compact Polish spaces, due to Mazurkiewicz and Sierpi\'{n}ski  \cite[p. 21]{MS}. We denote the cardinality of a set $E$ by $|E|$.
\begin{thm} \label{classification}
Let $E$ and $F$ be countable, compact Polish spaces, and assume that $|E|_{CB} = \al+1$. Then $E$ and $F$ are homeomorphic if and only if $|E|_{CB}=|F|_{CB}$ and $|\Gamma^{\al}(E)| = |\Gamma^{\al}(F)|$.
\end{thm}

\begin{rmk} \label{disjUnion}
Let $\al$ be a countable ordinal. Then $\Gamma^{\al}(\o^{\al}+1) = \{\o^{\al}\}$ and $|\o^{\al}+1|_{CB} = \al+1$. It follows from Theorem \ref{classification} that if $\gamma_k$ is any increasing sequence of ordinals such that $\sup_k \gamma_k = \o^{\al}$, then $\o^{\al}+1$ is homeomorphic to the one-point compactification of the disjoint union of the spaces $\gamma_k$, with the point at infinity corresponding to $\o^{\al}$.
\end{rmk}

Note that for any countable ordinal $\al$, the space $\o^{\al}n+1$ has Cantor-Bendixson rank $\al+1$ and exactly $n$ points of topological rank $\al$ given by $\o^{\al}k$ for $k = 1, \dots, n$. Then by the above classification, the space $\o^{\al}n+1$ provides a representative of the homeomorphism class of countable, compact Polish spaces with Cantor-Bendixson rank $\al+1$ and $n$ points of topological rank $\al$. 

\subsection{Upper-semicontinuity}

Now we consider functions $f : E \rightarrow \R$, where $E$ is a metrizable space. For such a function $f$, we let $||f|| = \sup_{x \in E} |f(x)|$, where the supremum is taken to be $+\infty$ if $f$ is unbounded.

\begin{defn} \label{USCdef}
Let $E$ be a compact metrizable space, and let $f : E \rightarrow \R$. Then $f$ is \textbf{upper semi-continuous} (u.s.c.) if one of the following equivalent conditions holds:
\begin{enumerate}
 \item $f = \inf_{\al} g_{\al}$ for some family $\{g_{\al}\}$ of continuous functions;
 \item $f = \lim_{n} g_n$ for some nonincreasing sequence $(g_n)_{n\in \N}$ of continuous functions;
 \item For each $r \in \R$, the set $\{x : f(x) \geq r \}$ is closed;
 \item $\limsup_{y\rightarrow x} f(y) \leq f(x)$, for all $x \in E$.
\end{enumerate}
For any $f : E \rightarrow \R$, the \textbf{upper semi-continuous envelope} of $f$, written $\widetilde{f}$, is defined, for all $x$ in $E$, by
\begin{equation*}
 \widetilde{f}(x) = \left\{ \begin{array}{ll} \inf \{ g(x) : g \text{ is continuous, and } g \geq f \}, & \text{ if } f \text{ is bounded} \\ +\infty, & \text{ if } f \text{ is unbounded.}
                            \end{array}
\right.
\end{equation*}
\end{defn}

Note that when $f$ is bounded, $\widetilde{f}$ is the smallest u.s.c. function greater than or equal to $f$ and satisfies
\begin{equation*}
 \widetilde{f}(x) = \max \left( f(x), \limsup_{y \rightarrow x} f(y) \right).
\end{equation*}
It is immediately seen that for any $f, g : E \rightarrow \R$, $\widetilde{f+g} \leq \widetilde{f}+\widetilde{g}$, with equality holding if $f$ or $g$ is continuous.

\begin{defn}
Let $\pi : E \rightarrow F$ be a continuous map. If $f: F \rightarrow \R$ is any function, we define the \textbf{lift} of $f$, denoted $^{\pi} f$, to be the function given by $f \circ \pi$.

If $\pi : E \rightarrow F$ is a surjection and $f: E \rightarrow \R$ is bounded, then the \textbf{projection} of $f$, denoted $f^{[F]}$, is the function defined on $F$ by
\begin{equation*}
f^{[F]}(x) = \sup_{y \in \pi^{-1}(x)} f(y).
\end{equation*}
\end{defn}

\begin{rmk} \label{pushDownPullUp} Let $\pi: E \rightarrow F$ be a continuous surjection.
\begin{enumerate}
 \item If $f : F \rightarrow \R$, then $(^{\pi} f)^{[F]} = f$.
 \item If $f: E \rightarrow \R$, then $^{\pi}(f^{[F]}) \geq f$, and the inequality is strict in general.
 \item If $f : E \rightarrow \R$ is u.s.c., then $f^{[F]}$ is also u.s.c. and the supremum is attained.
 \item If $f: F \rightarrow \R$ is u.s.c., then $^{\pi} f$ is also u.s.c.
\end{enumerate}
\end{rmk}

\subsection{Candidate Sequences} \label{candSeqs}

\begin{defn}
A \textbf{candidate sequence} on a compact, metrizable space $E$ is a non-decreasing sequence $\H = (h_k)$ of non-negative, real-valued functions on $E$ that converges pointwise to a function $h$. We often write $\lim \H = h$. We always assume by convention that $h_0 \equiv 0$.

A candidate sequence $\H$ \textbf{has u.s.c. differences} if $h_{k+1} - h_{k}$ is u.s.c. for all $k$. Note that in this case each $h_k$ is u.s.c., since $h_0 \equiv 0$. If $\H$ has u.s.c. differences, we may also refer to $\H$ as a u.s.c.d. candidate sequence, or we may write that $\H$ is u.s.c.d.
\end{defn}


Given a candidate sequence $\H$, it is natural to seek a precise description of the manner in which $h_k$ converges to $h$. For example, is this convergence uniform or not? The notion of \textit{uniform equivalence}, as defined by Downarowicz in \cite{D}, captures exactly the manner in which $h_k$ converges to $h$.

\begin{defn} \label{uniformEquivDef}
Let $\H$ and $\F$ be two candidate sequences on a compact, metrizable space $E$. We say that $\H$ \textbf{uniformly dominates} $\F$, written $\H \geq \F$, if for all $\epsilon > 0$, and for each $k$, there exists $\ell$, such that $f_k \leq h_{\ell}+\epsilon$.

The candidate sequences $\H$ and $\F$ are \textbf{uniformly equivalent}, written $\H \cong \F$, if $\H \geq \F$ and $\F \geq \H$.
\end{defn}
Note that uniform equivalence is in fact an equivalence relation.

\section{Basic Constructions} \label{basicConstructions}

\subsection{Order Of Accumulation}
\begin{defn}
Let $\H$ be a candidate sequence on $E$. The \textbf{transfinite sequence} associated to $\H$, which we write as $(u_{\al}^{\H})$ or $(u_{\al})$, is defined by transfinite induction as follows. Let $\tau_k = h - h_k$. Then
\begin{itemize}
 \item let $u_0 \equiv 0$;
 \item if $u_{\al}$ has been defined, let $ u_{\al+1} = \lim_{k} \widetilde{u_{\al} + \tau_k}$;
 \item if $u_{\beta}$ has been defined for all $\beta < \al$ for a limit ordinal  $\al$, let $u_{\al} = \widetilde{\sup_{\beta < \al} u_{\beta}}$.
\end{itemize}
\end{defn}

Note that for each $\al$, either $u_{\al} \equiv +\infty$ or $u_{\al}$ is u.s.c. (since a non-increasing limit of u.s.c. functions is u.s.c.). Furthermore, the sequence $(u_{\al})$ is non-decreasing in $\al$. It is also sub-additive in the following sense.

\begin{prop} \label{UisZero}
Let $\H$ be a candidate sequence on $E$. Then for any two ordinals $\al$ and $\beta$,
\begin{equation*}
 u_{\al+\beta} \leq u_{\al} + u_{\beta}.
\end{equation*}
\end{prop}
\begin{proof}
Let $\al$ be any ordinal. We prove the statement by transfinite induction on $\beta$. For $\beta=0$, the statement is trivial. Now assume by induction that the statement is true for $\gamma < \beta$. If $\beta$ is a successor ordinal, then by the inductive hypothesis,
\begin{equation*}
u_{\al+\beta} = \lim_k \widetilde{(u_{\al + (\beta-1)} + \tau_k)} \leq u_{\al} + \lim_k \widetilde{(u_{\beta-1} + \tau_k)} = u_{\al}+u_{\beta}.
\end{equation*}
If $\beta$ is a limit ordinal, then by the inductive hypothesis,
\begin{equation*}
u_{\al+\beta} = \widetilde{\sup_{\gamma< \beta} u_{\al+\gamma}} \leq u_{\al}+\widetilde{\sup_{\gamma<\beta}  u_{\gamma}} \leq u_{\al}+u_{\beta}.
\end{equation*}
\end{proof}

If $\H$ is a candidate sequence on $E$, then by Theorem 3.3 in \cite{BD}, there exists a countable ordinal $\al$ such that the associated transfinite sequence satisfies $u_{\al} = u_{\al + 1}$, which then implies that $u_{\beta} = u_{\al}$ for all $\beta > \al$.
\begin{defn}
In this setting, the least ordinal $\al$ such that $u_{\al} = u_{\al + 1}$ is called the \textbf{order of accumulation} of the candidate sequence $\H$, which we write as either $\al_0(\H)$ or $\al_0^{\H}$.
\end{defn}

Both the transfinite sequence and the order of accumulation are independent of the choice of representative of uniform equivalence class \cite{D}.

While it is true that $u_{\al}=u_{\al+1}$ implies $u_{\al} = u_{\beta}$ for all $\beta > \al$, it is not true that for a fixed $x$, $u_{\al}(x) = u_{\al+1}(x)$ implies $u_{\beta}(x) = u_{\al}(x)$ for all $\beta > \al$. In fact, in many of the constructions in Section \ref{examples} there is a point $\0$ and an ordinal $\al$ such that $u_{\gamma}(\0)=0$ for all $\gamma < \al$ and $u_{\al}(\0)=a>0$. Nonetheless, we make the following definition.
\begin{defn}
Let $\H$ be a candidate sequence on $E$. Then for each $x$ in $E$, we define the \textbf{pointwise order of accumulation} of $\H$ at $x$, $\al_0^{\H}(x)$ or $\al_0(x)$, as
\begin{equation*}
\al_0^{\H}(x) = \inf \{ \al : u_{\beta}(x) = u_{\al}(x) \text{ for all } \beta > \al \}.
\end{equation*}
\end{defn}

\begin{rmk} \label{accumVarPrin}
Note that $\al_0^{\H}(x)$ is always a countable ordinal, and
\begin{equation*} \al_0(\H) = \sup_{x \in E} \al_0^{\H}(x).
\end{equation*}
\end{rmk}

The following proposition relates the pointwise topological rank (Definition \ref{ptwiseCBDef}) to the pointwise order of accumulation.

\begin{prop} \label{pointwiseBound}
Let $\H$ be a candidate sequence on $E$. Then for any $x$ in $E$,
\begin{equation*}
\al_0(x) \leq \left\{
\begin{array}{ll}
 r(x) & \text{ if } r(x) \text{ is finite} \\
 r(x) +1 & \text{ if } r(x) \text{ is infinite}.
\end{array} \right.
\end{equation*}
\end{prop}
\begin{proof}
The proof proceeds by transfinite induction on $r(x)$. If $r(x) = 0$ it is easily seen that $u_{\gamma}(x)=0$ for all $\gamma$ and $\al_0(x)=0$.

Suppose the statement is true for all $y$ with $r(y) < \al$, and fix $x$ with $r(x) = \al$. If $\al$ is finite, let $\epsilon = \al$, and if $\al$ is infinite, let $\epsilon = \al+1$. We show that for all $\beta > \al$, $u_{\beta}(x) = u_{\epsilon}(x)$, and here we use transfinite induction on $\beta > \al$. Note that there is an open neighborhood $U$ of $x$ such that for all $y$ in $U$, $r(y) < r(x)$. Thus any real-valued function $f$ on $E$ satisfies $\limsup_{y \rightarrow x} f(y) = \limsup_{y \rightarrow x, \medspace r(y) < r(x)} f(y)$.

Suppose $\beta > \al$ is a successor. Then
\begin{equation*}
\widetilde{(u_{\beta -1} + \tau_k)}(x) = \max \Bigl( \limsup_{\substack{y \rightarrow x \\ \medspace r(y) < \al}} \, (u_{\beta-1}+\tau_{k})(y), \medspace (u_{\beta-1}+\tau_k)(x) \Bigr).
\end{equation*}
Applying the induction hypotheses to all $y$ with $r(y)<\al$ and $u_{\beta-1}(x)$ gives that
\begin{equation*}
\widetilde{(u_{\beta -1} + \tau_k)}(x) = \max \Bigl( \limsup_{\substack{y \rightarrow x \\ \medspace r(y) < \al}} \, (u_{\epsilon-1}+\tau_{k})(y), \medspace (u_{\epsilon}+\tau_k)(x) \Bigr).
\end{equation*}
Letting $k$ tend to infinity, we obtain $u_{\beta}(x) = u_{\epsilon}(x)$.

Suppose $\beta$ is a limit ordinal. Then the inductive hypotheses imply
\begin{align*}
u_{\beta}(x) & = \max \Bigl( \limsup_{\substack{ y\rightarrow x \\ \medspace r(y) < \al}} \, \sup_{\gamma < \beta} u_{\gamma}(y) , \medspace \sup_{\gamma < \beta} u_{\gamma}(x) \Bigr) \\
 & = \max \Bigl( \limsup_{\substack{y\rightarrow x \\ \medspace r(y) < \al}} \,  u_{\epsilon-1}(y) , \medspace \sup_{\gamma < \beta} u_{\epsilon}(x) \Bigr) \\
 & = u_{\epsilon}(x).
\end{align*}

\end{proof}
It follows from the proof of Theorem \ref{bauerThm} that these pointwise bounds on $\al_0(x)$ are optimal. Also, combining Remark \ref{accumVarPrin}, Proposition \ref{pointwiseBound}, and Proposition \ref{accumFacts} (3), we obtain the following result.
\begin{cor} \label{CBrankBound}
Let $\H$ be a candidate sequence on a countable, compact Polish space $E$. Then
\begin{equation*}
\al_0(\H) \leq \left\{ \begin{array}{ll}
                 |E|_{CB}-1, & \text{ if } |E|_{CB} \text{ is finite} \\
                 |E|_{CB}, & \text{ if } |E|_{CB} \text{ is infinite}.
                \end{array} \right.
\end{equation*}
\end{cor}

\subsection{Construction of Candidate Sequences}

Now we discuss various ways of creating candidate sequences. We first begin with elementary constructions that will be studied later in the context of Choquet simplices.
\begin{defn}
Let $\H$ be a candidate sequence on $E$. If $F$ is a compact subset of $E$, then we define the \textbf{restriction candidate sequence}, $\H|_{F}$, on $F$.
\end{defn}

\begin{defn}
Let $\H$ be candidate sequence on $E$, and let $F$ be a compact metrizable space with $\pi: F \rightarrow E$ a continuous surjection. Then the \textbf{lifted candidate sequence} of $\H$ to $F$, denoted $^{\pi} \H$, is the candidate sequence on $F$ given by $(^{\pi}h_k) = (h_k \circ \pi)$.
\end{defn}

\begin{defn}
Let $\F = (f_k)$ be a candidate sequence on $F$, and let $g:F\rightarrow E$ be an embedding (continuous injection). The \textbf{embedded candidate sequence}, $g\F = (h_k)$, on $E$ is defined to be
\begin{equation*}
h_k(x) = \left\{
\begin{array}{ll}
f_k \circ g^{-1}(x) & \text{ if } x \in g(F) \\
0 & \text{ if  } x \in E \setminus g(F).
\end{array} \right.
\end{equation*}
\end{defn}

While all of the constructions in this section will be used, the following two constructions (disjoint union and product candidate sequences) form the basis of the proofs of Theorem \ref{mainTech} and Corollary \ref{realizationCor}.

\begin{defn}
Let $(\H_n)$ be a countable collection of candidate sequences, where $\H_n = (h_k^n)$ is defined on $E_n$. Then we define the \textbf{disjoint union candidate sequence}, $\coprod \H_n$, as follows. Let $E$ be the one-point compactification of the disjoint union of the spaces $E_n$, with the point at infinity denoted $\0$. For each $k$, let $f_k$ be the function on $E$ such that $f_k|_{E_n} = h_k^n$ and $f_k(\0) = 0$. Then the disjoint union candidate sequence, $\coprod \H_n$, is defined to be $(f_k)$.
\end{defn}

Recall that $||f||$ denotes the supremum norm of the real-valued function $f$.

\begin{lem} \label{disUnionLemma}
Let $(\H_n)$ be a sequence of candidate sequences on $E_n$, where $h^n = \lim \H_n$. Let $\H = \coprod \H_n$. If $||h^n|| \rightarrow 0$, then for all $\beta$,
\begin{enumerate}
\item $u_{\beta}^{\H}(\0) = \limsup_{n} ||u_{\beta}^{\H_n}||$, and
\item $||u_{\beta}^{\H}|| = \sup_{n} ||u_{\beta}^{\H_n}||$.
\end{enumerate}
\end{lem}
\begin{proof}
For each $n$, $E_n$ is a clopen subset of $E$. It follows that $u_{\gamma}^{\H}(x) = u_{\gamma}^{\H_n}(x)$ for all ordinals $\gamma$, and for all $x$ in $E_n$. Then (2) follows from the definitions and (1). Also, upper semi-continuity of $u_{\beta}^{\H}$ implies that $u_{\beta}^{\H}(\0) \geq  \limsup_n ||u_{\beta}^{\H_n}||$. It remains only to show the reverse inequality.

The hypotheses imply that
\begin{equation*}
u_1^{\H}(\0) \leq \widetilde{h}(\0) \leq \lim_n ||h^n || = 0.
\end{equation*}
Now we use transfinite induction on $\beta$. The case $\beta=0$ is trivial. Suppose $\beta$ is a successor. By sub-additivity of the transfinite sequence (Lemma \ref{UisZero}) $u_{\beta}^{\H}(\0) \leq u_{\beta-1}^{\H}(\0) + u_1^{\H}(\0) = u_{\beta-1}^{\H}(\0)$, which, along with induction, implies the desired inequality. Now suppose $\beta$ is a limit ordinal. Monotonicity of the transfinite sequence and induction again imply that
\begin{equation*}
u_{\beta}^{\H}(\0) = \max \Bigl( \limsup_{y \rightarrow \0} u_{\beta}^{\H}(y), \; \sup_{\gamma < \beta} u_{\gamma}^{\H}(\0) \Bigr) \leq \limsup_n ||u_{\beta}^{\H_n}||.
\end{equation*}
\end{proof}

By a marked space $(E, \0)$, we mean a compact, metrizable space $E$ together with a marked point $\0$ in $E$.
\begin{defn}
Let $\F = (f_k)$ and $\mathcal{G} = (g_k)$ be two candidate sequences defined on the marked spaces $(E_1, \0_1)$ and $(E_2, \0_2)$, respectively. Then we define the \textbf{product candidate sequence}, $\H = \F \times \mathcal{G}$, on the marked product space $(E_1 \times E_2, (\0_1,\0_2))$ as the sequence
\begin{equation*}
h_k(x,y) = \left\{
\begin{array}{ll}
f_k(x) & \text{ if } y = \0_2 \\
g_k(y) & \text{ if } y \neq \0_2
\end{array} \right.
\end{equation*}
\end{defn}
Note that this definition is not symmetric under transposition of $\F$ and $\mathcal{G}$. In other words, this product is not commutative, but one may check easily that it is associative.

Let $\H$ be a candidate sequence on the marked space $(E,\0)$. Define $(\H)^{\times p}$ to be the candidate sequence on the product space $(E^p,\0^p)$ given by iterated multiplication: $(\H)^{\times p} = \H^{\times (p-1)} \times \H$.
\renewcommand{\labelenumi}{(\roman{enumi})}
\begin{lem}[Powers Lemma] \label{powerLem}
Let $\H$ be a candidate sequence on the marked space $(E,\0)$. Suppose that for some limit ordinal $\al$ and real number $a >0$,
\begin{enumerate}
 \item $||u_{\gamma}|| \leq a$ for all $\gamma$, and $||u_{\gamma}|| < a$ for $\gamma < \al$;
 \item $u_{\gamma}(\0) = 0$, for all $\gamma < \al$, and $u_{\al}(\0)=a$;
 \item $\al_0(x) \leq \al$, for all $x$ in $E$.
\end{enumerate}
Then the transfinite sequence associated to  $(\H)^{\times p}$ satisfies
\renewcommand{\labelenumi}{(\arabic{enumi})}
\begin{enumerate}
  \item $||u_{\gamma}^{\H^{\times p}}|| \leq p a$ for all $\gamma$;
  \item $||u_{\al k}^{\H^{\times p}}|| \leq k a$ and $||u_{\gamma}^{\H^{\times p}}|| < k a$, for all $\gamma < \al k$ and $k \leq p$;
  \item $\al_0^{\H^{\times p}}(x) \leq \al p$, for all $x $ in $E^p$;
  \item $u_{\gamma}^{\H^{\times p}}(\0^p) = \ell a$, for all $\al \ell \leq \gamma < \al (\ell+1)$, and $\ell = 0, \dots , p$;
  \item $\al_0^{\H^{\times p}}(\0^p) = \al p$.
\end{enumerate}
\end{lem}
\begin{proof}
We argue by induction on $p$. For $p=1$, the claims (1)-(5) follow from (i)-(iii). 

Assume that (1)-(5) hold for $p$. We prove that (1)-(5) also hold with $p+1$ in place of $p$. Let $(u_{\al}^p)$ be the transfinite sequence for $\H^{\times p} = (h_k^p )$, and let $h^p = \lim \H^{\times p}$. Recall that $E^{p+1} = E^p \times E$. The definition of $\H^{\times (p+1)}$ is that
\begin{equation*}
h_k^{p+1}(x,y) = \left \{ \begin{array}{ll}
                            h_k^p(x), & \text{ if } y= \0 \\
                            h_k(y), & \text{ if } y \neq \0.
                           \end{array} \right.
\end{equation*}
For all $(x,y)$ in $E^{p+1}$, $(h^{p+1}-h^{p+1}_k)(x,y) \leq (h^p-h^p_k)(x) + (h-h_k)(y)$. It follows from transfinite induction that for all $\gamma$, $u_{\gamma}^{p+1}(x,y) \leq u_{\gamma}^p(x) + u_{\gamma}(y)$. Using the inductive hypotheses, we obtain that $||u_{\gamma}^{p+1}|| \leq ap+a= a(p+1)$ for all $\gamma$, proving (1).

It follows from subadditivity that $||u^{p+1}_{\al k+ \gamma}|| \leq k ||u^{p+1}_{\al}||+ ||u^{p+1}_{\gamma}||$, which means that in order to establish (2) we need only show that for all $\gamma < \al$, $||u^{p+1}_{\gamma}|| < a$. Furthermore, since $u_{\gamma}^{p+1}$ is u.s.c. and therefore attains its supremum, it suffices to show that for all $\gamma < \al$ and all $(x,y)$ in $E^{p+1}$, $u_{\gamma}^{p+1}(x,y) < a$. Let $\gamma < \al$ and let $(x,y)$ be in $E^{p+1}$. If $y \neq \0$, then there exists an open neighborhood $U$ of $(x,y)$ in $E^{p+1}$ such that for all $(s,t)$ in $U$, $t \neq \0$. Then $h^{p+1}_k(s,t) = h_k(t)$ for all $(s,t)$ in $U$. It follows that $u_{\gamma}^{p+1}(x,y) = u_{\gamma}(y) < a$. Now suppose $y = \0$. Let $\epsilon>0$. Since $u_{\gamma}(\0) = 0$ and $u_{\gamma}$ is u.s.c., there exists an open neighborhood $U$ of $\0$ in $E$ such that for all $s$ in $U$, $u_{\gamma}(s) \leq \epsilon$. Then for all $(t,s)$ in the open set $E^p \times U$, $u_{\gamma}^{p+1}(t,s) \leq u_{\gamma}^p(t) + u_{\gamma}(s) \leq u_{\gamma}^{p}(t) + \epsilon$. Since $\epsilon$ was arbitrary, we obtain that $u_{\gamma}^{p+1}(x,\0) \leq u_{\gamma}^{p}(x)$. Using the induction hypothesis for $\H^{\times p}$, we conclude that $u_{\gamma}^{p+1}(x,\0) < a$.

For any point $(x,y)$ in $E^{p+1}$ with $y \neq \0$, we have already shown that $u_{\gamma}^{p+1}(x,y) = u_{\gamma}(y)$ for all $\gamma$. For any point of the form $(x,\0)$, we have shown that $u_{\al}^{p+1}(x,\0) \leq a$. Furthermore, by upper-semicontinuity of $u_{\al}^{p+1}$, we have that
\begin{equation*}
u_{\al}^{p+1}(x,\0) \geq \limsup_{y \rightarrow \0} u_{\al}^{p+1}(x,y) = \limsup_{y \rightarrow \0} u_{\al}(y) = u_{\al}(\0)  = a.
\end{equation*}
Thus $u_{\al}^{p+1}(x,\0)= a$ for all points of the form $(x,\0)$. This fact, in combination with the fact that $u_{\gamma}^{p+1}(x,y) = u_{\gamma}(y) \leq a$ for $y \neq \0$ and all $\gamma$, immediately implies that $u_{\al+\gamma}^{p+1}(x,\0) = u_{\gamma}^p(x)+a$ for all $x$ in $E^p$. Then induction gives statements (3)-(5).
\end{proof}

\begin{defn}
For the rest of this work, we let $\H^p$ denote the \textbf{renormalized product} of $\H$ taken $p$ times: if $\H^{\times p} = (h_k^{\times p})$, then let $\H^p = (h_k^p) = (\frac{1}{p} h_k^{\times p})$.
\end{defn}

Now we discuss more general products than just powers of the same candidate sequence. We will only consider products of marked spaces. Let $x$ be a point in the product space $(E_N \times \dots \times E_1, \0)$, where $\0 = (\0_N, \dots, \0_1)$. Let $\pi_i$ be projection onto $E_i$. Then define the function
\begin{equation*}
\ind(x) = \left\{
\begin{array}{ll}
\min\{ i : \pi_i(x) \neq \0_i \} & \text{ if } x\neq \0 \\
N & \text{ if } x = \0.
\end{array} \right.
\end{equation*}
Also, let $\eta_i(x_N, \dots, x_1) = (x_N, \dots, x_i)$. Note that with these notations, if $(h_k) = \H_N \times \dots \times \H_1 $, then $h_k(x) = h_k^{\H_{\ind(x)}}(\pi_{\ind(x)}(x))$ for all $x$.


\begin{lem}[Product Lemma] \label{prodLemma}
Let $\al$ be any non-zero countable ordinal, and let $\al = \o^{\beta_1} m_1 + \dots + \o^{\beta_N} m_N$ be the Cantor Normal Form of $\al$. Let $a>0$ be a real number, and suppose $a_1 > \dots > a_N > 0$ such that
\begin{equation*}
\sum_{i=1}^N a_i = a,
\end{equation*}
and for each $j = 1, \dots, N-1$,
\begin{equation} \label{aiCondition}
\frac{a_j}{m_j} \geq \sum_{i=j+1}^N a_i.
\end{equation}
(Note that for any $a>0$, such $a_1, \dots, a_N$ exist.) Now suppose that for each $j$ in $\{1, \dots, N\}$, $\F_j$ is a candidate sequence on $(E_j,\0_j)$ such that
\renewcommand{\labelenumi}{(\roman{enumi})}
\begin{enumerate}
 \item $||u_{\gamma}^{\F_j}|| \leq a_j$  for all $\gamma$, and $||u_{\gamma}^{\F_j}|| < a_j$  for $\gamma < \o^{\beta_j}$;
 \item $u_{\gamma}^{\F_j}(\0_j) = 0$, for all $\gamma < \o^{\beta_j}$;
 \item $u_{\o^{\beta_j}}^{\F_j}(\0_j)=a_j$;
 \item $\al_0(x) \leq \o^{\beta_j}$, for all $x \neq \0_j$;
 \item $\al_0(\0_j) = \o^{\beta_j}$.
\end{enumerate}
Denote $\H_j = \F_j^{m_j}$ and $\al_j = \o^{\beta_j} m_j$. Then the product $\H_N \times \dots \times \H_1$ satisfies
\renewcommand{\labelenumi}{(\arabic{enumi})}
\begin{enumerate}
 \item $||u_{\gamma}|| \leq  a$ for all $\gamma$, and $||u_{\gamma}|| <  a$ for $\gamma < \al$;
 \item $\al_0(x) \leq \al$, for all $x \neq \0$;
 \item $\al_0(\0) = \al$, and $u_{\al_0}(\0) = a$. In particular, $\al_0(\H_N \times \dots \times \H_1) = \al$.
\end{enumerate}
\end{lem}
\begin{proof}
The proof proceeds by induction on $N$. The case $N=1$ follows from (i)-(v). Now we assume that $N>1$ and the statement holds for $N-1$, and we show that it holds for $N$.

Let $\H_N \times \dots \times \H_1 = (h_k)$ be as above, with $h = \lim_k h_k$, and let $\H_{N} \times \dots \times \H_2 = (h'_k)$ with $h'= \lim_k h'_k$. By the definition of the product candidate sequence, we observe that $(h-h_k)(x) \leq (h'-h'_k)(\eta_2(x)) + (h^1 - h^1_k)(\pi_1(x))$. It follows that $u_{\al}(x) \leq u_{\al}^{\H_{N} \times \dots \times \H_2}(\eta_2(x)) + u_{\al}^{\H_1}(\pi_1(x))$ for all $x$ in $E$ and $\al$.

Let $x$ be in $E$. Then there exists an open neighborhood $U$ in $E$ such that for all $y$ in $U$, $\ind(y) \leq \ind(x)$.

If $\ind(x) = 1$, the existence of the neighborhood $U$ implies that $u_{\gamma}^{\H}(x) = u_{\gamma}^{\H_1}(\pi_1(x))$ for all $\gamma$.

Now we prove that for $\gamma < \o^{\beta_1}$ and $x$ such that $\ind(x) > 1$, we have $u_{\gamma}^{\H}(x) \leq u_{\gamma}^{\H_N \times \dots \times \H_2}(\eta_2(x))$. Since $\F_1$ satisfies the hypotheses $(i)-(v)$, we may apply Lemma \ref{powerLem} and conclude that $\H_1$ satisfies conclusions (1)-(5) in Lemma \ref{powerLem}. Now let $\gamma < \o^{\beta_1}$ and let $x$ be in $E$ with $\ind(x) >1$. By conclusion (4) in Lemma \ref{powerLem} applied to $\H_1$, $u_{\gamma}^{\H_1}(\0_1) = 0$. Then for any $\epsilon >0$, using that $u_{\gamma}^{\H_1}$ is u.s.c., there exists an open neighborhood $V$ of $x$ such that for all $y$ in $V$, $u_{\gamma}^{\H}(y) \leq u_{\gamma}^{\H_N \times \dots \times \H_2}(\eta_2(y)) + \epsilon$. Since $\epsilon >0$ was arbitrary, we have the desired inequality.

By the induction hypothesis on $N-1$ applied to $\H_N \times \dots \times \H_2$, we have $\sup_{\gamma < \o^{\beta_1}} u_{\gamma}^{\H_N \times \dots \H_2}(\eta_2(x)) \leq \sum_{j=2}^N a_j$. By conclusion (2) in Lemma \ref{powerLem} applied to $\H_1$, $||u_{\o^{\beta_1}}^{\H_1}|| \leq \frac{a_1}{m_1}$. Hence, for $x$ in $E$,
\begin{equation*}
u_{\o^{\beta_1}}^{\H}(x) = \widetilde{ \bigl(\sup_{\gamma < \o^{\beta_1}} u_{\gamma}^{\H} \bigr)}(x) \leq \max \Bigl( \frac{a_1}{m_1}, \, \sum_{j=2}^{N} a_j \Bigr) \leq \frac{a_1}{m_1}.
\end{equation*}
Then by upper semi-continuity of $u_{\o^{\beta_1}}^{\H}$, we have that for any $x$ with $\ind(x) >1$,
\begin{equation*}
u_{\o^{\beta_1}}^{\H}(x) \geq  \limsup_{\substack{y \rightarrow x \\ \ind(y)=1 }} u_{\o^{\beta_1}}^{\H}(y) = \limsup_{\substack{y \rightarrow x \\ \ind(y)=1 }} u_{\o^{\beta_1}}^{\H_1}(\pi_1(y)) = u_{\o^{\beta_1}}^{\H_1}(\0_1) = \frac{a_1}{m_1}.
\end{equation*}
We conclude that for any $x$ with $\ind(x)>1$, $u_{\o^{\beta_1}}^{\H}(x) = \frac{a_1}{m_1}$. By sub-additivity (Proposition \ref{UisZero}), we have that $u_{\o^{\beta_1}m_1}^{\H}(x) \leq a_1$. By upper semi-continuity, for all $x$ with $\ind(x)>1$,
\begin{equation*}
u_{\o^{\beta_1}m_1}^{\H}(x) \geq \limsup_{\substack{y \rightarrow x \\ \ind(y)=1 }} u_{\o^{\beta}m_1}^{\H_1}(\pi_1(y)) = u_{\o^{\beta}m_1}^{\H_1}(\0_1) = a_1.
\end{equation*}
It follows that $u_{\o^{\beta_1}m_1}^{\H}(x) = a_1$ for all $x$ with $\ind(x) >1$, and then $u_{\o^{\beta_1}m_1+ \gamma}^{\H}(x) = a_1 + u_{\gamma}^{\H_N \times \dots \times \H_2}(\eta_2(x))$ for all $x$ with $\ind(x) >1$ and all $\gamma$. Now with the induction hypothesis on $N-1$ applied to $\H_N \times \dots \times \H_2$, the properties (1)-(3) follow immediately.
\end{proof}

We end this section by stating the semi-continuity properties of these new candidate sequences.

\renewcommand{\labelenumi}{(\arabic{enumi})}
\begin{prop}
\begin{enumerate}
 \item If $\H_k$ is a sequence of u.s.c.d. candidate sequences and $||h^k|| \rightarrow 0$, then $\H = \coprod \H_k$ is a u.s.c.d. candidate sequence.
 \item If $\H_1$ and $\H_2$ are u.s.c.d. candidate sequences and $(\lim\H_2)(\0_2)=0$, then $\H = \H_1 \times \H_2$ is a  u.s.c.d. candidate sequence.
 \item If $\H$ is a u.s.c.d. candidate sequence on $E$ and $F$ is closed subset of $E$, then $\H|_F$ is a u.s.c.d. candidate sequence.
 \item If $\H$ is a u.s.c.d. candidate sequence on $E$ and $^{\pi} \H$ is the lift of $\H$ to $F$, where $\pi : F \rightarrow E$ is a continuous surjection, then $^{\pi} \H$ is a u.s.c.d. candidate sequence.
\end{enumerate}
\end{prop}
\begin{proof}
(1) The condition $||h^k|| \rightarrow 0$ implies that $\H$ has u.s.c. differences at $\0$ for all $k$. \\
(2) Because $\H_1$ is u.s.c.d., the condition $(\lim\H_2)(\0_2)=0$ implies that $\H$ has u.s.c. differences at $(x,\0_2)$ for all $x$ and $k$. \\
(3) The restriction of any u.s.c. function to a subset is also u.s.c. \\
(4) The lift of any u.s.c. function under a continuous map is also u.s.c. \\
\end{proof}

\subsection{Choquet Simplices and Candidate Sequences} \label{simplices}

The relevant chapters of \cite{Ph} provide a good reference for most of the basic facts about simplices required in this work.

Let $K$ be a metrizable, compact, convex subset of a locally convex topological vector space. Then the extreme points of $K$, $\ex(K)$, form a non-empty $G_{\delta}$ subset of $K$. We call a function $f: K \rightarrow \R$ \textbf{affine} (resp. \textbf{convex, concave}) if $f(tx + (1-t)y) = tf(x) + (1-t)f(y)$ (resp. $\leq, \geq$) for all $x$ and $y$ in $K$ and all $t$ in $[0,1]$.
\begin{defn}
Let $K$ be a compact, convex subset of a locally convex topological vector space. Then $K$ is a \textbf{Choquet simplex} if the dual of the continuous affine functions on $K$ is a lattice.
\end{defn}

For any Polish space $E$, let $\mathcal{M}(E)$ be the space of all Borel probabilities on $E$ with the weak* topology. If $E$ is compact, then $\mathcal{M}(E)$ is a Choquet simplex, with the extreme points given by the point measures.

\begin{defn}
Let $K$ be a Choquet simplex. Then we define the \textbf{barycenter map}, $\bary : \mathcal{M}(K) \rightarrow K$, to be the function given for each $\mu$ in $\M(K)$ by
\begin{equation*}
\bary(\mu) = \int y \medspace d\mu(y),
\end{equation*}
where the integral means that for all continuous, affine functions $f: K \rightarrow \R$,
\begin{equation*}
f(\bary(\mu)) = \int_K f \medspace d\mu.
\end{equation*}
\end{defn}
The barycenter map is well-defined, continuous, affine, and surjective (see \cite{Ph}).

If $K$ is a metrizable Choquet simplex, then a function $f: K \rightarrow \R$ is called \textbf{harmonic} (resp. sub-harmonic, sup-harmonic) if for all $\mu$ in $\M(K)$,
\begin{equation*}
f(\bary(\mu))  = \int\limits_{\ex(K)} f \; d\mu,
\end{equation*}
(resp. $\leq, \geq$). A harmonic (resp. sub-harmonic, sup-harmonic) function is always affine (resp. convex, concave), but an affine (resp. convex, concave) function need not be harmonic (resp. sub-harmonic, sup-harmonic). On the other hand, a continuous affine (resp. convex, concave) function is always harmonic (resp. sub-harmonic, sup-harmonic). Furthermore, by standard arguments, any u.s.c. affine (resp. concave) function is harmonic (resp. sup-harmonic). It is shown in the proof of Fact \ref{fonEharIsUSC} (see Appendix B, Section \ref{proofAppendix}) that any u.s.c. convex function is sub-harmonic.

In the metrizable case, Choquet proved the following characterization of Choquet simplices.
\begin{thm}[Choquet] \label{SimplexThm}
Let $K$ be a metrizable, compact, convex subset of a locally convex topological vector space. Then $K$ is a Choquet simplex if and only if for each point $x$ in $K$, there exists a unique Borel probability measure $\P_x$ on $\ex(K)$ such that for every continuous affine function $f : K \rightarrow \R$,
\begin{equation*}
f(x) = \int\limits_{\ex(K)} f \; d\P_x.
\end{equation*}
\end{thm}

\begin{defn}If $K$ is a metrizable Choquet simplex and $f: \ex(K) \rightarrow \R$ is measurable, the \textbf{harmonic extension} $f^{har} : K \rightarrow \R$ of $f$ is defined as follows: for $x$ in $K$, let
\begin{equation*}
f^{har}(x) = \int\limits_{\ex(K)} f \; d\P_x.
\end{equation*}
\end{defn}
\begin{rmk} \label{harmRemark}
Using Choquet's characterization of metrizable Choquet simplices, it is not difficult to show that if $f : K \to \R$ is a measurable function and for each $x$ in $K$,
\begin{equation*}
f(x) = \int f \; d\P_x,
\end{equation*}
then $f$ is harmonic. It follows that the harmonic extension of a function on $\ex(K)$ is, in fact, harmonic.
\end{rmk}

In the metrizable case, the following theorem of Choquet characterizes exactly which topological spaces appear as the set of extreme points of a Choquet simplex.
\begin{thm}[Choquet \cite{Cho}]  \label{choquetWeak}
The topological space $E$ is homeomorphic to the set of extreme points of a metrizable Choquet simplex if and only if E is a Polish space.
\end{thm}

The following fact is stated as Fact 2.5 in \cite{DM}, where there is a sketch of the proof. We include a proof as Appendix B (Section \ref{proofAppendix}) for the sake of completeness.

\begin{fact} \label{fonEharIsUSC}
Let $K$ be a metrizable Choquet simplex, and let $f: K \rightarrow [0, \infty)$ be convex and u.s.c. Then $(f|_{\ex(K)})^{har}$ is u.s.c.
\end{fact}

If $K$ is a metrizable Choquet simplex, we denote by $\M(\ex(K))$ the set of measures $\mu$ in $\M(K)$ such that $\mu(K \setminus \ex(K)) = 0$. Consider the map $\pi : \M(\ex(K)) \rightarrow K$ given by the restriction of the barycenter map to $\M(\ex(K))$. This restriction inherits the continuity and affinity of the barycenter map. Furthermore, this restriction is always bijective (by Choquet's characterization of metrizable Choquet simplices, Theorem \ref{SimplexThm}), but it may not have a continuous inverse. In fact, $\pi$ has a continuous inverse if and only if $\ex(K)$ is closed in $K$. These considerations lead to the study of Bauer simplices.

\begin{defn}
A metrizable, compact, convex subset $K$ of a locally convex topological vector space is a \textbf{Bauer simplex} if $K$ is a Choquet simplex such that $\ex(K)$ is a closed subset of $K$.
\end{defn}

If $E$ is any compact, metrizable space, then $\M(E)$ is a Bauer simplex with $\ex(\M(E))$ homeomorphic to $E$. If $K$ is a Bauer simplex, then the restriction of the barycenter map $\pi : \M(\ex(K)) \rightarrow K$ has a continuous inverse and is therefore an affine homeomorphism from $\mathcal{M}(\ex(K))$ to $K$.

\begin{prop} \label{uscBauerharm}
If $K$ is a Bauer simplex and $f: K \rightarrow [0,\infty)$ is bounded and harmonic, then $\widetilde{f}$ is harmonic and $\widetilde{f}|_{\ex(K)} = \widetilde{f|_{\ex(K)}}$.
\end{prop}
\begin{proof} Since $f$ is harmonic, in particular $f$ is affine. Let $x$ and $y$ be in $K$, and let $ax+by$ be a convex combination in $K$. We have $\widetilde{f}(ax+by) \geq f(ax+by) = af(x) + bf(y)$. For fixed $a,b,$ and $y$, the above formula implies that $\widetilde{f}(ax+by) \geq a\widetilde{f}(x) +b f(y)$. Now fixing $a,b,$ and $x$, we obtain $\widetilde{f}(ax+by) \geq  a\widetilde{f}(x) +b \widetilde{f}(y)$. Now since $\widetilde{f}$ is u.s.c. and concave, it follows that $\widetilde{f}$ is sup-harmonic.

Let $E = \ex(K)$. It follows from the definitions that
\begin{equation} \label{fIneq}
f(t) = \int_E f|_E \; d \P_t \leq \int_E \widetilde{(f|_E)} \; d\P_t \leq \int_{E} \widetilde{f} \; d\P_t.
\end{equation}
Now consider the two functions $g_1, g_2 : K \to \R$, given for each $t$ in $E$ by
\begin{align*}
g_1(t) & = \left\{ \begin{array}{ll} \widetilde{f|_E}(t), & \text{ if } t \in E, \\
                                   0, & \text{ if } t \notin E,
            \end{array} \right. \\
g_2(t) & = \left\{ \begin{array}{ll} \widetilde{f}|_E(t), & \text{ if } t \in E, \\
                                   0, & \text{ if } t \notin E.
            \end{array} \right.
\end{align*}
Since $E$ is closed, $g_1$ and $g_2$ are u.s.c. They are also obviously convex. Then by Fact \ref{fonEharIsUSC}, $G_1 = \bigl((g_1)|_E\bigr)^{har}$ and $G_2 = \bigl((g_2)|_E\bigr)^{har}$ are u.s.c. Note that for $t \in K$,
\begin{equation*}
G_1(t) = \int_E \widetilde{(f|_E)} \; d\P_t, \text{ and } G_2(t) = \int_{E} \widetilde{f} \; d\P_t.
\end{equation*}
Thus, taking the u.s.c. envelope of the expressions in Equation (\ref{fIneq}) and using that $G_1$ and $G_2$ are u.s.c., we have that
\begin{equation} \label{inequalities}
\widetilde{f}(t) \leq \int_E \widetilde{(f|_E)} \; d\P_t \leq \int_{E} \widetilde{f} \; d\P_t,
\end{equation}
which shows that $\widetilde{f}$ is sub-harmonic. Now we have shown that $\widetilde{f}$ is harmonic and the inequalities in Equation (\ref{inequalities}) are all equalities.
\end{proof}


A candidate sequence $\H = (h_k)$ on a Choquet simplex is said to be harmonic if each $h_k$ is harmonic. The following proposition relates the transfinite sequence of a candidate sequence $\H$ on a Bauer simplex $K$ to the transfinite sequence of $\H|_{\ex(K)}$.
\begin{prop}\label{Uaffine}
If $\H$ is a harmonic candidate sequence on the Bauer simplex $K$, then for each $\al$, $u_{\al}^{\H}$ is harmonic and
\begin{equation}\label{uBauer}
u_{\al}^{\H} = (u_{\al}^{\H|_{\ex(K)}})^{har}.
\end{equation}
\end{prop}
\begin{proof}
The proof proceeds by transfinite induction on $\al$. For all $k$, since $h_k$ and $h$ are harmonic, $\tau_k = h - h_k$ is harmonic.

Suppose $u^{\H}_{\al}$ is harmonic and Equation (\ref{uBauer}) holds. Then $u^{\H}_{\al}+\tau_k$ is harmonic. By Proposition \ref{uscBauerharm}, we deduce that $\widetilde{u^{\H}_{\al}+\tau_k}$ is harmonic, and for $t$ in $K$,
\begin{equation*}
(u^{\H}_{\al}+\tau_k)(t) = \int_E \widetilde{(u_{\al}^{\H}+\tau_k)|_E} d \P_t = \int_E \widetilde{(u_{\al}^{\H|_E}+\tau_k)|_E} d \P_t.
\end{equation*}
Recall that $\{u_{\al}+\tau_k\}_k$ is a non-increasing sequence in $k$. Thus we can take the limit in $k$ and apply the Monotone Convergence Theorem to obtain that $u^{\H}_{\al+1}$ is also harmonic, and for $t$ in $K$,
\begin{equation*}
u_{\al+1}^{\H}(t) = \int_E u_{\al+1}^{\H|_E} d \P_t,
\end{equation*}
which implies that Equation (\ref{uBauer}) holds with $\al+1$ in place of $\al$.

The previous arguments apply in a similar way to the case when $\al$ is a limit ordinal.

\end{proof}

\begin{rmk} \label{ConcaveRmk}
Let $K$ be a Choquet simplex which is not necessarily Bauer. Even when the candidate sequence $\H$ on $K$ is harmonic, the functions $u_{\al}^{\H}$ are not in general harmonic. However, we check now that if $\H$ is harmonic, then $u_{\al}^{\H}$ is concave for all $\al$. Assuming by induction that $u_{\al}^{\H}$ is concave, we have that $\widetilde{u_{\al}^{\H}+\tau_k}$ is concave, as it is the u.s.c. envelope of a concave function. Then $u_{\al+1}^{\H}$ is the limit of a sequence of concave functions, and so $u_{\al+1}^{\H}$ is concave. Now for any countable limit ordinal $\al$, there is a strictly increasing sequence $(\al_n)$ of ordinals tending to $\al$. Then $\sup_{\beta < \al} u_{\beta}^{\H} = \lim_n u_{\al_n}^{\H}$ since the sequence $(u_{\beta}^{\H})$ is increasing in $\beta$. Then $\sup_{\beta < \al} u_{\beta}^{\H}$ is concave, as it is the limit of a sequence of concave functions (by induction), and thus $u_{\al}^{\H}$ is concave for any countable limit ordinal as well.
\end{rmk}

When $\ex(K)$ is not compact, $\mathcal{M}(\ex(K))$ is not a Bauer simplex, and the restriction of the barycenter map to this set is not a homeomorphism. Instead of using this restriction in such cases, we consider the Bauer simplex $\mathcal{M}(\overline{\ex(K)})$ and the continuous surjection $\pi : \mathcal{M}(\overline{\ex(K)}) \rightarrow K$, where $\pi$ is the restriction of the barycenter map to $\mathcal{M}(\overline{\ex(K)})$. In the following two lemmas we consider candidate sequences which may arise as embedded candidate sequences.

\begin{lem} \label{embedIsHarmUSC}
Let $E$ be a compact, metrizable space, and let $K$ be a metrizable Choquet simplex. Suppose there exists a continuous injection $g : E \rightarrow K$. Let $\F$ be a u.s.c.d. candidate sequence on $E$, let $\H' = (h'_k)$ be the embedded candidate sequence $g\F$, and let $\H$ be the harmonic extension of $\H'|_{\ex(K)}$ to $K$. If $h'_{k+1}-h'_{k}$ is convex for each $k$, then $\H$ is u.s.c.d. In particular, if $g(E) \subset \ex(K)$ then $\H$ is u.s.c.d.
\end{lem}
\begin{proof}
Since $\F$ is u.s.c.d. and $g(E)$ is closed, we have that $h'_{k+1}-h'_k$ is u.s.c. for each $k$. Then $h'_{k+1}-h'_k$ is convex and u.s.c. for each $k$. By applying Fact \ref{fonEharIsUSC}, we obtain that $h_{k+1}-h_{k}$ is u.s.c. for each $k$. Thus $\H$ is u.s.c.d.

In particular, if $g(E) \subset \ex(K)$, then $h'_{k+1}-h'_k$ takes non-zero values only on $\ex(K)$. Therefore $h'_{k+1}-h'_k$ is convex for each $k$, and by the previous argument, $\H$ is u.s.c.d.
\end{proof}

The following lemma is used repeatedly throughout the rest of this work. The utility of this statement lies in the fact that it allows one to compute the transfinite sequence on a (frequently much simpler) subset of the simplex and then write the transfinite sequence on the entire simplex in terms the transfinite sequence on this subset. When $K$ is a Choquet simplex that is not Bauer and $\H$ is a harmonic candidate sequence on $K$, then this statement takes the place of an integral representation of $u_{\al}^{\H}$.

\begin{lem}[Embedding Lemma]\label{embeddingThm}
Let $K$ be a metrizable Choquet simplex with $E = \ex(K)$. Suppose $\H$ is a harmonic candidate sequence on $K$ and there is a set $F \subset E$ such that the sequence $\{(h-h_k)|_{E \setminus F} \}$ converges uniformly to zero. Let $L = \overline{F}$, and let $\pi:\M(\overline{E}) \rightarrow K$ be the restriction of the barycenter map. Then for all ordinals $\al$ and for all $x$ in $K$,
\begin{equation}\label{embeddingEqn}
u_{\al}^{\H}(x) = \max_{\mu \in \pi^{-1}(x)} \int_{L} u_{\al}^{\H|_L} d\mu,
\end{equation}
and $\al_0(\H) \leq \al_0(\H|_L)$. In particular, if $F$ is compact, then $u_{\al}^{\H}|_F = u_{\al}^{\H|_F}$ for all $\al$ and $\al_0(\H) = \al_0(\H|_F)$.
\end{lem}
\begin{proof}
Note that Equation (\ref{embeddingEqn}) implies immediately that $\al_0(\H) \leq \al_0(\H|_L)$. Further, suppose $F$ is compact. Then $L = F \subset \ex(K)$, and if $x$ is in $F$, then $\pi^{-1}(x)=\{\epsilon_x\}$, where $\epsilon_x$ is the point mass at $x$. In this case Equation (\ref{embeddingEqn}) implies that $u_{\al}^{\H}|_F = u_{\al}^{\H|_F}$ for all $\al$ and $\al_0(\H) = \al_0(\H|_F)$. We now prove Equation (\ref{embeddingEqn}).

Observe that since $L$ is closed and $u_{\al}^{\H|_L}$ is u.s.c., the function $\mathbf{1}_L \cdot u_{\al}^{\H|_L}$ is u.s.c., where $\mathbf{1}_L$ is the characteristic function of the set $L$. Then the function $\mu \mapsto \int_L u_{\al}^{\H|_L} \, d\mu$ is u.s.c., and therefore by Remark \ref{pushDownPullUp} (3), for each $x$ in $K$,
\begin{equation*}
\sup_{\mu \in \pi^{-1}(x)} \int_L u_{\al}^{\H_L} \, d\mu = \max_{\mu \in \pi^{-1}(x)} \int_L u_{\al}^{\H_L} \, d\mu.
\end{equation*}

Let $x$ be in $K$. Since $u_{\al}^{\H}$ is concave (see Remark \ref{ConcaveRmk}) and u.s.c., it follows that $u_{\al}^{\H}$ is sup-harmonic. Therefore
\begin{equation*}
u_{\al}^{\H}(x) \geq \int_{K} u_{\al}^{\H} d\mu, \quad \text{ for all } \mu \in \pi^{-1}(x).
\end{equation*}
Using the fact that $u_{\al}^{\H}|_L \geq u_{\al}^{\H|_L}$, we obtain, for all $\mu \in \pi^{-1}(x)$,
\begin{equation*}
u_{\al}^{\H}(x) \geq \int_{K} u_{\al}^{\H} d\mu \geq \int_{L} u_{\al}^{\H} d\mu \geq \int_{L} u_{\al}^{\H|_L} d\mu.
\end{equation*}
It follows that for each ordinal $\al$,
\begin{equation*}
u_{\al}^{\H}(x) \geq \max_{\mu \in \pi^{-1}(x)} \int_{L} u_{\al}^{\H|_L} d\mu.
\end{equation*}

We now prove using transfinite induction on $\al$ that for all $\al$ and $x$ in $K$,
\begin{equation} \label{reverseIneq}
u_{\al}^{\H}(x) \leq \max_{\mu \in \pi^{-1}(x)} \int_{L} u_{\al}^{\H|_L} d\mu,
\end{equation}
which will complete the proof of the Lemma.

The inequality in Equation (\ref{reverseIneq}) is trivial for $\al=0$. Suppose Equation (\ref{reverseIneq}) holds for some ordinal $\al$. For the sake of notation, we allow $y=x$ in all expressions involving $\limsup_{y \to x}$ below. First we claim that for any $y$ in $K$, there exists a measure $\mu_y$ supported on $L \cup E$ such that $\mu_y$ is in $\pi^{-1}(y)$ and
\begin{equation} \label{maxExpression}
\max_{\mu \in \pi^{-1}(y)} \int_L u_{\al}^{\H|_L} d\mu = \int_L u_{\al}^{\H|_L} d\mu_y.
\end{equation}
Indeed, suppose the maximum is obtained by the measure $\nu$. If $\nu(L)=1$, then we are done. Now suppose $\nu(L) < 1$. Then $\nu = \nu(L) \nu_L + (1-\nu(L)) \nu_{\overline{E} \setminus L}$, where $\nu_S$ is the zero measure on $S$ if $\nu(S)=0$ and otherwise $\nu_S(A) = \frac{1}{\nu(S)} \nu(S \cap A)$. Let $z = \bary(\nu_{\overline{E}\setminus L})$, which exists since $\nu_{\overline{E} \setminus L}$ is in $\M(\overline{E})$ (using that $\nu(\overline{E} \setminus L) = 1 - \nu(L) > 0$). Now let $\mu_y = \nu(L) \nu_L + (1-\nu(L)) \P_{z}$. Then $\mu_y$ is supported on $L \cup E$, $\bary(\mu_y) = y$, and
\begin{equation*}
\int_L u_{\al}^{\H|_L} d\nu \leq \int_L u_{\al}^{\H|_L} d\mu_y.
\end{equation*}
Thus the maximum in Equation (\ref{maxExpression}) is obtained by the measure $\mu_y$, which is supported on $L \cup E$ and satisfies $\bary(\mu)=y$.

Now let $\epsilon >0$. Since $\H$ is harmonic, we also have that $\tau_k$ is harmonic. Then for any $y$ in $K$ and $k$ large enough (depending only on $\epsilon$),
\begin{align}
u_{\al}^{\H}(y)+\tau_k(y) & = \max_{\mu \in \pi^{-1}(y)} \int_L u_{\al}^{\H|_L} d\mu + \tau_k(y) \label{induction} \\
& = \int_L u_{\al}^{\H|_L} d\mu_y + \int \tau_k d\mu_y \label{harmonicity} \\
& = \int_L u_{\al}^{\H|_L} d\mu_y + \int_L \tau_k d\mu_y + \int_{E \setminus L} \tau_k d\mu_y \label{integral} \\
& \leq \int_L u_{\al}^{\H|_L} d\mu_y + \int_L \tau_k d\mu_y + \epsilon \label{uniformity} \\
& =  \int_L (u_{\al}^{\H|_L} + \tau_k) d\mu_y + \epsilon \label{subset} \\
& \leq \int_L \widetilde{(u_{\al}^{\H|_L} + \tau_k)|_L} d\mu_y + \epsilon \\
& \leq \max_{\mu \in \pi^{-1}(y)} \int_L \widetilde{(u_{\al}^{\H|_L} + \tau_k)|_L} d\mu + \epsilon.
\end{align}
Then we have (allowing $y=x$ in the limit suprema) that
\begin{align}
u_{\al+1}^{\H}(x) & = \lim_k \thinspace \limsup_{y \rightarrow x} \thinspace u_{\al}^{\H}(y)+\tau_k(y) \\
 & \leq \lim_k \thinspace \limsup_{y \rightarrow x} \thinspace \max_{\mu \in \pi^{-1}(y)} \int_L \widetilde{(u_{\al}^{\H|_L} +\tau_k)|_L} d\mu + \epsilon \label{usc} \\
 & \leq \lim_k \thinspace \max_{\mu \in \pi^{-1}(x)} \int_L \widetilde{(u_{\al}^{\H|_L} +\tau_k)|_L} d\mu + \epsilon \label{limsup} \\
 & \leq \max_{\mu \in \pi^{-1}(x)} \int_L u_{\al+1}^{\H|_L} d\mu + \epsilon \label{lim},
\end{align}
where the inequalities in (\ref{limsup}) and (\ref{lim}) are justified by Lemmas \ref{limsupLemma} and \ref{limLemma}, respectively. Since $\epsilon$ was arbitrary, we have shown the inequality in Equation (\ref{reverseIneq}) with the ordinal $\al$ replaced by $\al+1$.

Now suppose the inequality in Equation (\ref{reverseIneq}) holds for all $\beta < \al$, where $\al$ is a limit ordinal. Using monotonicity of the sequence $u_{\al}^{\H|_L}$, we see that (allowing $y=x$ in the limit suprema)
\begin{align*}
u_{\al}^{\H}(x) & = \widetilde{\sup_{\beta < \al} u_{\beta}^{\H}}(x) \\
 & = \limsup_{y \rightarrow x} \thinspace \sup_{\beta< \al} \thinspace \max_{\mu \in \pi^{-1}(y)} \int_L u_{\beta}^{\H|_L} d\mu \\
 & \leq \limsup_{y \rightarrow x} \thinspace \max_{\mu \in \pi^{-1}(y)} \int_L u_{\al}^{\H|_L} d\mu \\
 & \leq \max_{\mu \in \pi^{-1}(x)} \int_L u_{\al}^{\H|_L} d\mu,
\end{align*}
where Lemma \ref{limsupLemma} justifies the last inequality. Thus we have shown that the inequality in Equation (\ref{reverseIneq}) holds for $\al$, which completes the induction and the proof.
\end{proof}

\begin{rmk} \label{embeddingRemark}
Given the assumptions of the Embedding Lemma, if $x$ is in $\ex(K)$, then $\pi^{-1}(x) = \{ \epsilon_x \}$, where $\epsilon_x$ is the point mass at $x$. It follows that, if $x$ is in $L \cap \ex(K)$, then $u_{\al}^{\H}(x) = u_{\al}^{\H|_{L}}(x)$ for all $\al$. Further, if $x$ is in $\ex(K) \setminus L$, then $u_{\al}^{\H}(x)=0$ for all $\al$.
\end{rmk}
\begin{rmk} \label{embeddingRemarktwo}
With the notation of the Embedding Lemma, Equation (\ref{embeddingEqn}) implies that $||u_{\al}^{\H}||=||u_{\al}^{\H|_L}||$ for all $\al$.
\end{rmk}

\begin{lem} \label{limsupLemma}
Let $K$ be a metrizable Choquet simplex and $L$ a closed subset of $K$. Let $f: K \rightarrow [0, \infty)$ be u.s.c. Then for all $x$ in $K$,
\begin{equation*}
\limsup_{y \rightarrow x} \max_{\mu \in \pi^{-1}(y)} \int_L f d\mu \leq \max_{\mu \in \pi^{-1}(x)} \int_L f d\mu,
\end{equation*}
where $\pi$ is the restriction of the barycenter map on $\M(K)$ to $\M(\overline{\ex(K)})$.
\end{lem}
\begin{proof}
Let $T : \M(K) \rightarrow \R$ be defined by $T(\mu) = \int_L f d\mu$. We have that $f \chi_L$ is u.s.c. since $f$ is non-negative and u.s.c. and $L$ is closed. It follows that $T$ is u.s.c. Then the result follows from Remark \ref{pushDownPullUp} (3).
\end{proof}

\begin{lem} \label{limLemma}
Let $K$ be a metrizable Choquet simplex and $L$ a closed subset of $K$. Let $\{f_k: K \rightarrow [0, \infty) \}$ be a non-increasing sequence of u.s.c. functions, with $\lim_k f_k = f$. Then for all $x$ in $K$,
\begin{equation*}
\lim_{k \rightarrow \infty} \max_{\mu \in \pi^{-1}(x)} \int_L f_k d\mu \leq \max_{\mu \in \pi^{-1}(x)} \int_L f d\mu,
\end{equation*}
where $\pi$ is the restriction of the barycenter map on $\M(K)$ to $\M(\overline{\ex(K)})$.
\end{lem}
\begin{proof}
Let $x$ be in $K$. Define $T_k : \M(K) \rightarrow \R$ and $T: \M(K) \rightarrow \R$ by the equations
\begin{equation*}
T_k(\mu)  = \int_L f_k d\mu, \; \text{ and } \; T(\mu) = \int_L f d\mu.
\end{equation*}
Since $f_k \chi_L$ and $f \chi_L$ are u.s.c., $T$ and $T_k$ are u.s.c. Proposition 2.4 of \cite{BD} states (in slightly greater generality) that
\begin{equation}\label{BDEqn}
\lim_k \max_{\mu \in \pi^{-1}(x)} T_k(\mu) = \max_{\mu \in \pi^{-1}(x)} \lim_k T_k(\mu).
\end{equation}
By the Monotone Convergence Theorem,
\begin{equation}\label{TequalsLimTk}
T(\mu) = \lim_k T_k(\mu).
\end{equation}
Combining Equations (\ref{BDEqn}) and (\ref{TequalsLimTk}) concludes the proof.
\end{proof}

Even when the hypotheses of the Embedding Lemma are satisfied, it is possible to have $\al_0(\H) < \al_0(\H|_L)$, as the next example shows.
\begin{example} \label{HlessThanPiH} This example provides a candidate sequence $\H$ satisfying the hypotheses of the Embedding Lemma and  $\al_0(\H) < \al_0(\H|_L)$, which proves that the inequality $\al_0(\H) \leq \al_0(\H|_L)$ is not an equality in general.  Suppose the set of extreme points of $K$ consists of two points, $b_1$ and $b_2$, sequences $\{ c_n \}$ and $\{d_n \}$ with $c_n \rightarrow b_1$ and $d_n \rightarrow b_2$, and a countable collection $\{ a_n \}$. Let $b = \frac{1}{2}(b_1+b_2)$ in $K$. Suppose further that with the subspace topology inherited from $K$, the set $\{ a_n \} \cup \{b\}$ is homeomorphic to $\o^2 +1$, with the homeomorphism given by $g_1: \o^2 +1 \rightarrow \{ a_n \} \cup \{b\}$ and $g_1(\o^2)=b$. One may construct such a simplex $K$ as the image of $\M(\{a_n\} \cup \{b, b_1, b_2\} \cup \{c_n\} \cup \{d_n\})$ under a continuous affine map (Lemma \ref{ConstructKLem}). Let $\F_1 = (f_k^1)$ be u.s.c.d. candidate sequence on $\o^2+1$ such that $\al_0(\F_1)=2$, $u_1^{\F_1}(t) = u_2^{\F_1}(t)$ for $t \neq \o^2$, and $||u_{2}^{\F_1}||=1$. Such a sequence is given by Corollary \ref{realizationCor}. Let $\F_2 = (f_k^2)$ be the u.s.c.d. candidate sequence on $\{c_n\} \cup \{d_n\} \cup \{b_1,b_2\}$ given, for $x$ in $\{c_n\} \cup \{d_n\} \cup \{b_1,b_2\}$ and $k \geq 1$, by
\begin{equation*}
 f_k^2(x) = \left \{ \begin{array}{ll} 0 & \text{ if } x = c_n \text{ or } x=d_n, \text{ with } k < n \\
                    1 & \text{ otherwise}.
                   \end{array}
            \right.
\end{equation*}

Now consider the candidate sequence $\H' = (h'_k)$ on $K$ such that for $x$ in $K$,
\begin{equation*}
h'_k(x) = \left \{
\begin{array}{ll}
f_k^1(g_1^{-1}(x)) & \text{ if } x = a_n \\
f_k^2(x) & \text{ if } x = c_n, \medspace d_n \\
0 & \text{ otherwise } \\
\end{array} \right.
\end{equation*}
Note that $\H'$ is u.s.c.d., convex, and $h'_{k+1} - h'_k$ is convex. Let $\H$ be the harmonic extension of $\H'|_{\ex(K)}$ on $K$. Then by Lemma \ref{embedIsHarmUSC}, $\H$ is harmonic and u.s.c.d.

Let $F = \ex(K)$ and $L = \overline{F} = \{a_n\} \cup \{b, b_1, b_2\} \cup \{c_n\} \cup \{d_n\}$. Since $L$ is the disjoint union the two (clopen in $L$) sets $\{a_n\} \cup \{b\}$ and $\{b_1, b_2\} \cup \{c_n\} \cup \{d_n\}$, we see that for $t$ in $L$,
\begin{align*}
u_{\al}^{\H|_L} = \left\{ \begin{array}{ll}
                  u_{\al}^{\F_1}(t), & \text{ if } t \in \{a_n\} \cup \{b\} \\
                  u_{\al}^{\F_2}(t), & \text{ if } t \in \{b_1, b_2\} \cup \{c_n\} \cup \{d_n\}.
              \end{array} \right.
\end{align*}
Thus $\al_0(\H|_L) = \max(\al_0(\F_1),\al_0(\F_2)) = \al_0(\F_1) = 2$ and $||u_2^{\H|_L}|| \leq 1$. Also, for all $t \neq b$, $u_1^{\H|_L}(t) = u_2^{\H|_L}(t)$, and for $t \in \{b_1,b_2\}$, $u_1^{\H|_L}(t) = 1$.

Applying the Embedding Lemma, we have that for all $t$ in $K$,
\begin{equation} \label{ETexampleEqn}
u_{\al}^{\H}(t) = \max_{\mu \in \pi^{-1}(t)} \int_L u_{\al}^{\H|_L} \, d\mu.
\end{equation}
If $\mu \in \pi^{-1}(t)$ and $\mu(\{b\}) >0$, then let $\nu = \frac{1}{2}\mu(\{b\})(\epsilon_{b_1} + \epsilon_{b_2}) + (1-\mu(\{b\})) \mu_{L \setminus \{b\}}$, where $\mu_{L \setminus \{b\}}$ is the measure $\mu$ conditioned on the set $L \setminus \{b\}$. Then $\nu \in \pi^{-1}(t)$, $\nu(\{b\}) = 0$, and $\int_L u_{i}^{\H|_L} \, d\mu \leq \int_L u_{i}^{\H|_L} \, d\nu$ for $i \in \{1, 2\}$. Thus the maximum in Equation (\ref{ETexampleEqn}) is obtained by a measure $\mu$ with $\mu(\{b\}) = 0$. Now if $\mu \in \pi^{-1}(t)$ and $\mu(\{b\}) = 0$, then $\int_L u_{1}^{\H|_L} \, d\mu = \int_L u_{2}^{\H|_L} \, d\mu$ since $u_1^{\H|_L}(s) = u_2^{\H|_L}(s)$ for $s \in L \setminus \{b\}$. From these facts we deduce $u_{1}^{\H}(t) = u_2^{\H}(t)$ for all $t$ in $K$, and therefore $\al_0(\H) = 1 < \al_0(\H|_L)$.
\end{example}


\section{Realization of Transfinite Orders of Accumulation} \label{examples}

Recall that for every countable ordinal $\al$, $\o^{\al}+1$ is a countable, compact, Polish space. Then let $K_{\al}$ be the (unique up to affine homeomorphism) Bauer simplex with $\ex(K_{\al}) = \o^{\al}+1$. For notation, let $\0_{\al}$ be the point $\o^{\al}$ in $K_{\al}$, and let $E_{\al} = \ex(K_{\al})$. In this section we construct, for each countable $\al$, a harmonic, u.s.c.d. candidate sequence $\H_{\al}$ on $K_{\al}$ such that $\al_0(\H_{\al})=\al$.

The idea of the following theorem is to construct, for each countable, irreducible ordinal $\al$, a candidate sequence $\H$ such that the transfinite sequence does not converge uniformly at $\al$, in some sense. The main tools of the proof are the disjoint union candidate sequence and the powers candidate sequences.

\begin{thm} \label{mainTech}
For all real numbers $0< \epsilon < a$, and for all countable, irreducible ordinals $\delta$ and $\al$, with $\delta < \al$, there exists a harmonic, u.s.c.d candidate sequence $\H_{\al}$ on $K_{\al}$ such that
\begin{enumerate}
  \item $||h|| \leq a$ if $\al$ is finite, and $||h|| \leq \epsilon$ if $\al$ is infinite;
  \item $||u_{\delta} || \leq \epsilon$;
  \item $||u_{\gamma}|| \leq a$ for all $\gamma$, and $||u_{\gamma}|| < a$ for $\gamma < \al$;
  \item $h(\0_{\al})=0, \medspace u_{\gamma}(\0_{\al}) = 0$, for all $ \gamma < \al$, and $u_{\al}(\0_{\al})=a$;
  \item $\al_0(\H_{\al}) = \al$.
\end{enumerate}
\end{thm}
\begin{proof}
Suppose that we have constructed an u.s.c.d. candidate sequence $\H'$ on $\o^{\al}+1$ and shown that it possesses properties (1)-(5). Since $K_{\al}$ is Bauer, Proposition \ref{Uaffine} implies that we can let $\H_{\al}$ be the harmonic extension of $\H'$ to $K_{\al}$ and properties (1)-(5) carry over exactly. So without loss of generality, we will define $\H_{\al}$ directly on $E_{\al}$ and work exclusively on $E_{\al}$.

The rest of the proof proceeds by transfinite induction on the non-zero irreducible ordinals $\al$ ($\al$ is non-zero because $\delta<\al$). This is equivalent, by Proposition \ref{irreducible}, to writing $\al = \o^{\beta}$ and using transfinite induction on $\beta$. The base case is when $\beta = 0$.

\begin{case}[$\beta = 0$]
In this case $E_{\o^0} = E_1 = \o + 1$, the one-point compactification of the natural numbers. Now $\delta$ must be $0$ and by definition $u_0 \equiv 0$. Let $\H= (h_k)$, where $h_k(n) = 0$ if $k \leq n$, $h_k(n) = a$ if $k > n$, and $h_k(\0_1) = 0$. Then $h \leq a$. Since each $n$ is isolated in $E_{1}$, $r(n) = 0$, which implies that $\al_0(n) = 0$ and $u_{\gamma}(n)=0$ for all $\gamma$ (by Proposition \ref{pointwiseBound}). The point at infinity, $\0_1$, has topological order of accumulation $1$, which implies that $\al_0(\0_1) \leq 1$ (by Proposition \ref{pointwiseBound}). It only remains to check that $u_1(\0_1)=a$. Fix $k$. For any $n > k$, $\tau_k(n) = h(n)-h_k(n) = a$. Thus $\widetilde{\tau_k}(\0_1) \geq a$. Letting $k$ go to infinity gives that $u_1(\0_1) \geq a$. Since $u_1 \leq \widetilde{h} \leq a$, we obtain that $u_1(\0_1) = a$, as desired.

\end{case}

\begin{case}[$\beta$ implies $\beta+1$]
We assume the statement is true for $\o^{\beta}$, and we need to show that it is true for $ \o^{\beta+1} = \sup_n \o^{\beta} n$. In this case $E_{\o^{\beta+1}}$ is homeomorphic to the one-point compactification of the disjoint union of the spaces $(E_{\o^{\beta} n})$ (by Theorem \ref{classification}). With this homeomorphism, we may assume without loss of generality that $E_{\o^{\beta+1}}$ is the one-point compactification of the disjoint union of the spaces $E_{\o^{\beta}n}$. Fix $0 < \epsilon < a$, and let $\{a_p\}$ be a sequence of positive real numbers such that $a_p < a$ for all $p$ and $\lim_p a_p = a$. Using the induction hypothesis, for each $p$, we choose a u.s.c.d. candidate sequence $\H_{\o^{\beta}}$ on $E_{\o^{\beta}}$ which satisfies conditions (1)-(5) with parameters $a_p$, $\epsilon$, and $\delta < \o^{\beta}$. For each $p$, let $\H_{\o^{\beta}}^p$ be the $p$-power sequence of this $\H_{\o^{\beta}}$ restricted to $E_{\o^{\beta} p}$ (note that $\o^{\o^{\beta}p}+1 \subset (\o^{\o^{\beta}}+1)^p$). Then $||\lim (\H_{\o^{\beta}}^p)|| \leq \frac{a}{p}$, and $||u_{\o^{\beta}}^{\H_{\o^{\beta}}^p}|| \leq \frac{a_p}{p}$. Let $N$ be such that $\frac{a}{N} \leq \epsilon$, and define $\H_{\o^{\beta+1}} = \coprod_{n \geq N} \H_{\o^{\beta}}^n$. It remains to check (1)-(5) for $\H_{\o^{\beta+1}}$. \\
(1) Using that $h(\0_{\o^{\beta+1}})=0$,
\begin{equation*}
||h|| = \sup_{n \geq N} ||\lim \H_{\o^{\beta}}^n|| = ||\lim \H_{\o^{\beta}}^N|| \leq \frac{a}{N} \leq \epsilon < a.
\end{equation*} \\
(2) For irreducible $\delta < \o^{\beta +1}$, we have $\delta \leq \o^{\beta}$. Monotonicity of the transfinite sequence implies
\begin{equation*} ||u_{\delta}^{\H_{\o^{\beta}}^n}|| \leq ||u_{\o^{\beta}}^{\H_{\o^{\beta}}^n}||,
\end{equation*}
for every $n$. Also, Lemma \ref{disUnionLemma} implies
\begin{equation*}
||u_{\delta}|| = \sup_{n \geq N} ||u_{\delta}^{\H_{\o^{\beta}}^n}||.
\end{equation*}
Putting these inequalities together gives
\begin{equation*}
||u_{\delta} || = \sup_{n \geq N} ||u_{\delta}^{\H_{\o^{\beta}}^n}|| \leq \sup_{n \geq N} ||u_{\o^{\beta}}^{\H_{\o^{\beta}}^n}|| \leq \frac{a}{N} \leq \epsilon.
\end{equation*} \\
(3) For every $\gamma$, Lemma \ref{disUnionLemma} and Lemma \ref{powerLem} (1) imply
\begin{equation*}
||u_{\gamma}|| = \sup_{n \geq N} ||u_{\gamma}^{\H_{\o^{\beta}}^n}|| \leq a.
\end{equation*}
Further, for any $\gamma < \al$, there exists $m$ such that $\gamma < \o^{\beta} m$. Using subadditivity (Lemma \ref{UisZero}), $||u_{\gamma}^{\H^n_{\o^{\beta}}}|| \leq ||u_{\o^{\beta}m}^{\H^n_{\o^{\beta}}}|| \leq \frac{m}{n} a_n$. Then
\begin{equation*}
||u_{\gamma}|| = \sup_{n \geq N} ||u_{\gamma}^{\H_{\o^{\beta}}^n}|| \leq \max \Bigl( a_1, \dots, \, a_m, \, \sup_{n > m} \frac{m}{n} a_n \Bigr) < a.
\end{equation*} \\
(4) By definition, $h(\0_{\o^{\beta+1}}) = 0$. Let $\gamma < \al$. There exists a $k$ such that $\gamma < \o^{\beta} k$. Then Lemma \ref{disUnionLemma}, monotonicity, and Lemma \ref{powerLem} imply
\begin{equation*}
u_{\gamma}(\0_{\o^{\beta+1}}) \leq \limsup_{n \rightarrow \infty} ||u_{\gamma}^{\H_{\o^{\beta}}^n}|| \leq \limsup_{n \rightarrow \infty} ||u_{\o^{\beta}k}^{\H_{\o^{\beta}}^n}|| \leq \limsup_{n \rightarrow \infty} \frac{k a}{n} = 0.
\end{equation*}
Also, Lemma \ref{disUnionLemma} and Lemma \ref{powerLem} imply
\begin{equation*}
u_{\al}(\0_{\o^{\beta+1}}) \geq \limsup_{n \rightarrow \infty} u_{\al}^{\H_{\o^{\beta}}^n}(\0_{\o^{\beta} n}) = a,
\end{equation*}
which (combining with (3)) implies that $u_{\al}(\0_{\o^{\beta+1}}) = a$. \\
(5) For $x \neq \0_{\o^{\beta+1}}$, there exists $n$ such that $x \in E_{\o^{\beta} n}$, which implies that $r(x) \leq \o^{\beta} n$. Then Proposition \ref{pointwiseBound} gives that $\al_0(x) \leq \o^{\beta} n +1 < \o^{\beta+1}$. The fact that $\al_0(\0_{\o^{\beta+1}}) = \o^{\beta+1}$ then follows immediately from (3) and (4). Thus $\al_0(\H) = \o^{\beta+1}$.
\end{case}

\begin{case}[$\beta$ limit ordinal]
We assume the statement is true for all $\o^{\xi}$ with $\xi < \beta$, and we need to show that it is true for $\o^{\beta}$. In this case there is a strictly increasing sequence of irreducible ordinals $(\o^{\beta_n})$ with $\sup_n \o^{\beta_n} = \o^{\beta}$, and $E_{\o^{\beta}}$ is homeomorphic to the one-point compactification of the disjoint union of the $E_{\o^{\beta_n}}$ (by Remark \ref{disjUnion}). With this homeomorphism, we may assume without loss of generality that $E_{\o^{\beta}}$ is the one-point compactification of the disjoint union of the spaces $E_{\o^{\beta_n}}$. Fix $0 < \epsilon < a$, and let $\{a_n\}$ be a sequence of positive real numbers with $a_n<a$ for all $n$ and $\lim_n a_n =a$. By the induction hypothesis, for each $n>1$, there exists a u.s.c.d. candidate sequence $\H_{\o^{\beta_n}}$ on $E_{\o^{\beta_n}}$ satisfying (1)-(5) with parameters $a_n$, $\frac{\epsilon}{n}$, $\o^{\beta_n}$ and $\delta_n = \o^{\beta_{n-1}}$. Now fix $\delta$ irreducible with $\delta < \o^{\beta}$. Since $\sup_n \o^{\beta_n} = \o^{\beta}$, there exists $N$ such that $\o^{\beta_{N-1}} > \delta$. Let $\H_{\o^{\beta}} = \coprod_{n \geq N} \H_{\o^{\beta_n}}$. All that remains is to verify (1)-(5). \\
(1) Using that $h(\0_{\o^{\beta}})=0$, we get
\begin{equation*}
||h|| = \sup_{n \geq N} ||\lim \H_{\o^{\beta_n}}|| \leq \frac{\epsilon}{N} \leq \epsilon .
\end{equation*} \\
(2) Since $\delta < \o^{\beta_{N-1}}$, Lemma \ref{disUnionLemma} and monotonicity imply (as in the previous case)
\begin{equation*}
||u_{\delta}|| \leq \sup_{n \geq N} ||u_{\delta}^{\H_{\o^{\beta_n}}}|| \leq \sup_{n \geq N} ||u_{\o^{\beta_{n-1}}}^{\H_{\o^{\beta_n}}}|| \leq \sup_{n \geq N} \frac{\epsilon}{n} \leq \epsilon.
\end{equation*} \\
(3) For any $\gamma$, by construction,
\begin{equation*}
||u_{\gamma}|| \leq \sup_{n \geq N} ||u_{\gamma}^{\H_{\o^{\beta_n}}}|| \leq a.
\end{equation*}
Further, for $\gamma < \al$, there exists $m$ such that $\gamma < \o^{\beta_m}$. For $n > m$, $||u_{\gamma}^{\H_{\o^{\beta_n}}}|| \leq \frac{\epsilon}{n}$. Then
\begin{equation*}
||u_{\gamma}|| \leq \sup_{n \geq N} ||u_{\gamma}^{\H_{\o^{\beta_n}}}|| \leq \max \Bigl( a_1, \dots, \, a_m, \, \sup_{n > m} \frac{\epsilon}{n} \Bigr) < a.
\end{equation*} \\
(4) By definition, $h(\0_{\o^{\beta}}) = 0$. For any $\gamma < \o^{\beta}$, there exists some $k$ such that for all $n \geq k$, $\o^{\beta_n} > \gamma$. Then
\begin{equation*}
u_{\gamma}(\0_{\o^{\beta}}) \leq \limsup_{n \rightarrow \infty} ||u_{\gamma}^{\H_{\o^{\beta_n}}}|| \leq \limsup_{n\rightarrow \infty} ||u_{\o^{\beta_{n-1}}}^{\H_{\o^{\beta_n}}}|| \leq \limsup_{n \rightarrow \infty} \frac{\epsilon}{n} = 0.
\end{equation*} \\
(5) For any $x \neq \0_{\al}$, there exists $n$ such that $x \in E_{\o^{\beta_n}}$. Then $\al_0(x) \leq r(x) \leq \o^{\beta_n} < \o^{\beta}$. By (3) and (4), $\al_0(\0_{\o^{\beta}}) = \o^{\beta}$. Therefore $\al_0(\H)=\o^{\beta}$.
\end{case}
\end{proof}

\begin{cor}\label{realizationCor}
For all positive real numbers $a$ and non-zero countable ordinals $\al$, there exists a harmonic, u.s.c.d. candidate sequence $\H$ on $K_{\al}$ such that the transfinite sequence corresponding to either $\H$ or $\H|_{\ex(K_{\al})}$ satisfies
\begin{enumerate}
  \item $||u_{\gamma}|| \leq a$ for all $\gamma$, and $||u_{\gamma}|| < a$ for all $\gamma < \al$;
  \item $h(\0_{\al})=0$, and $u_{\al}(\0_{\al})=a$;
  \item $\al_0(\H) = \al_0(\H|_{\ex(K_{\al})})= \al$.
\end{enumerate}
\end{cor}
\begin{proof}
Let $\al$ be a non-zero countable ordinal, and suppose the Cantor Normal Form of $\al$ (as in Theorem \ref{CNF}) is given by
\begin{equation*}
 \al = \al_1 m_1 + \dots + \al_N m_N.
\end{equation*}
Let $a_1 > \dots > a_N > 0$ be real numbers such that $\sum a_j = a$ and for each $j =1, \dots, N-1$,
\begin{equation*}
\frac{a_j}{m_j} \geq \sum_{i=j+1}^N a_i.
\end{equation*}
For each $j = 1, \dots, N$, let $\F_j$ be a harmonic, u.s.c.d. candidate sequence given by Theorem \ref{mainTech} with parameters $a_j$ and $\al_j$. Define $\H_j$ to be the product sequence $\F_j^{m_j}$ restricted to $K_{\al_j m_j}$, and let $\H = \H_N \times   \dots  \times \H_1 $ restricted to $K_{\al}$. By definition of $\H$, $h(\0_{\al})=0$. The rest of properties (1)-(3) follow from Lemma \ref{prodLemma}.

\end{proof}

\begin{cor} \label{alphaPlusOne}
Let $a>0$, and let $\al$ be a countable, infinite ordinal. Then there is a harmonic, u.s.c.d. candidate sequence $\H$ on $K_{\al}$ such that the transfinite sequence corresponding to either $\H$ or $\H|_{\ex(K_{\al})}$ satisfies
\begin{enumerate}
 \item $||u_{\gamma}|| \leq a$ for all $\gamma$, and $||u_{\gamma}|| < a$ for $\gamma < \al+1$;
 \item $h(\0_{\al})=0$ and $u_{\al+1}(\0_{\al}) = a$;
 \item $\al_0(\H) = \al_0(\H|_{\ex(K_{\al})}) =  \al+1$.
\end{enumerate}
\end{cor}
\begin{proof}
Using Proposition \ref{Uaffine}, we may deal exclusively with u.s.c.d. candidate sequences on $E_{\al}$ (as opposed to $K_{\al}$), and all properties will carry over to $K_{\al}$.

The proof is executed in two stages. First we prove the statement for the countably infinite, irreducible ordinals. In the second stage, we prove the statement for all countable, infinite ordinals.

\textbf{Stage 1.} Let $\al$ be a countably infinite, irreducible ordinal. Let $\al = \o^{\beta}$ (since $\al$ is infinite, $\beta >0$). and let $b = \frac{2}{3} a$. Let $\F$ be given by Theorem \ref{mainTech} with parameters $b$, $\al$, $\epsilon$, and $\delta$.  Recall from the proof of Theorem \ref{mainTech} that we may take $\F = \sqcup \F_n$, where the exact form of the $\F_n$ is as follows. Let $\{a_n\}$ be a sequence of positive real numbers with $a_n < b$ for all $n$ and $\lim_n a_n = b$. If $\beta$ is a successor, then we may take $\F_n = \G^n$, where $\G$ satisfies the conclusions of Theorem \ref{mainTech} with parameters $a_n$, $\epsilon$, $\o^{\beta-1}$, and $\delta$. Otherwise, if $\beta$ is a limit with $\beta_n$ increasing to $\beta$, then $\F_n$ satisfies the conclusions of Theorem \ref{mainTech} with parameters $a_n$, $\epsilon$, $\o^{\beta_n}$, and $\delta$. Let $\F = (f_k)$, and let $\0_n$ denote the marked point in $E_{\beta_n}$ (so $E_{\beta_n}$ is the domain of $\F_n$). Let $\H = (h_k)$ be defined by the rule
\begin{equation*}
h_k(x) = \left\{
\begin{array}{ll}
f_k(x) & \text{ if } x \neq \0_n \\
0 & \text{ if } x = \0_n, \medspace k \leq n \\
\frac{b}{2} & \text{ if } x = \0_n, \medspace k > n.
\end{array} \right.
\end{equation*}
By definition, let $h_k(\0_{\al})=0$. Note that $\H$ is again an u.s.c.d. sequence on $E_{\al}$, and $u_{\gamma}^{\H}(x) = u_{\gamma}^{\F}(x)$ for all $\gamma$ and all $x \neq \0_{\al}$. It follows that $u_{\gamma}^{\H}(x) \leq b$ for all $\gamma$ and all $x \neq \0_{\al}$. Computing the transfinite sequence at $\0_{\al}$, we see that
\begin{align*}
 u_{\ell}^{\H}(\0_{\al}) & = \frac{b}{2}, \text{ for } 1 \leq \ell < \al \\
 u_{\al}^{\H}(\0_{\al}) & = b \\
 u_{\al+1}^{\H}(\0_{\al}) & = b + \frac{b}{2} = a.
\end{align*}
Since $\al_0(\0_{\al}) \leq r(\0_{\al}) + 1 = \al+1$, we conclude that $\al_0(\0_{\al}) = \al+1$. Thus we obtain properties (1)-(3).

\textbf{Stage 2.} Let $\al = \o^{\beta_1} m_1 + \dots \o^{\beta_N} m_N$ be the Cantor Normal Form of $\al$.

The construction proceeds by cases. In the first case, suppose $\o^{\beta_N}$ is infinite. Let $a>0$, and select $a_1 > \dots > a_N$ as in Lemma \ref{prodLemma}. Let $\F_j$ be given by Lemma \ref{mainTech} with parameters $a_j$ and $\o^{\beta_j}$, for $j = 1, \dots , N$. Let $\F'_N$ be given by Stage 1 corresponding to $\frac{a_N}{m_N}$ and $\o^{\beta_N}$. For $j=1, \dots, N-1$, let $\H_j = \F_j^{m_j}$, and for $j=N$, if $m_N >1$, let $\H_j = \F_N^{m_N-1}$. Now let $\H'$ be given by the product (where $\H_N$ is omitted if $m_N = 1$)
\begin{equation*}
\H' = \F'_N \times (\H_N) \times \dots \times (\H_1),
\end{equation*}
Let $\H$ be the restriction of $\H'$ to $E_{\o^{\al}+1}$. Note that $h(\0_{\al})=0$. Then using Lemmas \ref{powerLem} and \ref{prodLemma}, we conclude that
\begin{equation*}
\al_0(\H) = \big( \sum_{i=1}^{N-1} \o^{\beta_i} m_i \big) + \o^{\beta_N} (m_N-1) + (\o^{\beta_N}+1) = \al+1.
\end{equation*}

For the second case, we suppose that $\o^{\beta_N}$ is finite, which implies that $\o^{\beta_N} = 1$. Let $a>0$, and select $a_1 > \dots > a_N$ as in Lemma \ref{prodLemma}, with the additional condition that $\frac{a_{N-1}}{3 m_{N-1}} \geq a_N$. Let $\F_j$ be given by Lemma \ref{mainTech} with parameters $a_j$ and $\o^{\beta_j}$, for $j = 1, \dots , N$.  Since $\al$ is infinite, it follows that $\o^{\beta_{N-1}}$ is infinite. Let $\F'_{N-1}$ be given by Stage 1 corresponding to $\frac{a_{N-1}}{m_{N-1}}$ and $\o^{\beta_{N-1}}$ (so that the condition $\frac{a_{N-1}}{3 m_{N-1}} \geq a_N$ implies $b/2 \geq a_N$ in the notation of Stage 1). For $j \in \{1, \dots, N-2, N\}$, let $\H_j = \F_j^{m_j}$. If $m_{N-1} >1$, let $\H_{N-1} = \F_{N-1}^{m_{N-1} - 1}$. Now let $\H'$ be given by the product (where $\H_{N-1}$ is omitted if $m_{N-1} = 1$):
\begin{equation*}
\H' =  (\H_N) \times \F'_{N-1} \times (\H_{N-1}) \times \dots \times (\H_1),
\end{equation*}
Let $\H$ be the restriction of $\H'$ to $E_{\o^{\al}+1}$. Note that $h(\0_{\al})=0$. Then the reader may easily adapt the proofs of Lemmas \ref{powerLem} and \ref{prodLemma} with the additional assumption that $\frac{a_{N-1}}{3 m_{N-1}} \geq a_N$ to check that
\begin{align*}
||u_{\o^{\beta_1}m_1 + \dots + \o^{\beta_{N-1}}m_{N-1}}^{\H} || & = \sum_{i=1}^{N-2} a_i + \bigl(\frac{a_{N-1}}{m_{N-1}}(m_{N-1}-1)\bigr) + \frac{a_{N-1}}{m_{N-1}}\bigl(\frac{2}{3}\bigr) \\
||u_{\o^{\beta_1}m_1 + \dots + \o^{\beta_{N-1}}m_{N-1}+1}^{\H} || & = \sum_{i=1}^{N-1} a_i  \\
||u_{\o^{\beta_1}m_1 + \dots + \o^{\beta_{N-1}}m_{N-1}+1 + k}^{\H} || & = \sum_{i=1}^{N-1} a_i + \frac{a_N}{m_N} k, \, \text{ for $k = 1, \dots, m_N$,}
\end{align*}
and,
\begin{equation*}
\al_0(\H) = \big( \sum_{i=1}^{N-2} \o^{\beta_i} m_i \big) + \o^{\beta_{N-1}} (m_{N-1}-1) + (\o^{\beta_{N-1}}+1) + m_N  = \al+1.
\end{equation*}
\end{proof}

\begin{rmk} \label{LimsupToLimitRmk}
In Corollaries \ref{realizationCor} and \ref{alphaPlusOne}, one may further require that $\H|_{\ex(K_{\al})}$ has the following property (P): for any $t$ in $\ex(K_{\al})$, for any sequence $\{s_n\}$ of isolated points in $\ex(K_{\al})$ that converges to $t$, $\limsup_n \tau_k(s_n) = \lim_n \tau_k(s_n)$. Let us prove this fact. In the case $\al = 1$, there is only one sequence of isolated points in $\ex(K_1) \cong \o +1$, and the candidate sequence $\F$ constructed in the proof of Theorem \ref{mainTech} satisfies (P). Then we note that if each of the candidate sequences $\F_1, \dots, \F_N$ satisfies this property, then so does the product $\F = \F_1 \times \dots \times \F_N$. To see this fact, note that the projection $\pi_N$ onto the last coordinate of any isolated point $x$ in the product space is not the marked point $\mathbf{0}_N$, and thus $\F(x) = \F_N( \pi_N(x) )$. Hence the product candidate sequence satisfies property (P) because $\F_N$ does.  Now suppose there is a sequence $(\F_n)_n$ of candidate sequences such that each $\F_n$ satisfies (P). Let $h^n = \lim \F_n$ and let $I_n$ be the set of isolated points in the domain of $\F_n$. Further suppose that $h^n|_{I_n}$ converges uniformly to $0$. Then $\coprod_n \F_n$ satisfies (P) as well (to see this, note that property (P) is satisfied on the domain of each candidate sequence $\F_n$ separately because $\F_n$ has property (P), and then it is satisfied at the point at infinity because $h^n|_{I_n}$ converges uniformly to $0$). The constructions used in the proofs of Theorem \ref{mainTech}, Corollary \ref{realizationCor} and Corollary \ref{alphaPlusOne} only rely on these three types of constructions ($\al = 1$, product sequences, and disjoint union sequences with $h^n|_{I_n}$ tending uniformly to $0$), and thus at each step we may choose candidate sequences satisfying (P). Making these choices yields $\H|_{\ex(K_{\al})}$ with the desired property.
\end{rmk}

We conclude this section by stating these results in the language of dynamical systems. The following corollary follows from Corollary \ref{realizationCor} by appealing to the Downarowicz-Serafin realization theorem (Theorem \ref{realization}).
\begin{cor} \label{DynRealizationCor}
For every countable ordinal $\al$, there is a minimal homeomorphism $T$ of the Cantor set such that $\al$ is the order of accumulation of entropy of $T$.
\end{cor}

\section{Characterization of Orders of Accumulation on Bauer Simplices} \label{Bauer}

\begin{defn}
For any non-empty countable Polish space $E$, we define
\begin{equation*}
\rho(E) = \left \{ \begin{array}{ll}
                      |E|_{CB}-1, & \text{ if } |E|_{CB} \text{ is finite} \\
                      |E|_{CB}, & \text{ if } |E|_{CB} \text{ is infinite}
                     \end{array} \right.
\end{equation*}
For any uncountable Polish space $E$, we let $\rho(E) = \o_1$, the first uncountable ordinal.
\end{defn}

\begin{defn} \label{SKdef}
For any metrizable Choquet simplex $K$, we define
\begin{equation*}
S(K) = \{ \gamma : \text{ there exists a harmonic, u.s.c.d sequence } \H \text{ on } K \text{ with } \al_0(\H) = \gamma \}.
\end{equation*}
\end{defn}
Recall our conventions that if $\beta < \o_1$, then $[\al, \beta]$ denotes the ordinal interval $\{ \gamma : \al \leq \gamma \leq \beta \}$, and if $\beta = \o_1$, then $[\al, \beta] = \{ \gamma : \al \leq \gamma < \beta \}$. We also require the use of ``open'' or ``half-open'' intervals, which have the usual definitions.

\begin{thm} \label{bauerThm}
Let $K$ be a Bauer simplex. Then
\begin{equation*}
S(K) = [0,\rho(\ex(K))].
\end{equation*}
\end{thm}

\begin{proof}
Let $\H$ be a harmonic, u.s.c.d. candidate sequence on $K$. Proposition \ref{pointwiseBound} implies that
\begin{equation*}
\al_0(\H|_{\ex(K)}) \leq \rho(\ex(K)).
\end{equation*}
and it is always true that $\al_0(\H|_{\ex(K)}) < \o_1$.
Then since $K$ is Bauer, Proposition \ref{Uaffine} implies the same bounds for $\al_0(\H)$. It remains to show that if $\ex(K)$ is countable, then $S(K) \supset [0,\rho(\ex(K))]$, and if $\ex(K)$ is uncountable, then $S(K) \supset [0, \o_1 [$.

Suppose $E = \ex(K)$ is countable. Let $\al < |E|_{CB}$. Then by Proposition \ref{accumFacts}, there exists $x$ in $E$ such that $r(x) = \al$, which implies that $x$ is isolated in $\Gamma^{\al}(E)$. Let $U$ be a clopen neighborhood of $x$ in $E$ such that $U \cap (\Gamma^{\al}(E)\setminus \{ x \} ) = \emptyset$. Then $|U|_{CB} = \al + 1$ and $|\Gamma^{\al}(U)| = 1$. Then by the classification of countable, compact Polish spaces (Theorem \ref{classification}), there is a homeomorphism $g : \o^{\al}+1 \rightarrow U$. Let $\H'$ be the u.s.c.d candidate sequence on $\o^{\al}+1$ given by Corollary \ref{realizationCor} with $\al_0(\H') = \al$. Define $\H$ on $K$ to be harmonic extension of the embedded candidate sequence $g\H'$, which is harmonic and u.s.c.d by Lemma \ref{embedIsHarmUSC}. Since $\H|_{E \setminus g(\o^{\al}+1)} \equiv 0$, the Embedding Lemma (Lemma \ref{embeddingThm}) applies. Since $g(\o^{\al}+1)$ is a compact subset of $\ex(K)$, we obtain that $\al_0(\H) = \al_0(\H') = \al$. Since $\al < |\ex(K)|_{CB}$ was arbitrary, this argument shows that $S(K) \supset [0,|\ex(K)|_{CB}-1]$ (note that since $K$ is Bauer, $\ex(K)$ is compact and $|\ex(K)|_{CB}$ is a successor). If $|\ex(K)|_{CB}$ is infinite, then let $\al = |\ex(K)|_{CB}-1$ and repeat the above argument with $\H'$ given by Corollary \ref{alphaPlusOne} so that $\al_0(\H)=\al+1$. In this case we obtain that $S(K) \supset [0,|\ex(K)|_{CB}]$. In any case, we conclude that $S(K) \supset [0,\rho(\ex(K))]$, as desired.

Now suppose $E = \ex(K)$ is uncountable. Fix $\al < \o_1$. Let $g:\o^{\al}+1 \rightarrow E$ be the embedding given by Proposition \ref{uncountPolish}, and let $\H_{\al}$ be the u.s.c.d. candidate sequence on $\o^{\al}+1$ given by Corollary \ref{realizationCor}. Then the harmonic extension $\H$ of the embedded candidate sequence $g\H_{\al}$ on $K$ is harmonic and u.s.c.d. by Lemma \ref{embedIsHarmUSC}. Furthermore, $\H$ satisfies $\al_0(\H)=\al_0(\H_{\al}) =\al$, by the Embedding Lemma (as $g(\o^{\al}+1)$ is a compact subset of $\ex(K)$). Since $\al < \o_1$ was arbitrary, $S(K) \supset [0,\o_1 [$.

\end{proof}

\section{Orders of Accumulation on Choquet Simplices} \label{Choquet}

In this section we address the extent to which the orders of accumulation that appear on a metrizable Choquet simplex $K$ are constrained by the topological properties of the pair $(\ex(K), \overline{\ex(K)})$.

We will require a relative version of Cantor-Bendixson rank, whose definition we give here.
\begin{defn} \label{RelCBDerivative} Given a Polish space $X$ contained in the Polish space $T$, we define the sequence $\{ \Gamma_X^{\al}(T) \}$ of subsets of $T$ using transfinite induction. Let $\Gamma_X^0(T) = T$. If $\Gamma_X^{\al}(T)$ has been defined, then let $\Gamma_X^{\al+1}(T) = \{ t \in T : \exists (t_n) \in \Gamma_X^{\al}(T) \cap X \setminus \{t\} \text{ with } t_n \rightarrow t \}$. If $\Gamma_X^{\beta}(T)$ has been defined for all $\beta < \al$, where $\al$ is a limit ordinal, then we let $\Gamma_X^{\al}(T) = \cap_{\beta < \al} \Gamma_X^{\beta}(T)$.
\end{defn}
Note that $\Gamma_X^{\al}(T)$ is closed in $T$ for all $\al$, and $\Gamma_X^{\al}(T) \subset \Gamma_X^{\beta}(T)$ for $\al > \beta$. For $X$ and $T$ Polish, there exists a countable ordinal $\beta$ such that $\Gamma_X^{\al}(T) = \Gamma_X^{\beta}(T)$ for all $\al > \beta$.
\begin{defn} \label{RelCBRank}
The \textbf{Cantor-Bendixson rank of} $T$ \textbf{relative to} $X$, denoted $|T|_{CB}^X$, is the least ordinal $\beta$ such that $\Gamma_X^{\al}(T) = \Gamma_X^{\beta}(T)$ for all $\al > \beta$.
\end{defn}
If $X$ is countable, then $\Gamma_X^{\al}(T) = \emptyset$ if and only if $\al \geq |T|_{CB}^X$. If $X$ is countable and $T$ is compact, then by the finite intersection property, $|T|_{CB}^X$ is a successor ordinal.
\begin{defn} \label{ptwiseRelTopRank} For $t$ in $T$, we also define the pointwise relative topological rank $r_X(t)$ of $t$ with respect to $X$:
\begin{equation*}
r_X(t) = \left\{
\begin{array}{ll}
\sup \{ \al : t \in \Gamma_X^{\al}(T) \} & \text{if } t \notin \Gamma_X^{|T|^X_{CB}}(T) \\
\o_1 & \text{if } t \in \Gamma_X^{|T|^X_{CB}}(E).
\end{array} \right.
\end{equation*}
\end{defn}
It follows that for $X$ countable, for all $t$ in $T$, $r_X(t) \leq |X|_{CB}$, and thus $|T|_{CB}^X \leq |X|_{CB} +1$. Also, $|X|_{CB} \leq |T|_{CB}^X \leq |T|_{CB}$.

For a Polish space $T$, the usual Cantor-Bendixson rank is obtained from the relative version by taking $X = T$ in the above construction. Thus, we have $|T|_{CB}^T = |T|_{CB}$.

\subsection{Results for Choquet Simplices}

\begin{defn}
Let $X$ and $T$ be non-empty Polish spaces, with $X \subset T$. If $X$ is countable, let
\begin{equation*}
\rho_X(T) = \left \{ \begin{array}{ll}
                      |T|_{CB}^X-1, & \text{ if } |T|_{CB}^X \text{ is finite} \\
                      |T|_{CB}^X, & \text{ if } |T|_{CB}^X \text{ is infinite}
                     \end{array} \right.
\end{equation*}
If $X$ is uncountable, let $\rho_X(T) = \o_1$.
\end{defn}

Now we present bounds on the set $S(K)$ (see Definition \ref{SKdef}) for any metrizable Choquet simplex $K$. Recall our convention that for a countable ordinal $\beta$, $[0,\beta] = \{ \al : 0 \leq \al \leq \beta \}$, but for $\beta= \o_1$, $[0,\beta] = \{ \al : 0 \leq \al < \o_1\} = [0,\o_1 [.$

\begin{thm} \label{topBounds}
Let $K$ be a metrizable Choquet simplex. Then
\begin{equation*}
[0, \rho_{\ex(K)}(\overline{\ex(K)}) \, ] \; \subset S(K) \subset [0, \rho(\overline{\ex(K)}) \,].
\end{equation*}
\end{thm}
\begin{proof}
First we prove the lower bound on $S(K)$.

Suppose $\ex(K)$ is uncountable, then by Corollary \ref{uncountPolish}, for any countable $\al$, there exists a map $g : \o^{\al}+1 \rightarrow \ex(K)$, where $g$ is a homeomorphism onto its image. Let $\F$ be a u.s.c.d. sequence on $\o^{\al}+1$, and let $\H$ be the harmonic extension of the embedded sequence $g\F$ on $K$. $\H$ is a harmonic, u.s.c.d. candidate sequence on $K$ by Lemma \ref{embedIsHarmUSC}. Also, $\H|_{\ex(K) \setminus g(\o^{\al}+1)} \equiv 0$. Thus the Embedding Lemma (Lemma \ref{embeddingThm}) applies, and then since $g(\o^{\al}+1)$ is a compact subset of $\ex(K)$, we obtain that $\al_0(\H) = \al_0(\F)$. Letting $\F$ vary over all u.s.c.d. candidate sequences on $\o^{\al}+1$, it follows that $S(\M(\o^{\al}+1)) \subseteq S(K)$. By Theorem \ref{bauerThm}, $S(\M(\o^{\al}+1)) = [0, \rho(\o^{\al}+1)]$. Now $\rho(\o^{\al}+1) = \al$ if $\al$ is finite and $\rho(\o^{\al}+1) = \al+1$ if $\al$ is infinite. In either case, $\rho(\o^{\al}+1) \geq \al$. Hence $S(K) \supset [0,\al]$. Since this inclusion holds for any countable ordinal $\al$, we have that $S(K) \supset [0,\o_1 [$, as desired.

If $\ex(K)$ is countable, then $|\overline{\ex(K)}|_{CB}^{\ex(K)}$ is a successor ordinal. For each ordinal $\al < |\overline{\ex(K)}|_{CB}^{\ex(K)}$, we have $\Gamma_{\ex(K)}^{\al}(\overline{\ex(K)}) \neq \emptyset$. Fix $\al < |\overline{\ex(K)}|_{CB}^{\ex(K)}$, and let $t$ be in $\Gamma_{\ex(K)}^{\al}(\overline{\ex(K)})$. Since $t$ lies in $\Gamma_{\ex(K)}^{\al}(\overline{\ex(K)})$, there exists a map $g : \o^{\al}+1 \rightarrow K$, where $g$ is a homeomorphism onto its image, $g(\o^{\al}+1) \subset \ex(K) \cup \{t\}$ and $g(\0_{\al})=t$, where $\0_{\al}$ is the point $\o^{\al}$ in $\o^{\al}+1$. Given some real number $a>0$, let $\F = (f_k)$ be a u.s.c.d. candidate sequence on $\o^{\al}+1$ with $\al_0(\F)=\al$ and satisfying (1)-(3) of Corollary \ref{realizationCor}. Recall that $f_k(\0_{\al}) =0$ for all $k$. Then let $\H' = (h'_k)$ be the embedded candidate sequence $g \F$ on $K$.

Note that for $s$ in $K \setminus \ex(K)$, $(h'_{k+1}-h'_k)(s) = 0$. Also, for $s$ in $\ex(K)$, $(h'_{k+1}-h'_k)(s) \geq 0$. It follows that $h'_{k+1}-h'_k$ is convex on $K$.

Now let $\H = (h_k)$, where $h_k$ is the harmonic extension of $h'_k$ on $K$. By Lemma \ref{embedIsHarmUSC}, $\H$ is a u.s.c.d. candidate sequence on $K$.

Let $F = g(\o^{\al}+1) \cap \ex(K)$, and note that $\H|_{\ex(K) \setminus F} \equiv 0$. Also $\overline{F} = g(\o^{\al}+1)$ and $\H|_{\overline{F}} = \F \circ g^{-1}$. Applying the Embedding Lemma (Lemma \ref{embeddingThm}), we obtain that $\al_0(\H) \leq \al_0(\H|_{\overline{F}}) = \al_0(\F) = \al$. We now show the reverse inequality. Recall that $t = g(\0_{\al})$. For $\gamma < \al$, the Embedding Lemma (Lemma \ref{embeddingThm}) implies that $u_{\gamma}^{\H}(t) \leq ||u_{\gamma}^{\F}|| < a$ (where the strict inequality comes from Corollary \ref{realizationCor} (1)). Also, $u_{\al}^{\H}(t) \geq u_{\al}^{\F}(\0_{\al}) = a$. From these facts, we have that $\al \leq \al_0^{\H}(t) \leq \al_0(\H)$. Thus $\al_0(\H)=\al$.

Since $\al < |\overline{\ex(K)}|_{CB}^{\ex(K)}$ was arbitrary, we obtain that $S(K) \supset [0,|\overline{\ex(K)}|_{CB}^{\ex(K)} [ $. If $|\overline{\ex(K)}|_{CB}^{\ex(K)}$ is infinite, then we may let $\al =   |\overline{\ex(K)}|_{CB}^{\ex(K)}-1$ and repeat the above argument with $\F$ given by Corollary \ref{alphaPlusOne} so that $\al_0(\H)=\al+1$. Thus we have that $S(K) \supset [0, \, \rho_{\ex(K)}(\overline{\ex(K)}) \, ]$.

Here we prove the upper bound on $S(K)$. Suppose $\overline{\ex(K)}$ is uncountable. Then $\rho(\overline{\ex(K)}) = \o_1$. Since the order of accumulation of any candidate sequence on $K$ is countable, we have (trivially) that $S(K) \subset [0,\o_1)$. Now suppose $\overline{\ex(K)}$ is countable. If $\H$ is a u.s.c.d., harmonic candidate sequence on $K$, then by Corollary \ref{CBrankBound}, the restricted sequence $\H|_{\overline{\ex(K)}}$ satisfies
\begin{equation*}
\al_0(\H|_{\overline{\ex(K)}}) \leq \left\{ \begin{array}{ll}
             |\overline{\ex(K)}|_{CB}-1, & \text{ if } |\overline{\ex(K)}|_{CB} \text{ is finite} \\
             |\overline{\ex(K)}|_{CB}, & \text{ if } |\overline{\ex(K)}|_{CB} \text{ is infinite},
                                    \end{array} \right.
\end{equation*}
which is exactly the statement that $\al_0(\H|_{\overline{\ex(K)}}) \leq \rho(\overline{\ex(K)})$. Also, the Embedding Lemma (Lemma \ref{embeddingThm}) implies that $\al_0(\H) \leq \al_0(\H|_{\overline{\ex(K)}})$. This establishes the upper bound on $S(K)$.
\end{proof}

\subsection{Optimality of Results for Choquet Simplices} \label{optimality}

In this section we study the optimality of the results in Theorem \ref{topBounds}.

The following theorem answers a question of Jerome Buzzi, and answers the question of whether the bounds in Theorem \ref{topBounds} can be improved using only knowledge of the ordinals $\rho_{\ex(K)}(\overline{\ex(K)})$ and $\rho(\overline{\ex(K)})$.

\begin{thm} \label{threeAlphas}
Let $\al_1 \leq \al_2 \leq \al_3$ be ordinals such that $\al_1$ and $\al_2$ are countable successors and $\al_3$ is either a countable successor ordinal or $\o_1$. Then there exists a metrizable Choquet simplex $K$ such that $\rho_{\ex(K)}(\overline{\ex(K)}) = \al_1$, $S(K)=[0,\al_2]$, and $\rho(\overline{\ex(K)}) = \al_3$.
\end{thm}

We postpone the proof of Theorem \ref{threeAlphas} until after the proof of Theorem \ref{optimal}. The proofs of these theorems are very similar and we prefer not to repeat the arguments unnecessarily.

Now we address the following question: can the bounds in Theorem \ref{topBounds} be improved with knowledge of the homeomorphism class of the compactification $(\ex(K),\overline{\ex(K)})$? We will need some definitions.

\begin{defn}[\cite{Dug}]
If $E$ is a topological space, then a compactification of $E$ is a pair $(\overline{E},g)$, where $\overline{E}$ is a compact, Hausdorff space and $g$ is a homeomorphism of $E$ onto a dense subset of $\overline{E}$.
\end{defn}
If $E$ is a topological space and $(\overline{E},g)$ is a compactification of $E$, then we may identify $E$ with $g(E)$ and assume that $E$ is a subset of $\overline{E}$. In such instances, we may refer to $\overline{E}$ as a compactification of $E$, or we may refer to the pair $(E, \overline{E})$ as a compactification. 

Consider compactifications $(E, \overline{E})$, where $E$ is a topological space and $\overline{E}$ is a compactification of $E$. Suppose there are two such compactifications, $(E_1, \overline{E_1})$ and $(E_2, \overline{E_2})$. We say that the compactifications are homeomorphic, written $(E_1, \overline{E_1}) \simeq (E_2, \overline{E_2})$, if there is a homeomorphism $g : \overline{E_1} \rightarrow \overline{E_2}$ such that $g(E_1)=E_2$. Recall that Theorem \ref{choquetWeak} may be strengthened as follows.
\begin{thm}[Choquet \cite{Cho}]  \label{ChoquetStrong}
Let $E$ be a topological space and $\overline{E}$ a metrizable compactification of $E$. Then there exists a metrizable Choquet simplex $K$ such that $(\ex(K), \overline{\ex(K)}) \simeq (E, \overline{E})$ if and only if $E$ is Polish.
\end{thm}
Given a Polish space $E$ and a compactification $\overline{E}$, the proof of Theorem \ref{optimal} below involves constructing a metrizable Choquet simplex $K$ such that $(\ex(K), \overline{\ex(K)}) \simeq (E,\overline{E})$ while simultaneously controlling the possible harmonic, u.s.c.d. candidate sequences on $K$. In this sense Theorem \ref{optimal} may be viewed as a partial generalization of Theorem \ref{ChoquetStrong}.

\begin{rmk}
In Theorem \ref{optimal}, we restrict our attention to metrizable compactifications of Polish spaces. Since we are only interested in studying pairs $(\ex(K),\overline{\ex(K)})$ where $K$ is a metrizable Choquet simplex, Theorem \ref{ChoquetStrong} implies that there is no loss of generality in making this restriction.
\end{rmk}

\begin{thm} \label{optimal}
Let $E$ be a non-compact, countably infinite Polish space, and let $\overline{E}$ be a metrizable compactification of $E$.

\begin{enumerate}
\item If $\overline{E}$ is countable, then for each successor $\beta \in [\rho_E(\overline{E}), \medspace \rho(\overline{E})]$, there exists a Choquet simplex $K$ such that $(\ex(K), \overline{\ex(K)}) \simeq (E, \overline{E})$ and $S(K) = [0, \beta]$.
\vspace{1mm}
\item If $E$ is countable and $\overline{E}$ is uncountable, then for each countable ordinal $\beta \geq \rho_E(\overline{E})$, there exists a Choquet simplex $K$ such that $(\ex(K), \overline{\ex(K)}) \simeq (E, \overline{E})$ and $S(K) \supset [0,\beta]$.
\end{enumerate}
\end{thm}

Observe that when $E$ is uncountable, Theorem \ref{topBounds} gives that for any metrizable Choquet simplex $K$ with $\ex(K)$ homeomorphic to $E$, $S(K) = [0,\o_1[$. The proofs of Theorem \ref{optimal} (1) and (2) rely very heavily Lemma \ref{ConstructKLem}, which in turn relies very heavily on Haydon's proof (see \cite{Hay} or \cite[pp. 126-129]{AE}) of Theorem \ref{choquetWeak}. 

\vspace{2 mm}

\noindent \textbf{Proof of Theorem \ref{optimal} (1)}.

\subsubsection{Setup for proof of Theorem \ref{optimal} (1)}

Let $\beta$ be a successor ordinal with $\rho_{E}(\overline{E}) \leq \beta \leq \rho(\overline{E})$. Let $\beta_0 = \beta$ if $\beta$ is finite, and let $\beta_0 = \beta-1$ if $\beta$ is infinite. For notation, we let $T = \overline{E}$ and $X = E$. Since $T$ is countable and compact, $T \cong \o^{|T|_{CB}-1} n + 1$ for some natural number $n$ (by Theorem \ref{classification}). We may assume without loss of generality that $n=1$ (if $n >1$, then $T$ is just the finite disjoint union of the case when $n=1$, and we may repeat the following constructions independently $n$ times). Using this homeomorphism of $T$ and $\o^{|T|_{CB}-1} + 1$, we obtain a well-ordering on $T$ such that the induced order topology coincides with the original topology on $T$. Thus we may assume without loss of generality that $T =\o^{|T|_{CB}-1} + 1$. Also, we fix a complete metric $d( \cdot, \cdot )$ on $T$.

Let $Y \subset T$ be the set $\o^{\beta_0}+1$ in $T = \o^{|T|_{CB}-1} + 1$. Let $Z = Y \setminus X$, which may be empty. There are two cases: either $Y = T$ or $Y \subsetneq T$. The case $Y=T$ occurs if and only if $\beta = \rho(T)$, while the case $Y \subsetneq T$ occurs if and only if $\beta < \rho(T)$. If $Y = T$, then one may ignore the constructions in Sections \ref{YnoteqTone}, \ref{YnoteqTtwo}, and \ref{YnoteqTthree}. If $Y \subsetneq T$, then $Z$ may be empty. If $Z$ is empty, then one may ignore the construction in Section \ref{ZnotEmpty}. We make the convention that an empty sum is zero.

\subsubsection{Definition of the points $z_m, u_m, v_m$} \label{ZnotEmpty}

Assuming $Z$ is not empty, we will define distinct points $z_m \in Z$ and $u_m, v_m \in X$. In the simplex $K$, they will satisfy $z_m = \frac{1}{2}(u_m+v_m)$, and it is exactly this formula which allows us to prove that $[0,\beta] \subseteq S(K)$.

Since $T$ is countable, $Z$ is countable, and we may enumerate $Z = \{ z_m \}$ (in the case when $Z$ is finite, this sequence is finite). If $z_m < \o^{|T|_{CB}-1}$ in $T$, then let $u_m = z_m+1$ and $v_m = z_m+2$ (successor ordinals). If $z_m = \o^{|T|_{CB}-1}$ in $T$, we let $u_m = 1$ and $v_m = 2$. Since $X$ is dense in $T$, any isolated point in $T$ must lie in $X$. Therefore any successor ordinals in $T$ must be in $X$. It follows that $u_m, v_m$ are points in $X$.

\subsubsection{Construction of the sets $V_k$} \label{YnoteqTone}

Here we will use notations defined previously, such as the relative topological rank, $r_X(x)$, of the point $x$ (Definition \ref{ptwiseRelTopRank}) and the relative Cantor-Bendixson derivatives $\Gamma_X^{\al}(Y)$ (Definition \ref{RelCBDerivative}). Also, since it is an important hypothesis in this section, we remind the reader that $Y$ is clopen in $T$.

In this section we assume that $T \setminus Y$ is not empty, which occurs exactly when $\beta < \rho(T)$, and we define certain sets $V_k$. The construction of the sets $V_k$ and the points $x_k$ and $y_k$ (see section \ref{YnoteqTtwo}) allows one to prove that $S(K) \subseteq [0,\beta]$. In the simplex $K$, all points in the set $V_k$ will lie in the convex hull of $x_k$ and $y_k$, which will imply that the order of accumulation cannot be increased by the points in $V_k \setminus \{ y_k \}$ (see Lemmas \ref{U_tLemma} and \ref{pointwiseOrderBd}).

Below, by an interval in a subset $A$ of $T$, we mean the intersection of an interval of $T$ (which may be a singleton) with $A$.

\begin{lem} \label{VkLemma}
If $T \setminus Y$ is not empty, then there exists a collection $\{V_k\}$ of non-empty subsets of $T$ with the following properties:
\begin{enumerate}
 \item if $V_k \cap V_j \neq \emptyset$, then $k = j$;
 \item for each $V_k$ there exists an ordinal $\al_k \geq 1$ such that $r_X(t) = \al_k$ for all $t$ in $V_k$;
 \item each $V_k$ is a clopen interval in $\Gamma_X^{\al_k}(T)$;
 \item if $V_k \cap X \neq \emptyset$, then $V_k \cap X = \{ \sup (V_k) \}$;
 \item $\Gamma_X^{1}(T) \setminus Y = \cup_k V_k$.
 \item $\lim_k \diam(V_k) = 0$.
\end{enumerate}
\end{lem}
\begin{proof}
Suppose  $\al \in [1, \rho(T)]$ and the set $A_{\al} = \{ t \in \Gamma_X^{\al}(T) \setminus Y : r_X(t) = \al \}$ is non-empty (which it must be for $\al = 1$ since $Y \neq T$). For $x \in X \cap A_{\al}$, let $a(x) = \min \{ a \in \Gamma_X^1(T) \setminus Y : [a,x] \cap (X \cap A_{\al}) = \{x\} \text{ and } [a,x]\cap \Gamma_X^{\al+1}(T) = \emptyset \}$. Let $U_x = [a(x),x] \cap \Gamma_X^{\al}(T)$ and note that $U_x \subset A_{\al}$. The set $\Gamma_X^{\al+1}(T)$ is closed and does not intersect $A_{\al}$, and the set $X \cap A_{\al}$ has no accumulation points in $A_{\al}$. Thus each $U_x$ is a clopen interval in $\Gamma_X^{\al}(T)$. Now let $U_{\o}^{\al} = A_{\al} \setminus \cup_{x \in X \cap A_{\al}} U_x$, which may be empty.

If $U_{\o}^{\al}$ is non-empty, then we claim that it is also a clopen interval in $\Gamma_X^{\al}(T)$. Let $y_0 = \sup(X \cap A_{\al})$. Note that $Y$ is an initial subinterval of $T$ and $X \cap A_{\al} \subset T \setminus Y$, which implies that $[y_0, \max(T)] \subset T \setminus Y$. We also have that $y_0$ is in $X \cup \Gamma_X^{\al+1}(T)$, which implies that $y_0$ is not in $U_{\o}^{\al}$. We will show that $U_{\o}^{\al} = [y_0+1, \max(T)] \cap \Gamma_X^{\al}(T)$. To see this fact, first note that if $y \leq x$ with $y \in A_{\al}$ and $x \in X \cap A_{\al}$, then $y \in \cup_{x \in X \cap A_{\al}} U_x$. Thus we have that $U_{\o}^{\al} \subset [y_0+1, \max(T)] \cap \Gamma_X^{\al}(T)$. To show the reverse inclusion, we show that $[y_0+1, \max(T)] \cap A_{\al} = [y_0+1, \max(T)] \cap \Gamma_X^{\al}(T)$. We assume for the sake of contradiction that there is a point $t$ in $[y_0+1, \max(T)] \cap \Gamma_X^{\al+1}(T)$. From this assumption and the fact that $[y_0+1, \max(T)]$ is open it follows that $[y_0+1, \max(T)] \cap \Gamma_X^{\al}(T) \cap X$ has $t$ as an accumulation point (and so, in particular, this set is non-empty). If $[y_0+1, \max(T)] \cap \Gamma_X^{\al}(T) \cap X$ contains a single point $s$ with $r_X(s)= \al$, then we see that $s \in X \cap A_{\al}$ and $s > y_0$, which contradicts the definition of $y_0$. Now suppose that for all $s$ in $[y_0+1, \max(T)] \cap \Gamma_X^{\al}(T) \cap X$, $r_X(s) > \al$. Then $[y_0+1, \max(T)] \cap \Gamma_X^{\al}(T) \cap X$ is a non-empty, countable, metrizable space with no isolated points, which implies that it is not Polish. But $[y_0+1, \max(T)] \cap \Gamma_X^{\al}(T)$ is closed in $T$, which implies that it is a $G_{\delta}$ in $T$, and $X$ is Polish in $T$, which implies it is a $G_{\delta}$ in $T$, and the intersection of two $G_{\delta}$ sets is a $G_{\delta}$. Also, any $G_{\delta}$ set in a Polish space is Polish. Thus, $[y_0+1, \max(T)] \cap \Gamma_X^{\al}(T) \cap X$ is Polish, and we arrive at a contradiction.

Let $\{V'_k\}$ be an enumeration of all the non-empty sets $U_x$ and $U_{\o}^{\al}$ constructed above, for any $\al \in [1, \rho(T)]$. The collection $\{ V'_k \}$ satisfies properties (1)-(5) but not necessarily (6). However, given $V'_k$ a clopen interval in $\Gamma_X^{\al}(T)$ contained in $A_{\al}$, we may find a finite collection of pairwise disjoint clopen intervals (in $\Gamma_X^{\al}(T)$)  $V'_{k,i}$, contained in $A_{\al}$, whose union is $V'_k$, such that each $V'_{k,i}$ has diameter at most $\frac{1}{k}$. Re-enumerating the collection $\{V'_{k,i} \}$, we obtain the required collection $\{V_k\}$.
\end{proof}

Note that since $T \setminus X \subset \Gamma_X^{1}(T)$, we have that $T \setminus (X \cup Y) = \sqcup_k V_k \setminus X$.

\subsubsection{Definition of the points $x_k$ and $y_k$} \label{YnoteqTtwo}

The points $x_k$ and $y_k$ are part of the construction that allows one to bound the possible orders of accumulation on $K$ from above.

Assuming $\beta < \rho(T)$, we let $\{V_k\}$ be a collection of non-empty subsets of $T$ given by Lemma \ref{VkLemma}, and fix a natural number $k$. There are two cases: either $V_k \cap X = \emptyset$ or $V_k \cap X \neq \emptyset$. Suppose $V_k \cap X = \emptyset$. Then choose a point $t_k$ in $V_k$. If $t_k = \sup(T)$, let $x_k=\o^{\beta_0}+3$ and $y_k = \o^{\beta_0}+4$, and otherwise let $x_k = t_k +1$ and $y_k = t_k+2$. If $V_k \cap X \neq \emptyset$, then let $y_k = \sup(V_k)$ (which is in $X$ by conclusion (4) of Lemma \ref{VkLemma}). If $y_k = \sup(T)$, let $x_k = \o^{\beta_0}+5$ and otherwise let $x_k = y_k+1$. The fact that the $V_k$ are pairwise disjoint implies that the points $x_k$ and $y_k$ are all distinct. Note that for all $k$, $x_k$ and $y_k$ are in $X$.

Notice that the points $x_k, y_k, z_m, u_m$, and $v_m$ and the sets $V_k$ have been chosen so that (i) the quantities $\diam(V_k)$, $\max_{t \in V_k} \dist(x_k , t)$, and $\max_{t \in V_k} \dist(y_k, t)$ each converge to zero as $k$ tends to infinity, (ii) $d(z_m, u_m)$ and $d(z_m, v_m)$ each converge to zero as $m$ tends to infinity, (iii) the points $x_k, y_k, z_m, u_m$, and $v_m$ are all distinct, (iv) the points $x_k, y_k, u_m, v_m$ are all in $X$, and (v) if $V_k \cap X \neq \emptyset$, then $V_k \cap X = \{y_k\}$ (v) for all $m$ and $k$, $z_m \notin V_k$, and (vi) the sets $V_k$ are pairwise disjoint.

\subsubsection{Definition of $F_k$ and $G_k$} \label{YnoteqTthree}

Choose Borel measurable functions $F_k: T \rightarrow [0,1]$ and $G_k : T \rightarrow [0,1]$ with the following properties:
\begin{enumerate}
 \item $F_k, G_k < 1$ on $T \setminus X$;
 \item $F_k$ and $G_k$ are continuous and injective on $V_k$ and $0$ on $T \setminus V_k$;
 \item $F_k+G_k = \chi_{V_k}$;
 \item $F_k(y_k) = 0$ and $G_k(y_k)=1$.
\end{enumerate}
The existence of such maps follows easily from the fact that $T$ can be order-embedded in $(0,1)$ and $V_k$ is closed.

\subsubsection{Conclusion of the proof of Theorem \ref{optimal} (1)}
Let $C_n=(\cup_{k=1}^n V_k) \cup \{z_1, \dots, z_n\}$ for each $n$. Consider the collection of points $\{x_k\}\cup\{y_k\}\cup\{u_m\}\cup\{v_m\}$. To each point $x_k$ we associate the function $F_k$. To each point $y_k$ we associate the function $G_k$. To each point $u_m$ or $v_m$, we associate the function $\frac{1}{2}\chi_{z_m}$. Then the hypotheses in Lemma \ref{ConstructKLem} are satisfied by the countable collection of closed sets $\{C_n\}\cup\{D_n\}$, the countable collection of points $\{x_k\}\cup\{y_k\}\cup\{u_m\}\cup\{v_m\}$ in $X$, and the associated functions $\{F_k\}\cup\{G_k\}\cup\{\frac{1}{2}\chi_{z_m}\}$. Lemma \ref{ConstructKLem} gives a metrizable Choquet simplex $K$  and a homeomorphism $\phi : T \to \overline{\ex(K)}$ such that $\phi(X)=\ex(K)$ and such that for all $t$ in $T \setminus X$,
\begin{equation} \label{phitEqnOne}
\phi(t) = \sum_k F_k(t) \phi(x_k) + G_k(t) \phi(y_k) + \frac{1}{2} \sum_m \chi_{z_m}(t) (\phi(u_m)+\phi(v_m)).
\end{equation}
\begin{lem}\label{U_tLemma}
Let $X,Y,T$, and $K$ be as above. Then for every $t \in T \setminus Y$, there exists an open (in $T$) neighborhood $U_t$ and points $x_t$ and $y_t$ in $X \setminus Y$ such that for all $s$ in $U_t$, either $r_X(s) < r_X(t)$ or else $r_X(s)=r_X(t)$ and $\phi(s) = a_s \phi(x_t) + b_s \phi(y_t)$ in $K$, with $0 \leq a_s, b_s \leq 1$ and $a_s + b_s = 1$.
\end{lem}
\begin{proof}
Let $t \in T \setminus Y$. If $r_X(t)=0$, then $t$ is isolated in $T$ and $t$ is in $X$, since $X$ is dense in $T$. In this case we may choose $U_t = \{t\}$ and the requirement is trivially satisfied.

If $r_X(t) \geq 1$, then $t$ is in $V_k$ for some $k$. Let $U_t$ be any open (in $T$) neighborhood of $t$ with $\Gamma_X^{r_X(t)} \cap U_t \subseteq V_k$ (such a neighborhood exists since $V_k$ is an open interval in in $\Gamma_X^{r_X(t)}(T)$), and let $x_t = x_k$ and $y_t = y_k$. We have that for each $s$ in $U_t$, either $r_X(s) < r_X(t)$ or $s$ is in $V_k$. If $s$ is in $V_k$, then $r_X(s) = r_X(t)$, and it follows from Equation (\ref{phitEqnOne}) that $\phi(s) = F_k(s) \phi(x_k) + G_k(s) \phi(y_k)$ in $K$. Also, we have that $F_k(s) + G_k(s) = 1$.
\end{proof}

By Lemmas \ref{pointwiseOrderBd} and \ref{SKsubset}, we have that $S(K) \subset [0,\rho(Y)]$. By Lemma \ref{inclusionLemma}, $S(K) \supset [0, \rho(Y)]$. Thus $S(K)=[0, \rho(Y)] = [0,\beta]$.

\smallskip
\noindent \textit{This concludes the proof of Theorem \ref{optimal} (1)}.

\vspace{1 mm}

\noindent \textbf{Proof of Theorem \ref{optimal} (2)}.
\subsubsection{Setup for proof of Theorem \ref{optimal} (2)}

Let $\beta$ be a successor ordinal with $\beta \geq \rho_E(\overline{E})$. Let $\beta_0 = \beta$ if $\beta$ is finite and let $\beta_0 = \beta-1$ if $\beta$ is infinite. For notation, we let $T = \overline{E}$ and $X = E$. Fix a metric $d$ on $T$ that is compatible with the topology of $T$. Since $T$ is uncountable and compact, $T$ contains an uncountable perfect set $P$. Since $\beta_0$ is countable, $P$ contains a set $Y$ that is homeomorphic to $\o^{\beta_0}+1$. Let $\{a_{\al}\}_{\al=0}^{\o^{\beta_0}}$ be a transfinite sequence of real numbers $a_{\al}$ such that $0 < a_{\al} \leq 1$ and $\sum_{\al \leq \o^{\beta_0}} a_{\al} = 1$ (such a sequence exists since $\o^{\beta_0}$ is countable). Let $Z = Y \setminus X$, and choose an enumeration of $Z = \{z_m\}$. Note that $Z$ may be empty or finite. In the construction to follow, if $Z$ is empty then we do not choose points $u_m$ and $v_m$, and any summation over the index $m$ will be zero by convention.

Let $X_0 = X \sqcup Z = X \cup Y$. Recall that since $X$ is a completely metrizable subset of the the compact metrizable space $T$, $X$ is a $G_{\delta}$ in $T$ (see, for example, \cite{S}). $Y$ is a $G_{\delta}$ in $T$ because it is compact. Therefore $X_0$ is a $G_{\delta}$ in $T$, since it is the union of two $G_{\delta}$ sets in $T$. Thus we may let $X_0 = \cap_{n \in \N} G_n$, where $G_n$ is open, $G_1 = T$, and $G_{n+1} \subset G_n$. Let $F_n = T \setminus G_n$, which is compact. Fix $n$. Choose a sequence $\epsilon_{\ell}$ strictly decreasing to $0$. Let $D_{\epsilon}(F_n) = \{ t \in T : \dist(t, F_n) \geq \epsilon \}$, which is compact for any $\epsilon$. Then for each $\ell$ there exists a countable collection of open sets $\{U_{\ell}^{j}\}_{j=1}^{\infty}$ such that
\begin{itemize}
 \item $D_{\epsilon_{\ell}}(F_n) \subset \cup_j U_{\ell}^j \subset D_{\epsilon_{\ell+1}}(F_n)$;
 \item $\diam(U_{\ell}^j) \leq 2^{-(\ell+n)}$ for all $j$;
 \item $\diam(U_{\ell}^j)$ tends to $0$ as $j$ tends to infinity;
 \item the collection $\{U_{\ell}^j\}$ separates the points in $D_{\epsilon_{\ell}}(F_n)$.
\end{itemize}
Then we may enumerate the collection of all sets $U_{\ell}^j$ to form the sequence $\{V_k^n\}_{k=1}^{\infty}$. Repeating this procedure for all $n$, we obtain a collection of open sets $V_k^n$ such that $\diam(V_k^n) \leq 2^{-n}$  and $\diam(V_k^n)$ converges to $0$ as $k$ tends to infinity with $n$ fixed. The sets $V_k^n$ also satisfy $\cup_k V_k^n=G_n$ and separate points in $G_n$, for each $n$. For each $n$ and $k$, let $g_k^n(t) = \min( \dist(t, T \setminus V_k^n), 1)$ and $f_k^n = 2^{-k} g_k^n$. Then for each $n$, $\sum_k f_k^n$ converges uniformly on $T$. Now let
\begin{equation*}
h^n_k(t) = \left\{ \begin{array}{ll}
                    0, & \text{ if } \sum_k f^n_k(t) =0 \\
                    \frac{f^n_k(t)}{\sum_k f^n_k(t)}, & \text{ if } \sum_k f^n_k(t) > 0
                   \end{array} \right.
\end{equation*}
The functions $h^n_k$ are all continuous and satisfy $h^n_k(t) > 0$ if and only if $t \in V_k^n$. Furthermore, $\sum_k h^n_k = \chi_{G_n}$. Now we let $p^n_k = h^n_k \cdot \chi_{T \setminus G_{n+1}}$ and notice that $\sum_n \sum_k p^n_k = \chi_{T \setminus (X \cup Y)}$. Also, the collection $p^n_k$ separates points in the sense that if $t \neq s$ with $t$ and $s$ in $T \setminus (X \cup Y)$, then there exists $n$ and $k$ such that $p^n_k(t) >0$ and $p^n_k(s) = 0$.

Using induction (on $m$, $n$, and $k$ simultaneously) and the fact that $X$ is dense in $T$, we choose points $u_m$, $v_m$, $x^n_k$, and $y^n_k$ in $X$ such that (i) $d(z_m, u_m) \leq a_{\al_{z_m}}$ and $d(z_m, v_m) \leq a_{\al_{z_m}}$, (ii) for each $m$, $u_m$ and $v_m$ are not accumulation points of $Y$ (which is possible since the isolated points of $Y$, corresponding to successors of $\o^{\beta_0}+1$, are dense in $Y$ and the set $X \setminus Y$ accumulates at each of the isolated points of $Y$ that is not in $X$) (iii) $x^n_k$ and $y^n_k$ are in $V^n_k$, and (iv) the union of all of these points is a disjoint union.\

\subsubsection{Conclusion of the proof of Theorem \ref{optimal} (2)}

Let $C_n = (T \setminus G_n) \cup \{z_1, \dots, z_n\}$. To each point $x^n_k$ or $y^n_k$, we associate the function $\frac{1}{2}p^n_k$. To each point $u_m$ or $v_m$, we associate the function $\frac{1}{2}\chi_{z_m}$. Then the hypotheses of Lemma \ref{ConstructKLem} are satisfied by the countable collection of closed sets $\{C_n\}$, the countable collection of points $\{x^n_k\}\cup\{y^n_k\}\cup\{u_m\}\cup\{v_m\}$ in $X$, and the associated functions$\{\frac{1}{2}p^n_k\}\cup\{\frac{1}{2}\chi_{z_m}\}$. Lemma \ref{ConstructKLem} gives a metrizable Choquet simplex $K$ and a homeomorphism $\phi : T \to \overline{\ex(K)}$ such that $\phi(X)=\ex(K)$ and such that for all $t$ in $T \setminus X$,
\begin{equation*}
 \phi(t) = \frac{1}{2} \sum_n \sum_k p^n_k(t) (\phi(x^n_k)+\phi(y^n_k)) + \frac{1}{2} \sum_m \chi_{z_m}(t)( \phi(u_m)+\phi(v_m)).
\end{equation*}
By Lemma \ref{inclusionLemma}, $S(K) \supset [0,\beta]$.

\smallskip
\noindent \textit{This concludes the proof of Theorem \ref{optimal} (2)}.
\smallskip

\noindent \textbf{Proof of Theorem \ref{threeAlphas}}.

Fix $\al_1 \leq \al_2 \leq \al_3$ as above. Let $X_1 = \o^{\al_1}+1$, and let $T_1=X_1$. If $\al_2$ is finite, let $T_2 = \o^{\al_2}+1$, and if $\al_2$ is infinite, let $T_2 = \o^{\al_2-1}+1$. In either case, let $X_2$ be all the isolated points (successors) in $T_2$. Let $S$ be a non-empty compact subset of $(0,1) \times \{0\}$ in $\R^2$, chosen so that if $\al_3$ is finite, then $\rho(S) = \al_3-1$, and otherwise $\rho(S) = \al_3$. Let $X_3$ be a bounded, countable subset of $\R^2 \setminus (\R \times \{0\})$ whose set of accumulation points is exactly $S$. Let $T_3 = X_3 \cup S$. Now we let $T = T_1 \sqcup T_2 \sqcup T_3$ and $X = X_1 \sqcup X_2 \sqcup X_3$. Below we will construct a Choquet simplex $K$ such that $(X,T) \simeq (\ex(K),\overline{\ex(K)})$. Let $Y = T_1 \sqcup T_2$, and $Z = Y \setminus X$. Note that $Z$ is actually just the set of accumulation points in $T_2$. We have
\begin{equation*}
\rho_X(T) = \rho(X_1) = \al_1, \; \rho(T) = \rho(T_3) = \al_3, \text{ and } \rho(Y) = \rho(T_2) = \al_2.
\end{equation*}
Let $Z = \{z_m\}$. If $z_m < \sup(T_2)$, choose $u_m=z_m+1$ and $v_m=z_m+2$. If $z_m = \sup(T_2)$, choose $u_m =1$ and $v_m=2$ in $T_2$. Let $x_0$ and $y_0$ be a choice of two isolated points in $X_3$. Let $F:T \rightarrow [0,1]$ be the function such that, for a point $t$ in $T$,
\begin{equation*}
F(t) = \left \{ \begin{array}{ll} s, & \text{ if } t = (s,0) \in S \\
                                  0, & \text{ otherwise }.
                \end{array}
       \right.
\end{equation*}
Let $G: T \rightarrow [0,1]$ be such that for $t$ in $T$,
\begin{equation*}
 G(t) = \left \{ \begin{array}{ll} 1-s, & \text{ if } t = (s,0) \in S \\
                                  0, & \text{ otherwise }.
                \end{array}
       \right.
\end{equation*}
Let $C_n=S\cup\{z_1, \dots, z_n\}$ for each $n$. To each point $u_m$ or $v_m$, we associate the function $\frac{1}{2}\chi_{z_m}$. To the point $x_0$, we associate the function $F$, and to the point $y_0$ we associate the function $G$. Then the hypotheses in Lemma \ref{ConstructKLem} are satisfied by the collection of closed sets $\{C_n\}$, the collection of points $\{x_0,y_0\}\cup\{u_m,v_m\}$, and the associated functions $\{F,G\}\cup\{\frac{1}{2}\chi_{z_m}\}$. Lemma \ref{ConstructKLem} gives a Choquet simplex $K$ and a homeomorphism $\phi : T \to \overline{\ex(K)}$ such that $\phi(X)=\ex(K)$ and such that for all $t$ in $T \setminus X$,
\begin{equation} \label{phitEqnTwo}
\phi(t) = F(t) \phi(x_0) + G(t) \phi(y_0) + \frac{1}{2}\sum_m \chi_{z_m}(t) \bigl(\phi(u_m) + \phi(v_m)\bigr).
\end{equation}

It follows immediately that $\rho_{\ex(K)}(\overline{\ex(K)}) = \rho_X(T) = \al_1$ and $\rho(\overline{\ex(K)}) = \rho(T) = \al_3$. We show that $S(K) = [0,\al_2]$.
\begin{lem}\label{U_tLemmaThreeAlphas}
Let $X,Y,T$, and $K$ be as above. Then for every $t \in T \setminus Y$, there exists an open (in $T$) neighborhood $U_t$ and points $x_t$ and $y_t$ in $X \setminus Y$ such that for all $s$ in $U_t$, either $r_X(s) < r_X(t)$ or else $r_X(s)=r_X(t)$ and $\phi(s) = a_s \phi(x_t) + b_s \phi(y_t)$ in $K$, with $0 \leq a_s, b_s \leq 1$ and $a_s + b_s = 1$.
\end{lem}
\begin{proof}
Let $t$ be in $T \setminus Y = T_3$. If $t$ is in $X_3$, then $t$ is isolated in $T$ and we let $U_t = \{t\}$. In this case the requirement on $U_t$ is trivially satisfied.

If $t$ is in $T_3 \setminus X_3$, then $t$ is in $S$ and $r_X(t) = 1$. Let $U_t$ be any open neighborhood of $t$ in $T_3$, and let $x_t = x_0$ and $y_t = y_0$. Let $s$ be in $U_t$. If $s$ is in $X_3$, then $r_X(s) = 0 < r_X(t)$. If $s$ is in $T_3 \setminus X_3$, then $s$ is in $S$, and then we have $r_X(s)=1$ and by Equation (\ref{phitEqnTwo}), $\phi(s) = F(s) \phi(x_0)+G(s)\phi(y_0)$, where $F(s) +G(s) = 1$.
\end{proof}

Now by Lemmas \ref{pointwiseOrderBd} and \ref{SKsubset}, we have that $S(K) \subset [0,\rho(Y)]$. By Lemma \ref{inclusionLemma} we have that $S(K) \supset [0,\rho(Y)]$. Then $S(K)=[0,\rho(Y)] = [0,\al_2]$.

\smallskip
\noindent \textit{This concludes the proof of Theorem \ref{threeAlphas}}.
\smallskip

\subsubsection{Helpful Lemmas}

Recall the following notations. Suppose $T$ is a compact, metrizable space. Let $\mathcal{SM}(T)$ denote the set of all signed, totally finite, Borel measures on $T$. Recall that $\mathcal{SM}(T) = C_{\R}(T)^*$, and therefore $\mathcal{SM}(T)$ inherits the structure of a normed topological vector space over $\R$. For $\mu$ in $\SM(T)$, let $\mu = \mu_1 - \mu_2$ be the Jordan decomposition of $\mu$. Let $|\mu| = \mu_1 + \mu_2$. The norm on $\SM(T)$ is then given by $||\mu|| = |\mu|(T)$. We will use $\mathcal{SM}(T \setminus X)$ to denote the set of measures $\mu$ in $\mathcal{SM}(T)$ such that $|\mu|(X) = 0$. We write $\mathcal{SM}_{prob}(T) = \{ \mu \in \mathcal{SM} : \mu \geq 0, \; ||\mu||=1\}$, and for any subset $\M$ of $\mathcal{SM}(T)$, $\M_1 = \{\mu \in \mathcal{SM} : ||\mu ||\leq 1\}$. Let $\epsilon_{x_k}$ be the point mass at $x_k$.
\begin{lem} \label{ConstructKLem}
Let $T$ be a compact, metric space, and let $X$ be a dense, Polish subset of $T$. Suppose $\{C_n\}$ is a countable collection of closed subsets of $T$. Suppose $\{w_k\}$ is a countable collection of distinct points in $X$, and to each point $w_k$ there is an associated Borel measurable function $H_k : T \rightarrow [0,1]$. Let $W_k = \supp(H_k)$. Furthermore, suppose the following conditions are satisfied:
\renewcommand{\labelenumi}{(\roman{enumi})}
\begin{enumerate}
 \item $C_n \subset C_{n+1}$ for all $n$, $C_0 = \emptyset$, and $\cup_n C_n \setminus X = T \setminus X$;
 \item $\sum_k H_k \leq 1$ and $\bigl(\sum_k H_k\bigr)|_{T\setminus X} \equiv 1$;
 \item for all $t$ in $T \setminus X$, $H_k(t)<1$;
 \item if $H_k(s)=H_k(t)$ for all $k$ with $t,s \in T \setminus X$, then $s=t$;
 \item for each $k$, there exists $n_k$ such that $W_k \subset C_{n_k+1} \setminus C_{n_k}$, and with this notation, $H_k$ is continuous on $C_{n_k+1}$;
 \item $\max_{t \in W_k} d(t,w_k)$ converges to $0$ as $k$ tends to infinity;
 \item if $H_k(x) >0$ for $x$ in $X$, then $x=w_k$ and $H_k(w_k)=1$.
\end{enumerate}
Let $\xi : \mathcal{SM}(T) \rightarrow \mathcal{SM}(T)$, where for $\mu$ in $\mathcal{SM}(T)$,
\begin{equation*}
\xi(\mu) = \mu -  \sum_k \bigl(\int H_k \; d\mu \bigr)\, \epsilon_{w_k}.
\end{equation*}
Let $\M = \{ \xi(\mu) : \mu \in \mathcal{SM}(T \setminus X) \}$, and let $q : \mathcal{SM}(T) \rightarrow \mathcal{SM}(T)/ \M$ be the natural quotient map. Let $\psi : T \rightarrow \mathcal{SM}_{prob}(T)$ be $\psi(t)=\epsilon_t$, and let $\phi = q \circ \psi$. Finally, let $K = q(\mathcal{SM}_{prob}(T))$. Then
\renewcommand{\labelenumi}{(\arabic{enumi})}
\begin{enumerate}
 \item $\M$ is a closed linear subspace of $\mathcal{SM}(T)$, and thus $\phi$ is continuous;
 \item $K$ is a metrizable Choquet simplex;
 \item $\phi$ is injective on $T$;
 \item $\ex(K) = \phi(X)$;
 \item for $t$ in $T \setminus X$, $\phi(t) = \sum_k H_k(t) \phi(w_k)$ in $K$.
\end{enumerate}
\end{lem}
\begin{proof}
This lemma is almost entirely a restatement of Haydon's proof (see \cite{Hay} or \cite[pp. 126-129]{AE}) of Theorem \ref{choquetWeak}. There are two differences. Firstly, we allow $H_k$ to be positive on $X$, while Haydon does not. Secondly, we claim that $\phi$ is injective on all of $T$, whereas Haydon claims injectivity of $\phi$ only on $X$. For the proofs of properties (2), (4) and (5), theses differences do not play any role, and one may repeat Haydon's proof. For this reason, we will prove only (1) and (3).

(1) Note that $\mathcal{M}$ is a linear subspace. Recall that $\mathcal{M}$ being closed in the weak* topology is equivalent to $\mathcal{M}_1$ being closed in the weak* topology (a proof of this general fact, which follows from the Banach-Dieudonn\'{e} Theorem, can be found in \cite{SW}).
Let $\sigma_i$ be a sequence of measures in $\mathcal{M}_1$. Since $|| \xi(\mu) || \geq ||\mu||$ for all $\mu$ in $\mathcal{SM}(T \setminus X)$, there exist measures $\mu_i$ in $\mathcal{SM}(T \setminus X)_1$ such that $\sigma_i = \xi(\mu_i)$. Since each $C_{n}$ is compact, each $\mathcal{SM}(C_n)_1$ is compact in the weak* topology. Therefore a diagonal argument gives a subsequence $\{\nu_i\}$ of $\{\mu_i\}$ such that there exist measures $\hat{\nu}^n \in \mathcal{SM}(C_{n+1})$ such that $\nu_i|_{C_{n+1}}$ converges to $\hat{\nu}^n$ for each $n$. (We note that there may not be a measure $\hat{\nu}$ such that $\hat{\nu}|_{C_{n+1}} = \hat{\nu}^{n}$, since $\hat{\nu}^n|_{C_n}$ may not equal $\hat{\nu}^{n-1}$.)

Let $\nu^n = \hat{\nu}^n|_{C_{n+1} \setminus X}$, and let $\charfun_A$ be the characteristic function of the set $A$. Now define
\begin{equation*}
\nu = \sum_{n} \nu^n|_{C_{n+1} \setminus C_n} = \sum_{k,n} H_k \charfun_{C_{n+1} \setminus C_n} \nu^n,
\end{equation*}
where the second equality follows from hypotheses (i) and (ii). Let $g^n_k = H_k \cdot \charfun_{C_{n+1} \setminus C_n}$, and note that by hypothesis (v), $g^n_k$ is continuous on $C_{n+1}$ for all $k$ and $n$. Then $g^n_k \nu_i$ weak* converges to $g^n_k \hat{\nu}^n$ as $i$ tends to infinity. Since $||\nu_i|| \leq 1$ and $g^n_k \nu_i$ weak* converges to $g^n_k \hat{\nu}^n$, it follows that $||\nu|| \leq 1$ and $\nu$ is in $\mathcal{SM}(T \setminus X)$. Let us show that $\xi(\nu_i)$ converges to $\xi(\nu)$ in the weak* topology. Let $f \in C_{\R}(T)$. Then for any $\mu$ in $\mathcal{SM}(T \setminus X)$ we have
\begin{equation*}
\int f d\xi(\mu)  =  \; \int f d\mu - \sum_{k} \int f(w_k) H_k \; d\mu
  =  \; \sum_{k} \int (f-f(w_k))H_k  \; d\mu
  =  \; \sum_{n,k} \lambda^n_k(\mu) ,
\end{equation*}
where
\begin{equation*}
\lambda^n_k(\mu)  = \int (f-f(w_k))g^n_k \; d\mu.
\end{equation*}
For each $k$ and $n$, we have that $(f-f(w_k))g^n_k$ is continuous on $C_{n+1}$ by hypothesis (v). Therefore, by the choice of subsequence $\nu_i$, $\lambda^n_k(\nu_i)$ converges to $\lambda^n_k(\hat{\nu}^n)$. Also, using hypothesis (vii), we have that if $H_k(x) \charfun_{C_{n+1} \setminus C_n}(x) >0$ for some $x$ in $X$, then $x = w_k$. It follows that
\begin{align*}
\lambda^n_k( \hat{\nu}^n)  = & \int (f-f(w_k))H_k \charfun_{C_{n+1}\setminus C_n} \; d\hat{\nu}^n \\
  = & \int (f-f(w_k))H_k \charfun_{C_{n+1}\setminus C_n} \; d\nu^n \\
    & \quad + (f(w_k)-f(w_k))H_k(w_k) \charfun_{C_{n+1} \setminus C_n}(w_k) \hat{\nu}^n(\{w_k\})  \\
  = & \int (f-f(w_k))H_k \charfun_{C_{n+1} \setminus C_n} \; d\nu^n \\
  = & \lambda^n_k(\nu^n) = \lambda^n_k(\nu).
\end{align*}
This calculation shows that $\lambda^n_k(\nu_i)$ converges to $\lambda^n_k(\nu)$. For fixed $f$ in $C_{\R}(T)$ and $\epsilon >0 $, there exists a $\delta > 0$ such that $|f(t) - f(s)| < \epsilon$ whenever $d(t,s) < \delta$, by uniform continuity. Then since $\max_{t \in W_k} d(w_k,t)$ tends to zero as $k$ tends to infinity, there exists $k_0$ such that for $k \geq k_0$ and $z \in W_k$, $|f(z) - f(w_k)| < \epsilon$. Then for any $\mu$ in $\mathcal{SM}(T \setminus X)$, and $K \geq k_0$ and $N$,
\begin{align*}
\Bigl|\int f d\xi(\mu) -  \sum_{n=1}^N \sum_{k=1}^{K} \lambda^n_k(\mu)   \Bigr| &   =   \; \Bigl|\sum_{n > N} \sum_{k > K} \lambda^n_k(\mu) \Bigr| \\
 & \leq  \;  \sum_{n > N} \sum_{k > K} \int |f-f(w_k)|g^n_k \; d|\mu|
 \leq  \;  \epsilon ||\mu||,
\end{align*}
which implies that $\sum_{n=1}^N \sum_{k=1}^{K} \lambda^n_k(\mu)$ converges uniformly on $\mathcal{SM}(T \setminus X)_1$ to $\int f \; d\xi(\mu)$. Using this uniform convergence and the fact that $\lambda^n_k(\nu_i)$ converges to $\lambda^n_k(\nu)$, we conclude that $\xi(\nu_i)$ converges to $\xi(\nu)$.

(3) Suppose that $\phi(t) = \phi(s)$, or equivalently, $\epsilon_t - \epsilon_s$ is in $\mathcal{M}$. Thus there exists a measure $\mu$ in $\mathcal{M}(T \setminus X)$ such that $\epsilon_t - \epsilon_s = \xi(\mu)$. We consider three cases.

If $t$ and $s$ are both in $X$, then we notice that $\xi(\mu)=\epsilon_t-\epsilon_s$ has no mass in $T \setminus X$. As $w_k$ are all in $X$, it follows from the definition of $\xi(\mu)$ that we must have $|\mu|(T \setminus X) = 0$, which implies that $\mu$ is the zero measure. Then $\xi(\mu)$ is the zero measure, and we have that $\epsilon_t = \epsilon_s$, which means that $t = s$.

If exactly one of $t$ and $s$ is in $X$, then we may assume without loss of generality that $t \in X$ and $s \in T \setminus X$. In this case, we notice that $- \epsilon_s = (\epsilon_t - \epsilon_s)|_{T \setminus X} = \xi(\mu)|_{T \setminus X} = \mu|_{T \setminus X} = \mu$. Therefore we conclude that
\begin{equation}
\epsilon_t  = \xi(\mu) + \epsilon_s
  = \xi(\mu) - \mu
  =  \sum_{k} H_k(s) \cdot \epsilon_{w_k} . \label{epsilont}
\end{equation}
From this equation, we deduce that $t =  w_k$ for some $k$. Then $H_k(s) = 1$, which gives a contradiction since $H_k < 1$ on $T \setminus X$ by hypothesis (iii). 

If $t$ and $s$ are both in $T \setminus X$, then we see that $\xi(\mu) = \epsilon_t - \epsilon_s = \xi(\mu)|_{T\setminus X} = \mu$, which implies that $\int H_k d\mu = 0$ for all $k$. Hence $H_k(t) =H_k(s)$ for all $k$. By hypothesis (iv), we obtain that $t=s$.
\end{proof}

\begin{lem}\label{pointwiseOrderBd}
Let $K$ be a metrizable Choquet simplex. Let $X$ be a Polish subspace of a compact metrizable space $T$, and let $Y$ be clopen in $T$.  Let $\phi : T \rightarrow \overline{\ex(K)}$ be a homeomorphism with $\phi(X)=\ex(K)$. Suppose that for every point $t$ in $T \setminus Y$, there exists an open (in $T$) neighborhood $U_t$ and points $x_t$ and $y_t$ in $X \setminus Y$ such that for all $s$ in $U_t$, either $r_X(s) < r_X(t)$ or else $r_X(s)=r_X(t)$ and $\phi(s) = a_s \phi(x_t) + b_s \phi(y_t)$ in $K$, with $0 \leq a_s, b_s \leq 1$ and $a_s + b_s = 1$. Then for each point $t$ in $T \setminus Y$, and any harmonic, u.s.c.d. candidate sequence $\H$ on $K$,
\begin{equation} \label{pointwiseBdEqn}
 \al_0^{\H|_{\phi(T)}}(t) \leq \left\{
\begin{array}{ll}
 r_X(t) & \text{ if } r_X(t) \text{ is finite} \\
 r_X(t) +1 & \text{ if } r_X(t) \text{ is infinite}.
\end{array} \right.
\end{equation}
\end{lem}
\begin{proof} For the sake of notation, we identify $X,Y$, and $T$ with their images under $\phi$. Observe that $T \setminus Y$ is clopen in $T$. Thus, for every $t$ in $T \setminus Y$, $u_{\beta}^{\H|_T}(t) = u_{\beta}^{\H|_{T\setminus Y}}(t)$ for all ordinals $\beta$, which implies $\al_0^{\H|_T}(t) = \al_0^{\H|_{T \setminus Y}}(t)$. For the sake of notation, we assume that $\H$ is defined only on $T \setminus Y$ and $u_{\beta}^{\H} = u_{\beta}$.

Now we prove the lemma by transfinite induction on $\al = r_X(t)$. For $\al = 0$, we have that $r_X(t) = 0$, and thus $t$ is isolated in $T$. Then $\al_0^{\H}(t) = 0$.

Suppose the lemma holds for all $t$ in $T \setminus Y$ such that $r_X(t) < \al$. If $\al$ is finite, let $\delta = \al$. If $\al$ is infinite, let $\delta = \al+1$. We now prove that for all $t$ in $T \setminus Y$ with $r_X(t)=\al$ and all $\gamma \geq \delta$, $u_{\gamma}(t)=u_{\delta}(t)$. The proof of this statement is by transfinite induction on $\gamma$.

Let $\gamma > \delta$ be a successor ordinal, and let $t$ be in $T \setminus Y$ with $r_X(t) = \al$. Let $U_t$ be an open neighborhood of $t$, and let $x_t$ and $y_t$ be corresponding to $U_t$ according the hypotheses. Fix $\epsilon > 0$. Choose $k_0$ such that $\max (\tau_k(x_t), \tau_k(y_t)) \leq \epsilon$ for all $k \geq k_0$. Then if $(s_n)$ is a sequence in $U_t$ with $r_X(s_n) < r_X(t)$ for all $n$, then using the inductive hypotheses, we get
\begin{equation*}
(u_{\gamma-1}+\tau_k)(s_n)  = (u_{\delta-1}+\tau_k)(s_n).
\end{equation*}
If $(s_n)$ is a sequence in $U_t$ with $r_X(s_n) \geq r_X(t)$, then by the hypotheses, we have that $r_X(s_n) = r_X(t) = \al$ and $s_n = a_{s_n} x_t +b_{s_n} y_t$. Then by the induction hypothesis on $\gamma$ and the harmonicity of $\tau_k$, we have that
\begin{equation*}
(u_{\gamma-1}+\tau_k)(s_n) = u_{\delta}(s_n) + a_{s_n} \tau_k(x_t) + b_{s_n} \tau_k(y_t) \leq \epsilon.
\end{equation*}
Thus we may conclude that
\begin{equation*}
\widetilde{(u_{\gamma-1}+\tau_k)}(t) \leq \max \bigl( \widetilde{(u_{\delta-1}+\tau_k)}(t), u_{\delta}(t) + \epsilon \bigr).
\end{equation*}
Letting $k$ tend to infinity, we obtain that $u_{\gamma}(t) \leq u_{\delta}(t) + \epsilon$. Since $\epsilon$ was arbitrary, we have that $u_{\gamma}(t) = u_{\delta}(t)$.

Now let $\gamma > \delta$ be a limit ordinal, and let $t$ be in $T \setminus Y$ with $r_X(t) = \al$. Fix $U_t$, $x_t$, and $y_t$ as in the hypotheses. Note that by the induction hypotheses, if $s$ is in $U_t$, then $u_{\beta}(s) = u_{\delta}(s)$ for all $\beta < \gamma$. Then $\sup_{\beta < \gamma} u_{\beta}(s) = u_{\delta}(s)$ for all $s$ in $U_t$. Taking upper semi-continuous envelope at $t$, we have that $u_{\gamma}(t) = u_{\delta}(t)$.

We conclude that for all $t$ in $T \setminus Y$ with $r_X(t) = \al$, $\al_0^{\H}(t) \leq \delta$, as desired.
\end{proof}

\begin{lem}\label{SKsubset}
Let $K$ be a metrizable Choquet simplex. Let $X$ be a Polish subspace of a compact metrizable space $T$, and let $Y$ be clopen in $T$.  Let $\phi : T \rightarrow \overline{\ex(K)}$ be a homeomorphism with $\phi(X)=\ex(K)$. Suppose that for each point $t$ in $T \setminus Y$ and any harmonic, u.s.c.d. candidate sequence $\H$ on $K$, Equation (\ref{pointwiseBdEqn}) holds. Further, suppose that $\rho_X(T) \leq \rho(Y)$. Then $S(K) \subset [0,\rho(Y)]$
\end{lem}
\begin{proof}
Let $\H$ be a harmonic, u.s.c.d. candidate sequence on $K$. For $t$ in $Y$, we have that $\al_0^{\H|_T}(t) = \al_0^{\H|_Y}(t)$ since $Y$ is open in $T$. By Remark \ref{accumVarPrin}, $\al_0^{\H|_Y}(t) \leq \al_0(\H|_Y)$. By Proposition \ref{pointwiseBound}, $\al_0(\H|_Y) \leq \rho(Y) $. Putting these facts together, we obtain $\al_0^{\H|_T}(t) \leq \rho(Y)$ for all $t$ in $Y$.

For $t$ in $T \setminus Y$, Equation (\ref{pointwiseBdEqn}) gives that if $r_X(t)$ is finite, then $\al_0^{\H|_T}(t) \leq r_X(t)$, and if $r_X(t)$ is infinite, then $\al_0^{\H|_T}(t) \leq r_X(t)+1$. Since $X$ is countable and $T$ is compact, $|T|_{CB}^X$ is a successor, and we have $r_X(t)  \leq |T|_{CB}^X -1$. If $|T|_{CB}^X$ is finite, then for all $t$ in $T \setminus Y$ we have $\al_0^{\H|_T}(t) \leq r_X(t) \leq |T|_{CB}^X-1 = \rho_X(T)$. If $|T|_{CB}^X$ is infinite, then for all $t$ in $T \setminus Y$ we have $\al_0^{\H|_T}(t) \leq r_X(t)+1 \leq |T|_{CB}^X = \rho_X(T) \leq \rho(Y)$.

We have shown that for all $t$ in $T$, $\al_0^{\H|_T}(t) \leq \rho(Y)$. Taking supremum over all $t$ in $T$, we have that $\al_0(\H|_T) \leq \rho(Y)$. Now using the Embedding Lemma (Lemma \ref{embeddingThm}), we get that $\al_0(\H) \leq \al_0(\H|_T) \leq \rho(Y)$. Hence $S(K) \subset [0,\rho(Y)]$.
\end{proof}

\begin{lem} \label{inclusionLemma}
Let $K$ be a metrizable Choquet simplex. Let $X$ be a Polish subspace of a compact metric space $T$, and let $Y$ be a subset of $T$ with $Y \cong \o^{\beta_0}+1$, where $\beta_0$ is a countable ordinal.  Let $\phi : T \rightarrow \overline{\ex(K)}$ be a homeomorphism with $\phi(X)=\ex(K)$. Let $Y \setminus X = \{z_m\}$. Suppose there is countable collection of distinct points $W = \{u_m\}\cup\{v_m\}$ in $X$ such that each point $w$ in $W$ is isolated in $Y \cup W$ and for each $z_m$ in $Y \setminus X$, $\phi(z_m)=\frac{1}{2}(\phi(u_m)+\phi(v_m))$. Further suppose that $d(u_m,z_m)$ and $d(v_m,z_m)$ both tend to $0$ as $m$ tends to infinity. Then $S(K) \supset [0,\rho(Y)]$.
\end{lem}
\begin{proof}
For the sake of notation, we identify $X,Y,W$ and $T$ with their images under $\phi$ and refer to these sets as subsets of $K$. Then let $g : Y \to \o^{\beta_0}+1 $ be a homeomorphism. For any $\gamma$ in $[0,\beta_0]$, there is an u.s.c.d. candidate sequence $\F$ on $\o^{\gamma}+1$ given by Corollary \ref{realizationCor} with $\al_0(\F) = \gamma$. Since $\o^{\gamma}+1 \subset \o^{\beta_0}+1$, we may extend $\F$ to a u.s.c.d. candidate sequence on $\o^{\beta_0}+1$ (still denoted $\F$) by letting $\F$ be uniformly $0$ off of $\o^{\gamma}+1$. Note that $\F$ on $\o^{\beta_0}+1$ still has the properties stated in Corollary \ref{realizationCor}. We now construct a harmonic, u.s.c.d. sequence $\H$ on $K$ such that $\al_0(\F) = \al_0(\H)$. Let $\F = (f_k)$ be given as above. Then let $\H' = (h'_k)$ be the candidate sequence on $K$ defined as follows. For $t$ in $K$, let
\begin{equation*}
h'_k(t) = \left \{ \begin{array}{ll}
                   0, & \text{ if } t \notin Y \cup W \\
                   f_k(g(t)), & \text{ if } t \in Y \setminus W \\
                   f_k(g(z_m)), & \text{ if } t = u_m \text{ or } t = v_m.
                  \end{array} \right.
\end{equation*}
We claim that for each $k$, $h'_{k+1}-h'_k$ is convex and u.s.c. Let $t$ be in $K$. If $t$ is in $X$, then $(h'_{k+1}-h'_k)(t) = \int_X (h'_{k+1}-h'_k) d\P_t$ since $\P_t = \epsilon_t$. If $t$ is in $K \setminus (Y \cup W)$, then $0 =(h'_{k+1}-h'_k)(t) \leq \int_X (h'_{k+1}-h'_k) d\P_t$. If $t$ is in $(Y \cup W) \setminus X = Y \setminus X = Z$, then $t=z_m$ for some $m$, and we have that $\P_{z_m} = \frac{1}{2}(\epsilon_{u_m}+\epsilon_{v_m})$. Then
\begin{align*}
(h'_{k+1}-h'_k)(z_m) = (f_{k+1}-f_k)(g(z_m)) & = \frac{1}{2}\Bigl( (h'_{k+1}-h'_k)(u_m) + (h'_{k+1}-h'_k)(v_m) \Bigr)  \\  & = \int_X h'_{k+1}-h'_k d\P_t.
\end{align*}
We have shown that $h'_{k+1}-h'_k$ is convex.

Let us prove that $h'_{k+1}-h'_{k}$ is u.s.c. Since $\{u_m\}$, $\{v_m\}$ and $\{z_m\}$ each have the same limit points, which are in $Y$ (since $\{z_m\}$ is in $Y$ and $Y$ is closed), we obtain that $Y\cup W$ is compact in $K$. Thus if $t$ is in $K \setminus (Y \cup W)$, then $\widetilde{(h'_{k+1}-h'_k)}(t)=0 = (h'_{k+1}-h'_k)(t)$. For $t$ in $Y \setminus W$, assume $\{t_n\}$ is a sequence in $K \setminus \{t\}$ converging to $t$ in $K$. Since $(h'_{k+1}-h'_k)|_{K \setminus (Y \cup W)} \equiv 0$, we may assume that $t_n$ lies in $Y \cup W$ for all $n$. For each $n$, if $t_n$ is not in $Y$, then there exists a natural number $m_n$ such that $t_n \in \{u_{m_n}, v_{m_n}\}$. If $t_n$ is in $Y$, then $(h'_{k+1}-h'_k)(t_n) = (f_{k+1}-f_k)(g(t_n))$, and if $t_n$ is not $Y$, then there exists a natural number $m_n$ such that $t_n \in \{u_{m_n},v_{m_n}\}$ and $(h'_{k+1}-h'_k)(t_n) = (f_{k+1}- f_k)(g(z_{m_n}))$. By the choice of $\{u_m\}$ and $\{v_m\}$, we have that$\{z_{m_n}\}$ also converges to $t$. Then since $\F$ is u.s.c.d. and $g$ is continuous, we have that $\limsup_n (h'_{k+1}-h'_k)(t_n) \leq (f_{k+1}-f_{k})(g(t)) = (h'_{k+1}-h'_k)(t)$. Thus $\widetilde{(h'_{k+1}-h'_k)}(t) = (h'_{k+1}-h'_k)(t)$. For $t$ in $W$, $t$ is isolated in $Y \cup W$, and we conclude that $\widetilde{(h'_{k+1}-h'_k)}(t) = (h'_{k+1}-h'_k)(t)$. Thus $(h'_{k+1}-h'_k)$ is u.s.c.

Now for $t$ in $K$, let $\H = (h_k)$, where $h_k$ is the harmonic extension of $h'_k$ on $K$. $\H$ is harmonic by definition. Fact \ref{fonEharIsUSC} states that the harmonic extension of a non-negative, convex, u.s.c. function on a Choquet simplex $K$ is a harmonic, u.s.c. function on $K$. Applying this fact to each element in the sequence $(h'_{k+1}-h'_k)$, we obtain that $\H$ is a harmonic, u.s.c.d. candidate sequence.

Let $F = (Y \cap X) \cup W$ and $L = \overline{F} = Y \cup W$. Note that $\H|_{X \setminus F} \equiv \H'|_{X \setminus F} \equiv 0$, which implies that we may apply the Embedding Lemma (Lemma \ref{embeddingThm}). The Embedding Lemma gives that for all ordinals $\al$ and all $t$ in $K$,
\begin{equation*}
u^{\H}_{\al}(t) = \max_{\mu \in \pi^{-1}(t)} \int_{L} u_{\al}^{\H|_L} \; d\mu.
\end{equation*}

Let us now show that for all $t$ in $K$,
\begin{equation} \label{computeH}
u^{\H}_{\al}(t) =  \max_{\mu \in \pi^{-1}(t)} \int_{Y} u_{\al}^{\H|_L} \; d\mu  =  \max_{\mu \in \pi^{-1}(t)} \int_{Y} u_{\al}^{\H|_Y} \; d\mu =  \max_{\mu \in \pi^{-1}(t)} \int_{Y} u_{\al}^{\F} \circ g \; d\mu.
\end{equation}
The first equality in Equation (\ref{computeH}) has already been justified as an application of the Embedding Lemma. The second equality in (\ref{computeH}) will be justified by showing that for all ordinals $\al$, $u_{\al}^{\H|_L}|_{L \setminus Y} \equiv 0$ and $u_{\al}^{\H|_L}|_Y = u_{\al}^{\H|_Y}$. Recall that $\H|_Y = \F' \circ g$, where $\F' = (f'_k)$ is the candidate sequence on $\o^{\beta_0}+1$ defined in terms of $\F = (f_k)$ as follows. If $t$ is in $(\o^{\beta_0}+1) \setminus g(W \cap Y)$, then $f'_k(t) = f_k(t)$, and if $t$ is in $g(W \cap Y)$, then $f'_k(t) = f_k(z_m)$ for $t = g(u_m)$ or $t = g(v_m)$. Since $g$ is a homeomorphism, we have that $u_{\al}^{\H|_Y} = u_{\al}^{\F'} \circ g$ for all ordinals $\al$. Then we will justify the third equality in Equation (\ref{computeH}) by proving that $u_{\al}^{\F} = u_{\al}^{\F'}$ for all ordinals $\al$. We proceed with these steps and then conclude the proof of the lemma using Equation (\ref{computeH}).

Notice that for all $t$ in $W$, $r_L(t) = 0$ ($t$ is isolated in $L$ by hypothesis). Thus, if $t \in W$, then $u_{\al}^{\H|_L}(t) = 0$ for all $\al$.

For $t$ in $Y$, suppose there is a sequence $s_n \in W$ such that $s_n$ converges to $t$ and $\limsup_{s \to t} \tau_k^{\H|_L}(s) = \lim_n \tau_k^{\H|_L}(s_n)$. Since $s_n$ is in $W$, for each $n$ there exists $m_n$ such that $s_n \in \{u_{m_n}, v_{m_n}\}$. Then $\tau_k^{\H|_L}(s_n) = \tau_k^{\H|_L}(z_{m_n})$, $z_{m_n}$ also converges to $t$, and since $z_{m_n}$ is in $Y$, $\tau_k^{\H|_Y}(z_{m_n}) = \tau_k^{\H|_L}(z_{m_n})$. Thus $\limsup_{s \to t} \tau_k^{\H|_L}(s) = \lim_n \tau_k^{\H|_Y}(z_{m_n})$. By these considerations, we have that for all $t$ in $Y$, $\widetilde{\tau_k^{\H|_L}}(t) = \widetilde{\tau_k^{\H|_Y}}(t)$. Letting $k$ tend to infinity gives that $u_1^{\H|_L}(t) = u_1^{\H|_Y}(t)$, for all $t$ in $Y$.

Now we show by transfinite induction that $u_{\al}^{\H|_Y} = u_{\al}^{\H|_L}|_Y$ for all ordinals $\al$. The equality holds for $\al = 1$ by the previous paragraph. Suppose by induction that it holds for some ordinal $\al$. For the sake of notation, we allow $s = t$ in the following limit suprema. Also, the limit supremum over an empty set is assumed to be $0$ by convention. For $t$ in $Y$, the induction hypothesis implies that
\begin{align*}
\widetilde{(u_{\al}^{\H|_L}+\tau_k)}(t) = & \max \Bigl( \limsup_{\substack{ s \to t \\ s \in W}} u_{\al}^{\H|_L}(s)+\tau_k(s), \, \limsup_{\substack{ s \to t \\ s \in Y}} u_{\al}^{\H|_L}(s)+\tau_k(s) \Bigr) \\
 = & \max \Bigl( \limsup_{\substack{ s \to t \\ s \in W}} \tau_k(s), \, \limsup_{\substack{ s \to t \\ s \in Y}} u_{\al}^{\H|_Y}(s)+\tau_k(s) \Bigr)
\end{align*}
Taking the limit as $k$ tends to infinity gives that
\begin{equation*}
u_{\al+1}^{\H|_L}(t) = \max \Bigl( u_1^{\H|_L}(t), \, u_{\al+1}^{\H|_Y}(t)\Bigr) = \max \Bigl( u_1^{\H|_Y}(t), \, u_{\al+1}^{\H|_Y}(t)\Bigr) = u_{\al+1}^{\H|_Y}(t).
\end{equation*}
Thus we conclude that $u_{\al+1}^{\H|_Y} = u_{\al+1}^{\H|_L}|_Y$, proving the successor case of our induction. For the limit case, suppose the equality holds for all ordinals $\beta$ less than a limit ordinal $\al$. Then for $t$ in $Y$, we have
\begin{align*}
u_{\al}^{\H|_L}(t) = & \max \Bigl( \limsup_{\substack{ s \to t \\ s \in W}} \sup_{\beta < \al} u_{\beta}^{\H|_L}(s), \, \limsup_{\substack{ s \to t \\ s \in Y}} \sup_{\beta < \al} u_{\beta}^{\H|_L}(s) \Bigr) \\
= & \max \Bigl( 0, \limsup_{\substack{ s \to t \\ s \in Y}} \sup_{\beta < \al} u_{\beta}^{\H|_Y}(s) \Bigr) \\
= & u_{\al}^{\H|_Y}(t),
\end{align*}
which concludes the limit step of the transfinite induction.

Now we turn our attention towards showing that $u_{\al}^{\F'} = u_{\al}^{\F}$ for all ordinals $\al$. By Remark \ref{LimsupToLimitRmk}, we assume (without loss of generality) that $\F$ has the property (P) that for $t$ in $\o^{\beta_0}+1$,
\begin{equation} \label{LimsupToLimitEqn}
\limsup_{\substack{s \to t \\ r(s) = 0}} \tau_k^{\F}(s) = \lim_{\substack{s \to t \\ r(s) = 0}} \tau_k^{\F}(s).
\end{equation}
We also require the following topological fact. For every point $t$ in $Y \setminus I$, there is a sequence in $I \setminus W$ that tends to $t$, where $I$ is the set of isolated points in $Y$. To prove this fact, let $t$ be a point with $r(t) \geq 1$ and let $U$ be an open  (in $Y$) neighborhood of $t$. Suppose for the sake of contradiction that $(I \setminus W) \cap U = \emptyset$. Since $Y \cong \o^{\beta_0}+1$ (a countable, compact Polish space), we have that $I$ is dense in $Y$ and $\Gamma^1(Y) \setminus \Gamma^2(Y)$ is dense in $\Gamma^1(Y)$. Since $\Gamma^1(Y) \setminus \Gamma^2(Y)$ is dense in $\Gamma^1(Y)$, there is a point $t'$ in $U$ with $r(t')=1$. Since $I$ is dense in $Y$, there is a sequence $w_n$ in $I \cap U$ tending to $t'$. Since $(I \setminus W) \cap U = \emptyset$, we must have that $w_n$ is in $W$ and then there is a sequence $m_n$ such that $w_n \in \{u_{m_n},v_{m_n}\}$ for all $n$. Then $z_{m_n}$ tends to $t'$. Note that $z_{m_n}$ is not in $W$ by hypothesis, and since $r(t')=1$, we must have that $z_{m_n}$ is isolated in $Y$ for all large $n$. Thus $(I \setminus W) \cap U \neq \emptyset$, a contradiction.

Using that $\F$ satisfies property (P) and the topological fact from the previous paragraph, let us show that for any non-isolated point $t$ in $\o^{\beta_0}+1$, we have $\widetilde{\tau_k^{\F'}}(t) = \widetilde{\tau_k^{\F}}(t)$. First note that for every sequence $s_n$ converging to $t$, there is a sequence $t_n$ converging to $t$ such that $\tau_k^{\F'}(s_n) = \tau_k^{\F}(t_n)$: if $s_n$ is not in $g(W \cap Y)$, then let $t_n = s_n$, and if $s_n$ is in $g(W \cap Y)$, then there exists $m_n$ such that $s_n \in \{g(u_{m_n}),g(v_{m_n})\}$, and one may take $t_n = g(z_{m_n})$. It follows that $\limsup_{s \to t} \tau_k^{\F'}(s) \leq \limsup_{s \to t} \tau_k^{\F}(s)$. Also, since $t$ is not isolated, $t$ is not in $g(W \cap Y)$ and $\tau_k^{\F'}(t) = \tau_k^{\F}(t)$. We deduce that $\widetilde{\tau_k^{\F'}}(t) \leq \widetilde{\tau_k^{\F}}(t)$. Now we show the reverse inequality. If $s_n$ is a sequence converging to $t$ with $r(s_n) >0$, then $s_n$ is not in $g(W \cap Y)$ and thus $\tau_k^{\F'}(s_n) = \tau_k^{\F}(s_n)$. In such a case, we have $\limsup_n \tau_k^{\F'}(s_n) = \limsup_n \tau_k^{\F}(s_n)$. Now let $s_n$ be a sequence converging to $t$ with $r(s_n)=0$. By the topological fact from the previous paragraph, there is a sequence $t_n$ of isolated points in $\o^{\beta_0}+1$ that are not in $g(W \cap Y)$ such that $t_n$ converges to $t$. Using the fact that $\F$ satisfies property (P) (see Equation (\ref{LimsupToLimitEqn})), we have $\limsup_n \tau_k^{\F}(s_n) = \limsup_n \tau_k^{\F}(t_n)$. Since the points $t_n$ are not in $g(W \cap Y)$ we also have that $\limsup_n \tau_k^{\F}(t_n) = \limsup_n \tau_k^{\F'}(t_n) \leq \limsup_{s \to t} \tau_k^{\F'}(s)$. We have shown that for every sequence $s_n$ converging to $t$, $\limsup_n \tau_k^{\F}(s_n) \leq \limsup_{s \to t} \tau_k^{\F'}(s)$. It follows that $\widetilde{\tau_k^{\F'}}(t) \geq \widetilde{\tau_k^{\F}}(t)$, and therefore we have shown that $\widetilde{\tau_k^{\F'}}(t) = \widetilde{\tau_k^{\F}}(t)$.

Finally, we show that for all ordinals $\al$, $u_{\al}^{\F'} = u_{\al}^{\F}$ by transfinite induction on $\al$. We make the conventions that we allow $s = t$ in the following limit suprema, and the limit supremum over an empty set is $0$. Note that if $t$ is isolated in $\o^{\beta_0}+1$, then $u_{\al}^{\F}(t) = 0 = u_{\al}^{\F'}(t)$ for all $\al$, and thus we need only show the equality at non-isolated points $t$ in $\o^{\beta_0}+1$. For the sake of induction, suppose the equality holds for an ordinal $\al$. Let $t$ be in $(\o^{\beta_0}+1) \setminus g(I)$. For every sequence $s_n$ converging to $t$, there is a sequence $t_n$ converging to $t$ such that $(u_{\al}^{\F'}+\tau_k^{\F'})(s_n) = (u_{\al}^{\F}+\tau_k^{\F})(t_n)$: if $s_n$ is not in $g(W \cap I)$, then let $t_n = s_n$, and if $s_n$ is in $g(W \cap I)$, then there exists $m_n$ such that $s_n \in \{g(u_{m_n}),g(v_{m_n})\}$, and one may take $t_n = z_{m_n}$. It follows that $\limsup_{s \to t} (u_{\al}^{\F'} + \tau_k^{\F'})(s) \leq \limsup_{s \to t} (u_{\al}^{\F} + \tau_k^{\F})(s)$. Thus we have that $\widetilde{(u_{\al}^{\F'} + \tau_k^{\F'})}(t) \leq \widetilde{(u_{\al}^{\F} + \tau_k^{\F})}(t)$. Now we show the reverse inequality. Let $s_n$ be a sequence in $g(I)$ converging to $t$. Then $(u_{\al}^{\F} + \tau_k^{\F})(s_n) = \tau_k^{\F}(s_n)$ and so $\limsup_n (u_{\al}^{\F} + \tau_k^{\F})(s_n) = \limsup_n \tau_k^{\F}(s_n) \leq \widetilde{\tau_k^{\F}}(t) = \widetilde{\tau_k^{\F'}}(t)$ (recall that we showed the last equality in the previous paragraph). Now let $s_n$ be a sequence in $(\o^{\beta_0}+1) \setminus g(I)$ converging to $t$. Since $s_n$ is not isolated, $s_n$ is not in $g(W \cap Y)$, and we have $(u_{\al}^{\F} + \tau_k^{\F})(s_n) = (u_{\al}^{\F'} + \tau_k^{\F'})(s_n)$. Also, $\limsup_n (u_{\al}^{\F'} + \tau_k^{\F'})(s_n) \leq \widetilde{(u_{\al}^{\F'} + \tau_k^{\F'})}(t)$. Combining these considerations, we have shown that
\begin{equation*}
 \widetilde{(u_{\al}^{\F} + \tau_k^{\F})}(t) \leq \max\Bigl( \widetilde{\tau_k^{\F'}}(t), \widetilde{(u_{\al}^{\F'} + \tau_k^{\F'})}(t) \Bigr) = \widetilde{(u_{\al}^{\F'} + \tau_k^{\F'})}(t).
\end{equation*}
Then we deduce that $\widetilde{(u_{\al}^{\F} + \tau_k^{\F})} = \widetilde{(u_{\al}^{\F'} + \tau_k^{\F'})}$. Taking the limit in $k$ gives that $u_{\al+1}^{\F} = u_{\al+1}^{\F'}$, which concludes the successor step of the transfinite induction. For the limit step, assume that $u_{\beta}^{\F} = u_{\beta}^{\F'}$ for all ordinals $\beta$ less than a limit ordinal $\al$. We show that $u_{\al}^{\F} = u_{\al}^{\F'}$. For $t$ in $\o^{\beta_0}+1$, the induction hypothesis gives that (allowing $s = t$ in the the limit suprema)
\begin{equation*}
 u_{\al}^{\F}(t) = \limsup_{s \to t} \sup_{\beta < \al} u_{\beta}^{\F}(s) = \limsup_{s \to t} \sup_{\beta < \al} u_{\beta}^{\F'}(s) = u_{\al}^{\F'}(t).
\end{equation*}
We conclude that $u_{\al}^{\F} = u_{\al}^{\F'}$ for all ordinals $\al$. This fact completes the verification of Equation (\ref{computeH}).

It follows immediately from Equation (\ref{computeH}) that $\al_0(\H) \leq \al_0(\F) = \gamma$. We now show the reverse inequality. Let $\0_{\gamma}$ be the marked point in Corollary \ref{realizationCor}, and let $t=g^{-1}(\0_{\gamma})$. Then $u_{\gamma}^{\H}(t) \geq u_{\gamma}^{\F}(\0_{\gamma}) = a$. For an arbitrary $\al < \gamma$, we also have that $u_{\al}^{\H}(t) \leq ||u_{\al}^{\F}|| < a$ by Equation (\ref{computeH}) and Corollary \ref{realizationCor} (1). Thus $\gamma = \al_0(t) \leq \al_0(\H)$, and we conclude that $\al_0(\H) = \gamma$.

Since $\gamma \leq \beta_0$ was arbitrary, we deduce that $S(K) \supset [0,\beta_0]$. For $\beta$ finite, $\beta_0 = \beta$ and the proof is finished in this case. On the other hand, if $\beta$ is infinite, then $\beta_0 = \beta-1$ and we may repeat the above argument starting with $\F$ on $\o^{\beta_0}+1$ given by Corollary \ref{alphaPlusOne} such that $\al_0(\F)=\beta_0+1$. In this case, we conclude that $S(K) \supset [0,\beta_0+1] = [0, \beta]$, which concludes the proof.
\end{proof}

\subsection{Open Questions} \label{OpenQuestions}

In general, our analysis leaves open the following problem.
\begin{ques}
For a metrizable Choquet simplex $K$, what is $S(K)$?
\end{ques}

Theorem \ref{topBounds} completely answers this question when $\rho_{\ex(K)}(\overline{\ex(K)}) = \rho(\overline{\ex(K)})$. In particular, when $K$ is Bauer or when $\ex(K)$ is uncountable, Theorem \ref{topBounds} gives a complete answer. In general, Theorem \ref{topBounds} gives upper and lower bounds on $S(K)$.

Theorem \ref{threeAlphas} shows that the bounds in Theorem \ref{topBounds} cannot be improved using only knowledge of the ordinals $\rho_{\ex(K)}(\overline{\ex(K)})$ and $\rho(\overline{\ex(K)}))$. Theorem \ref{optimal} (1) shows that if $\overline{\ex(K)}$ is countable, then the bounds in Theorem \ref{topBounds} cannot be improved using only knowledge of the homeomorphism class of the compactification $(\ex(K),\overline{\ex(K)})$. Theorem \ref{optimal} (2) shows that the upper bound in Theorem \ref{topBounds} cannot be improved using only knowledge of the homeomorphism class the compactification $(\ex(K),\overline{\ex(K)})$. Thus we have the following question remaining.
\begin{ques} \label{Qone}
Let $E$ be a countable, non-compact Polish space, and let $\overline{E}$ be an uncountable metrizable compactification of $E$. Let $\beta$ be a successor in $[\rho_E(\overline{E}), \o_1[$. Must there exist a metrizable Choquet simplex $K$ such that $(E,\overline{E}) \simeq (\ex(K),\overline{\ex(K)})$ and $S(K) = [0,\beta]$?
\end{ques}
Also, when $E$ is countable and $\overline{E}$ is uncountable, we do not know whether the upper bound on $S(K)$ may be attained. We state this problem as a question as follows.
\begin{ques} \label{Qtwo}
Let $E$ be a countable, non-compact Polish space, and let $\overline{E}$ be an uncountable metrizable compactification of $E$. Must there exist a metrizable Choquet simplex $K$ such that $(E,\overline{E}) \simeq (\ex(K),\overline{\ex(K)})$ and $S(K) = [0,\o_1 [$?
\end{ques}

If the answers to Questions \ref{Qone} and \ref{Qtwo} are affirmative, then one could conclude that the bounds in \ref{topBounds} cannot be improved using knowledge of the homeomorphism class of the compactification $(\ex(K),\overline{\ex(K)})$, and furthermore, one could conclude that these bounds are obtained.

Notice that for every simplex $K$ for which we can compute $S(K)$, $S(K)$ is either $[0,\o_1 [$ or $[0,\beta]$ for a countable successor $\beta$. This observation leads to the following two questions.
\begin{ques} \label{Qthree}
If $K$ is a metrizable Choquet simplex, must $S(K)$ be an ordinal interval?
\end{ques}
\begin{ques} \label{Qfour}
If $K$ is a metrizable Choquet simplex, must $S(K)$ be either $[0,\o_1 [$ or $[0,\beta]$ for a countable successor $\beta$?
\end{ques}

If the answers to Questions \ref{Qone}, \ref{Qtwo}, \ref{Qthree}, and \ref{Qfour} are all affirmative, then these results would give a complete description of the constraints imposed on orders of accumulation by the compactification of the ergodic measures for a dynamical system.

\vspace{1mm}
\textbf{Acknowledgment:} The authors would like to express special thanks to Mike Boyle, who contributed many helpful ideas, conversations, and hours to this work.

\appendix
\section{Entropy Structures, Symbolic Extensions, and Dynamical Systems} \label{Rel2DynSys}

For general references on the ergodic theory of topological dynamical systems, see \cite{Glas,Pet,W}. For a topological dynamical system $(X,T)$, we write $M(X,T)$ to denote the space of Borel probability measures on $X$ which are invariant under $T$. We give $M(X,T)$ the weak* topology. It is well known that in this setting $M(X,T)$ is a metrizable, compact, convex subset of a locally convex topological vector space (see, for example, \cite{Glas, Pet}). The set of extreme points of $M(X,T)$ is the set of ergodic measures, $M_{\erg}(X,T)$. Furthermore, the fact that each measure $\mu$ in $M(X,T)$ has a unique ergodic decomposition (see \cite{Glas,Pet}) translates to the fact that $M(X,T)$ is a Choquet simplex. Since we are only interested in simplices arising from dynamical systems, we consider only metrizable Choquet simplices. It was shown in \cite{D3} that every metrizable Choquet simplex $K$ can be obtained as the space of invariant Borel probability measures for a dynamical system.

We write $h : M(X,T) \rightarrow [0, \medspace \infty)$ to denote the function that assigns to each measure $\mu$ in $M(X,T)$ its metric entropy. For any dynamical system $(X,T)$, Boyle and Downarowicz defined a reference candidate sequence $\H_{ref}(X,T)$ on $M(X,T)$ that is u.s.c.d. and harmonic. Further, Downarowicz defined an \textbf{entropy structure} on $M(X,T)$ to be any candidate sequence on $M(X,T)$ that is uniformly equivalent to $\H_{ref}$ (see Section \ref{candSeqs} for definitions). Almost all known methods of defining or computing entropy can be adapted to form an entropy structure \cite{D}. The work of Downarowicz and Serafin \cite{DS} implies the following realization theorem:
\begin{thm}[\cite{D,DS}] \label{realization}
Let $\H$ be a candidate sequence on a Choquet simplex $K$ that is uniformly equivalent to a harmonic candidate sequence with u.s.c. differences. Then $\H$ is (up to affine homeomorphism) an entropy structure for a minimal homeomorphism of the Cantor set.
\end{thm}
The importance of Theorem \ref{realization} lies in the fact that it allows one to translate questions in the theory of entropy structures and dynamical systems into the terms of functional analysis. To understand the theory of entropy structure in dynamical systems, it helps to consider symbolic extensions, and we briefly recall this theory.


A good introduction to symbolic dynamical systems is given in \cite{LM}. For any finite set $\cA$, we refer to $\cA^{\Z}$ as the full shift on $\cA$. The elements of $\cA$ are referred to as symbols. We give $\cA^{\Z}$ the product topology induced by the discrete topology on $\cA$, which makes $\cA^{\Z}$ a compact metrizable space. Then the left-shift map, $\sigma : \cA^{\Z} \rightarrow \cA^{\Z}$, given by $\sigma(x)_n = x_{n+1}$, is a homeomorphism of $\cA^{\Z}$. If $Y$ is closed subset of $\cA^{\Z}$ satisfying $\sigma(Y) = Y$ and $S = \sigma|_Y$, then we refer to $(Y,S)$ as a symbolic dynamical system, or possibly a subshift of $\cA^{\Z}$.

\begin{defn}
Let $(X,T)$ be a dynamical system. A \textbf{symbolic extension} of $(X,T)$ is a subshift $(Y,S)$ of a full shift on a finite number of symbols, along with a continuous surjection $\pi: Y \rightarrow X$ such that $T \pi = \pi S$.
\end{defn}
We think of a symbolic extension as a ``lossless finite encoding'' of the dynamical system $(X,T)$ \cite{D}.

Downarowicz introduced the study of the entropy of symbolic extensions at the level of measures \cite{D2}.

\begin{defn}
If $(Y,S)$ is a symbolic extension of $(X,T)$ with factor map $\pi$, then the \textbf{extension entropy function}, $h_{ext}^{\pi} : M(X,T) \rightarrow [0,\infty)$, is defined for $\mu$ in $M(X,T)$ by
\begin{equation*}
h_{ext}^{\pi}(\mu) = \sup \{ h(\nu) : \pi^* \mu = \nu \}.
\end{equation*}
The \textbf{symbolic extension entropy function} of a dynamical system $(X,T)$, $h_{sex} : M(X,T) \rightarrow [0, \medspace \infty]$, is defined for $\mu$ in $M(X,T)$, as
\begin{equation*}
h_{sex}(\mu) = \inf \{ h_{ext}^{\pi}(\mu) : \pi \text{ is the factor map of a symbolic extension of } (X,T) \}.
\end{equation*}
and the \textbf{residual entropy function}, $h_{res}: M(X,T) \rightarrow [0,\infty]$ is defined for $\mu$ in $M(X,T)$ as
\begin{equation*}
h_{res}(\mu) = h_{sex}(\mu)-h(\mu).
\end{equation*}
\end{defn}

The study of symbolic extensions is related to entropy structures by the following striking result.
\begin{thm}[\cite{BD}] \label{sexEntThm}
Let $(X,T)$ be a dynamical system with entropy structure $\H$. Then
\begin{equation*}
h_{sex} = h + u_{\al_0(\H)}^{\H},
\end{equation*}
\end{thm}
This theorem relates the notion of how entropy emerges on refining scales to the symbolic extensions of a system, showing that there is a deep connection between these topics. Using this connection, some progress has been made in understanding the symbolic extensions of certain classes of dynamical systems, with particular interest in smooth dynamical systems. For results of this type, see \cite{A, BD, Bur2, BurNew, DF, D, DM, DN}. Note that the functions $u_{\al}$ are, in general, not harmonic, which stands in stark contrast to most objects of study in ergodic theory (in particular, the entropy function $h$ is harmonic \cite{Glas, Pet}).

The order of accumulation $\al_0(X,T)$, which is defined as $\al_0(\H)$ for any entropy structure $\H$ of the system $(X,T)$, measures on how many ``layers'' residual entropy accumulates in system. From Theorem \ref{sexEntThm} we see that the complexity in these layers accounts for the extra entropy that must be added to each measure in the system in order to obtain symbolic extensions. Thus the order of accumulation of entropy measures some additional complexity in the system that is not detected by the entropy function $h$.

\section{Proof of Fact \ref{fonEharIsUSC}} \label{proofAppendix}
The following fact was given as Fact 2.5 in \cite{DM}, where there is a sketch of the proof. In this appendix we fill in some details of this proof for the sake of completeness.
\begin{fact*}[Fact \ref{fonEharIsUSC}]
Let $K$ be a metrizable Choquet simplex, and let $f: K \rightarrow [0, \infty)$ be convex and u.s.c. Then $(f|_{\ex(K)})^{har}$ is u.s.c.
\end{fact*}
\begin{proof}
Let $f: K \to [0,\infty)$ be convex and u.s.c. Let $g : \M(K) \to [0,\infty)$ be defined for each $\mu$ in $\M(K)$ as
\begin{equation*}
g(\mu) = \int f d\mu.
\end{equation*}
Now let $G: K \to [0,\infty)$ be given by $G(x) = \sup\{g(\mu) : \bary(\mu)=x\}$ for all $x$ in $K$. We have that $g$ is u.s.c. because $f$ is u.s.c., and $G$ is u.s.c. because $g$ is u.s.c. (Remark \ref{pushDownPullUp} (iii)).

Now we claim that $f(x) \leq \int f d \mu$ for any $\mu$ such that $\bary(\mu) = x$. To see this, fix $x$ and $\mu$ such that $\bary(\mu)=x$. Let $f_m$ be a decreasing sequence of continuous functions, $f_m : K \to [0,\infty)$, whose limit is $f$. Let $\delta >0$. Partition the support of $\mu$ into a finite number of sets $S_j$ of diameter smaller than $\delta$. For each $j$, if $\mu(S_j)>0$, let $z_j = \bary( \mu_{S_j} )$, where $\mu_{S_j}$ is the measure $\mu$ conditioned on the set $S_j$. Then let $\nu = \sum_j \mu(S_j) \epsilon_{z_j}$. Note that $\bary(\nu) = \bary(\mu) = x$, and $\nu$ tends to $\mu$ in $\M(K)$ as $\delta$ tends to zero. We have shown that there exists a sequence of measures $\nu_k$ such that each $\nu_k$ is a finite convex combination of point measures, $\nu_k$ converges to $\mu$ in $\M(K)$, and $\bary(\nu_k)=x$ for each $k$. Now choose such a sequence $\nu_k$, and note that for any $m$, any $\epsilon >0$, and any large enough $k$ (depending on $\epsilon$ and $m$), by the convexity of $f$,
\begin{equation*}
f(x) \leq  \int f \; d\nu_k \leq  \int f_m \;  d \nu_k \leq \int f_m \; d\mu  + \epsilon.
\end{equation*}
Letting $m$ tend to infinity, the Dominated Convergence Theorem implies that $f(x) \leq \int f d\mu +\epsilon$. Since $\epsilon$ was arbitrary, we see that $f(x) \leq \int f d\mu$, which implies in particular that $f(x) \leq \int f d\P_x$.

Then for any $\mu$ with $\bary(\mu) = x$,
\begin{equation*}
\int f d\mu \leq \int \Bigl( \int f d\P_y \Bigr) \; d \mu(y) = \int f \; d\P_x,
\end{equation*}
where the equality of the last two expressions follows from the fact that $x \mapsto \int f \; d\P_x$ defines a harmonic function on $K$ (Remark \ref{harmRemark}).

Thus $G(x) = \int f d\P_x$, which shows that $G = (f|_{\ex(K)})^{har}$. Since $G$ is u.s.c., the proof is complete.
\end{proof}


\end{document}